\documentclass[12pt]{article}

\usepackage{amsfonts, amsmath,amssymb,array,latexsym, hyperref, amscd}

\usepackage[authoryear]{natbib}
\usepackage{lipsum}
\usepackage{fancyhdr,a4wide}
\usepackage{appendix}
\usepackage{amsthm}
\usepackage{mathtools} 
\usepackage{subcaption}
\usepackage{xcolor}
\usepackage{amsmath}
\usepackage{float}
\usepackage{enumitem}
\usepackage{authblk}
\RequirePackage{amsthm,amsmath,amsfonts,amssymb}
\RequirePackage[authoryear]{natbib}
\RequirePackage{graphicx}
\usepackage{subcaption}
\usepackage{float}
\usepackage{pgf}

\newtheorem{theorem}{Theorem}[section]
\newtheorem{lemma}[theorem]{Lemma}
\newtheorem{corollary}[theorem]{Corollary}
\newtheorem{proposition}[theorem]{Proposition}

\newtheorem{definition}[theorem]{Definition}

\newtheorem{assumption}[theorem]{Assumption}
\newtheorem{example}[theorem]{Example}

\numberwithin{equation}{section}

\newcommand{\eps}{\varepsilon}
\newcommand{\norm}[1]{\left|\left|#1\right|\right|}

\newcommand{\E}[1]{\mathbb E\left[#1\right]}

\def\Var{\mathrm{Var}}  
\def\Cov{\mathrm{Cov}} 

\newcommand{\R}{\mathbb R}

\newcommand{\N}{\mathbb N}

\newcommand{\tr}{\text{tr}}
\renewcommand{\P}{\mathbb P}
\usepackage{xr}

\usepackage{graphicx} 

\title{Shrinkage to Infinity: Reducing Test Error by Inflating the Minimum Norm Interpolator in Linear Models}
\author{Jake Freeman\thanks{Department of Operations Research and Financial Engineering, Princeton University. Email: \texttt{jake.freeman@princeton.edu}}}
\begin{document}

\date{}
\maketitle{}

\begin{abstract}
\cite{10.1214-21-AOS2133} found that ridge regularization is essential in high dimensional linear regression $y=\beta^Tx + \epsilon$ with isotropic co-variates $x\in \R^d$ and $n$ samples at fixed $d/n$. However, \cite{10.1214-21-AOS2133} also notes that when the co-variates are anisotropic and $\beta$ is aligned with the top eigenvalues of population covariance, the “situation is qualitatively different.” In the present article, we make precise this observation for linear regression with highly anisotropic covariances and diverging $d/n$. We find (both theoretically and empirically) that simply \textit{scaling up} (or inflating) the minimum $\ell_2$ norm interpolator by a constant greater than one can improve the generalization error. This is in sharp contrast to traditional regularization/shrinkage prescriptions. Moreover, we use a data-splitting technique to produce consistent estimators that achieve generalization error comparable to that of the optimally inflated minimum-norm interpolator. Our proof relies on matching upper and lower bounds for expectations of Gaussian random projections for a general class of anisotropic covariance matrices when $d/n\rightarrow \infty$.
\end{abstract}
\section{Introduction}
In this paper, we study over-parameterized least square regression with a particular emphasis on data-generating processes that have a strongly anisotropic covariance matrix. Specifically, we assume we are given:
\begin{align*}
    \{(x_i,y_i)\}^n_{i=1} &\text{ where } x_i\sim N(0,\Sigma), \Sigma\in\mathbb{R}^{d\times d}\\
    y_i=x_i^T\beta +\eps &\text{ where } \E{\eps\,|\,x_i}=0,\E{\eps^2_i\,|\,x_i}\leq \sigma^2_{max}
\end{align*}

The parameters $d,\Sigma,\sigma_{max}^2,\text{ and }\beta$ are all non-random functions of $n$. We assume $d>n+1$ and take $tr(\Sigma)/n\to\infty$ and assume $\Sigma$ satisfies the structural conditions in Assumption \ref{assumpt:weak_canonical_case}, including that $\Sigma$ has  an effective rank proportional to $n$, where we define the effective rank to be:
\begin{align*}
    \text{eff-rank}(\Sigma)=\frac{tr(\Sigma)}{\norm{\Sigma}_2}
\end{align*}
Examples of allowable $\Sigma$ include covariances with power-law eigenvalue decay and spectra consisting of two well-separated blocks.  We focus on analyzing generalization error for functionals related to the minimum $\ell_2$-norm interpolator, $\theta_{MN}$:

\begin{equation}\label{equ:intro_min_norm_def}
\begin{aligned}
    \theta_{MN}&:=\arg\min\left\{\norm{\theta}_2 : x_i^T\theta=y_i \;\forall i\in[n]\right\}
\end{aligned}    
\end{equation}

\cite{10.1214-21-AOS2133} studied $\theta_{MN}$ in the proportional regime ($d/n\to\gamma\in (0,\infty)$). They found that when $\Sigma$ is isotropic, optimally tuned ridge regression has notably better test error compared to $\theta_{MN}$, across all $\gamma$ and signal-to-noise ratios. The paper also noted that if $\Sigma$ was anisotropic and $\beta$ aligns well with the top eigenvectors of $\Sigma$, the ``the situation is qualitatively different'' and suggested that $d/n\to\infty$ is the optimal setting to view this setup.
Our core observation is that in this setup the estimator,
\begin{align}\label{equ:def_theta_hat}
 \hat{\theta}=c\theta_{MN}   
\end{align}
with a well-chosen $c>1$ results in an improvement of the generalization error over $\theta_{MN}$. That is,
\begin{align}\label{equ:inflation}
G(\hat{\theta})<G(\theta_{MN}) \text{ where } G(\theta(\cdot))=\E{(x^T\theta(X)-x^T\beta)^2}
\end{align}

We call this phenomenon the ``Inflation Property'' (see Theorems \ref{thm:add_improve}, \ref{theorem:canon_multiplicative}). This property is in tension with the traditional notion that shrinking $\theta_{MN}$ towards zero often leads to better generalization.

The Inflation Property can be heuristically understood by the following two core intuitions:

\begin{itemize}
    \item Recall that the Johnson--Lindenstrauss lemma (JL) states that a random Gaussian projection matrix (with isotropic variates) is approximately an isometry. The connection to JL is most apparent in the noiseless regime (the regime  where $\eps=0$ a.s.) because $\hat{\theta}_c=c\Pi_X\beta$ (where $\Pi_X:=X^T(XX^T)^{-1}X$). When $\Sigma$ is sufficiently anisotropic, choosing $c>1$ (in Equ. \ref{equ:def_theta_hat}) can be thought of as increasing lengths to account for the structure of the tail. Indeed, in the noiseless isotropic regime, the optimal $c$ from Equ. \ref{equ:def_theta_hat} to use is 1.
    \item In the noiseless case,
    \begin{align*}
        G(c\theta_{MN})=\beta^T\Sigma\beta-2c\E{\beta^T\Sigma\Pi_X\beta}+\beta^T\E{\Pi_X\Sigma\beta}\beta
    \end{align*}
    As a result, in order for $c>1$ (and hence the Inflation Property to hold), we require that:
    \begin{align*}
        \beta^T(\E{\Pi_X\Sigma}-\E{\Pi_X\Sigma\Pi_X})\beta>0
    \end{align*}
    This can re-written as requiring $\beta_{||}^T\Sigma\beta_{\perp}>0$ on average (where $\beta_{||}=\Pi_X\beta$ and $\beta_{\perp}=\beta-\beta_{||}$). This implies that we need $\beta_{\perp}$ and $\beta_{||}$ to be aligned with some of the same eigenvectors of $\Sigma$ (on average). Thus, the Inflation Property comes from the span of the data not living in an $n$-dimensional hyperplane of the eigenvectors of $\Sigma$.
\end{itemize}

The Inflation Property (Equ. \ref{equ:inflation}) is somewhat analogous to classical results in James--Stein shrinkage \citep{james_estimation_1961}. In the James--Stein setup, shrinking towards an arbitration point improves the generalization error. In our setting, shrinking towards the wrong point will result in worse generalization error (see Theorem \ref{thm:add_improve}). Shrinking $\theta_{MN}$ towards (or away from) any direction other than zero results in some data generation processes having diverging generalization error (see Prop. \ref{prop:optimal_direction_shrink}).

The Inflation Property (Equ. \ref{equ:inflation}) relates to implicit regularization---i.e. that ``small'' variance directions of $\Sigma$ are able to absorb label noise and add an ``effective'' $\ell_2$ regularization penalty. \cite{10.1214-21-AOS2133} showed in the regime when $d$ and $n$ are proportional, the predictions given by the $\ell_2$ minimum-norm interpolator (Equ. \ref{equ:intro_min_norm_def}) can be improved by adding positive regularization when the covariance structure is close to isotropic (in Prop. \ref{prop:lambda_reg_isotropic_high_dim}, we recover these results when $d/n\to \infty$ when the noise is non-trivial).

However, when the data is sufficiently anisotropic, our analysis reveals a different regime---where the $\ell_2$ minimum-norm interpolator is actually overly regularized and some sort of anti-shrinkage or anti-regularization can improve generalization error. A related result is \cite{JMLR:v21:19-844} which showed that a negative ridge penalty was optimal in spiked covariance model and found that the number of dimensions needed to grow super-linearly to the sample size for this effect to occur. In Prop. \ref{prop:spiked_cov_model_improv}, we relate to \cite{JMLR:v21:19-844} by showing that under certain assumptions, the spiked covariance model experiences the Inflation Property.
\begin{figure}[H]
  \centering 
  \subfloat{\includegraphics[width=0.8\linewidth]{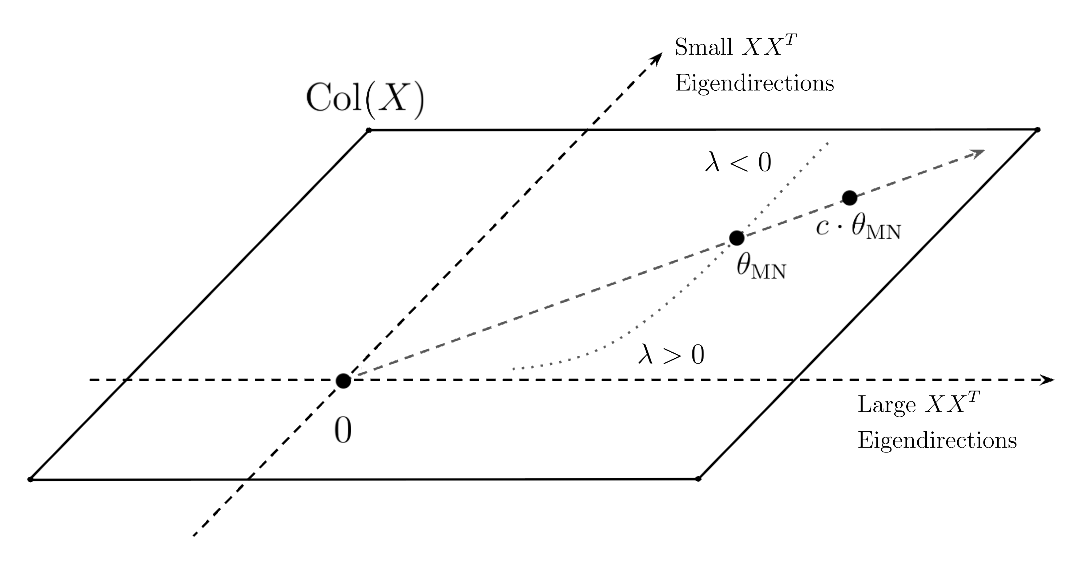}\label{fig:f1}}
  \captionsetup{justification=centering}
  \caption{This figure illustrates how the Inflation Property relates to $\theta_{MN}$, positive ridge regression, and negative ridge regression.}
  \label{fig:fig_1}
\end{figure}

Although the Inflation Property occurs in similar settings as when a negative ridge penalty is beneficial, the two are not equivalent (as seen in Fig. \ref{fig:fig_1}). While the Inflation Property increases all the eigenvalues equally, negative ridge regression increases the size of the eigenvalues in the smaller directions of $X$ disproportionately (compared to the eigenvalues of the larger directions of $X$). This can be seen by considering the singular values of the Tikhonov-regularized projection matrix with a ridge penalty $\lambda$ (in terms of the singular values of $X$),
\begin{align*}
    s_i^{ridge}=\frac{s_i^2}{s_i^2+\lambda}
\end{align*}

Further, while the Inflation Property reduces the bias, it is a data-dependent form of bias reduction that leverages the structure of the feature matrix $X$. In contrast, naively creating an unbiased estimator, $\hat{\theta} = \E{\Pi_X}^{-1} \theta_{MN}$, leads to diverging generalization error (see Prop. \ref{prop:make_unbiased}). 

We conclude with a brief summary of our main results: 

\begin{itemize}
    \item In some cases, the improvement created by the Inflation Property (Equ. \ref{equ:inflation}) is multiplicative, even in the presence of a vanishing signal-to-noise ratio ($\frac{\beta^T\Sigma\beta}{\E{\eps^2}}\to0$). See Theorem \ref{theorem:canon_multiplicative}, Prop. \ref{prop:snr_ex_prop}. That is, for a general class of covariance matrices there is an $\alpha\in(0,1)$, not depending on $\Sigma$, such that for some $c>1$ (depending on $\Sigma$),
    \begin{align*}
        G(c\theta_{MN})<\alpha G(\theta_{MN})
    \end{align*}
    \item To show that the multiplicative improvement holds in high probability and to obtain even stronger improvements from the Inflation Property, we use a data splitting technique (see Section \ref{sec:data_splitting}). Here, we can create an estimator, $\theta_{ds}$, which is the sum of $\ell_2$ minimum norm interpolators performed on each subset of the data. The advantage of the data splitting technique is that we are able to show the concentration of empirical risk around its mean. That is, for some estimator $\hat{c}>1$ when the Inflation Property holds,
    \begin{align*}
        \P\left(\norm{\hat{c}\theta_{ds}-\beta}_{\Sigma}\geq \alpha_{min} G(\theta_{MN})\right)\to0
    \end{align*}

    Additionally, under certain assumptions, data splitting allows us to achieve, for a well choosen $c>1$, $G(c\theta_{ds})=o(1)$, despite $G(\theta_{ds})=\Theta(1)$.
    \item On the technical side, we show a variety of results relating to the expectation of anisotropic Gaussian random projection matrices and concentration of functions of anisotropic Gaussian random projection matrices, including showing that $\E{\Pi_X}=\Theta(\frac{n}{tr(\Sigma)})\Sigma$ when $d/n\to\infty$. See Props. \ref{prop:proj_expect} and \ref{prop:var_Sig_proj}. 
    \item We also provide empirical results showing the existence of the Inflation Property on different real-world data sets and show using simulations that the Inflation Property can hold with non-Gaussian covariates. See Figures \ref{fig:fig_sim_1}, \ref{fig:fig_2}, \ref{fig:fig_3}, and \ref{fig:fig_4}.
\end{itemize}
\subsection{Related Literature}
The present work is motivated by recent papers showing that minimum-norm interpolators perform near optimally in the over-parameterized context, despite being overfit and with no additional regularization \citep{Belkin_2019,pmlr-v89-belkin19a}. In the context of linear models, \cite{bartlett2020benign} characterized when the $\ell_2$ minimum-norm interpolator results in an asymptotically vanishing risk (also called benign overfitting). Through these works, the authors found that the minimum-norm interpolator experiences implicit regularization. 

Similar results have shown the benign overfitting result in many other contexts---see e.g. ridgeless kernel regression \citep{44f6fc04-5e81-3369-9330-afe10d4a9993}, two-layer neural networks \citep{NEURIPS2022_a12c999b}, and multiclass classification \citep{NEURIPS2021_caaa29ea}. Of particular note is \cite{pmlr-v151-wang22k} where the authors analyzed the case of the $\ell_1$ minimum-norm interpolator in the case of a sparse $\beta$ and isotropic covariance. Here, benign overfitting does not occur unless $d\gg n$. We differ from the sparsity case considered in \cite{pmlr-v151-wang22k} by exploiting the joint structure of $\beta$ and $\Sigma$ rather than just $\beta$. Additionally, \cite{belkin2359–2376} showed that there is a strong relation between minimum-norm interpolation and the max-margin support-vector-machine when $d\gg n$.

Other papers have explored whether the $\ell_2$ minimum-norm interpolator undergo an anti-regularization procedure. In \cite{10.1214-21-AOS2133}, the authors focused on the proportional regime (when $n/d\to(0,\infty)$) and found that a positive ridge penalty is helpful when the data is isotropic or close to isotropic. Although this paper did not consider using a negative ridge penalty case, it found that a positive ridge penalty is not helpful when the data experiences sufficient anisotropy and $\beta$ is adequately aligned with the top eigenvectors of $\Sigma$.

In \cite{JMLR:v21:19-844}, the authors (1) empirically found that for some data sets the optimal ridge penalty is in fact negative and (2) theoretically showed for a spiked covariance model that the optimal ridge penalty is not positive (in the proportional regime). As discussed above, we relate to \cite{JMLR:v21:19-844} by showing that the spiked covariance matrix also experiences the Inflation Property. Additionally, \cite{pmlr-v130-richards21b} and \cite{NEURIPS2020_72e6d323} explored the asymptotics of ridge regression with a negative ridge penalty in the proportional regimes and provided conditions when a negative ridge penalty is optimal. One of the key insights of these papers is that positive ridge regression increases the bias and does not decrease the variance when the signal is well aligned with the strongest directions of the features. That is, a negative ridge penalty is optimal when the signal-to-noise ratio is high and $\beta$ is well aligned with $\Sigma$. The Inflation Property (Equ. \ref{equ:inflation}) is able to handle a vanishing signal-to-noise ratio, potentially, due to the substantially larger set of small directions to ``absorb'' the noise compared to the proportional regime.  \cite{NEURIPS2020_72e6d323} also showed that the negative ridge regression is unique to the over-parameterized case.

In \cite{JMLR:v24:22-1398}, the authors explored how negative regularization interacts with two types of noise: (1) the noise added to the labels and (2) the noise created by the tail of the covariance matrix when trying to estimate the top eigenvalues. Furthermore, the paper hypothesized that negative ridge regression could result in a better than multiplicative improvement. \cite{patil2024optimalridgeregularizationoutofdistribution} found that a negative ridge penalty can be optimal in the out-of-distribution case, even when the training data is isotropic---although the optimality of a negative ridge penalty can also occur in the under-parameterized out-of-distribution case.

The study of the $d/n\to\infty$ regime is natural in the context of modern machine learning methods, where the model parameters often greatly exceed the data size. For example, \cite{kaplan2020scalinglawsneurallanguage} found that the empirical optimal scaling for certain neural networks is in the regime where $d/n\to\infty$. Additionally, when dealing with rare diseases, researchers often grapple with dealing with high-dimensional low sample size data \citep{konietschke2021small}, which is analogous to the $d/n\to\infty$ regime.

Almost all of the aforementioned results are in regimes where $d$ and $n$ are proportional to each other or $d=\infty$. This is in large part due to either the nice properties of a fixed Hilbert space (for the $d=\infty$ case) or the asymptotical laws provided by the Marchenko–-Pastur distribution, respectively. Such distributional advantages are not present when $d/n\to\infty$ at an arbitrary rate because the distribution of the eigenvalues will shrink to an atom around zero in an arbitrary manner.

As a result, many of the techniques discussed in the above papers are inapplicable. For example, \cite{bartlett2020benign} requires that $tr(\Sigma)=o(n)$, \cite{10.1214-21-AOS2133} is predicated on the Marchenko--Pastur distribution, and \cite{JMLR:v21:19-844} explored only the spiked ridge model locally around zero regularization. Further, the tools used in \cite{pmlr-v151-wang22k} are also inapplicable because they heavily rely on a sparsity assumption. 

In \cite{10.1214-21-AOS2133}, the authors consider the proportional case, when $d/n\to\gamma\in(0,\infty)$, and take limit as $\gamma\to\infty$. Several papers have shown that these limits do not commute \citep{benaychgeorges2010eigenvalueseigenvectorsfinitelow, 10.1214/009117905000000233}. The reason these limits do not commute is because the design matrices have many singular values that diverge and are not captured by the limiting Marchenko--Pastur distribution, as a result taking the limit as $\gamma\to\infty$ does not resurrect the mass that has escaped.

We also note that the Inflation Property had not previously been studied and as a result, some of our techniques are necessarily different than the existing literature.
\section{Results}\label{sec:results}
In this section, we state our main results. To do this, we first establish some relevant definitions and notation. We show that, under general conditions, inflating the minimum norm interpolant can achieve an additive improvement in the generalization error. We, then, show stronger conditions where we can obtain a multiplicative improvement in the generalization error. Next, we focus on how to characterize and estimate the optimal amount of inflation. Finally, we show that the use of a data-splitting technique to create an estimator, $\theta_{ds}$, which can sharpen our bounds and improve our estimation procedure of the optimal constant.
\subsection{Definitions and Notation}\label{sec:def_not}
To present our results, we provide the following notation and definitions. We note that we are using the standard Euclidean inner and outer products ($x^Ty$ and $xy^T$) and all vectors and matrices are real-valued.

We study linear regression in the following sense:
\begin{definition}[Linear Regression Setup]\label{def:regress_setup}
We have (i.i.d.) training samples $(x_i,y_i)_{i=1}^n$ where $x_i\in\mathbb{R}^d$, $y_i=\beta^Tx_i+\eps_i$, and $n+1<d$.
\begin{enumerate}
    \item $x_i\sim N(0,\Sigma)$ where $\Sigma$ is symmetric, positive definite.
    \item $\eps_i\sim \P_\eps$ where $\P_\eps$ is such that $\E{\eps_i \, | \, x_i}=0$ and $\E{\eps_i^2\,|\,x_i}\leq \sigma_{max}^2$ for $x_i$-a.s (for some constant $\sigma_{max}^2$ which depends $n$).
    \item $X\in\mathbb{R}^{n\times d}$ s.t. $X_i=x_i^T$, $Y=(y_1,...,y_n)$, and $\vec{\eps}=(\eps_1,...,\eps_n)$.
\end{enumerate}
\end{definition}
We emphasize that that $d$, $\Sigma$, $\mathbb{P}_\eps$, $\sigma_{max}^2$, and $\beta$ are each non-random functions of $n$. We further note that $X$ is often called the design matrix. 

As discussed in the Introduction, the $\ell_2$ minimum norm interpolator plays a key role in our analysis. We restate the definition provided in Equ. \ref{equ:intro_min_norm_def}.
\begin{definition}[Minimum-Norm Interpolator]\label{def:mn_interpolator}
Let $X$, $Y$, and $\vec{\eps}$ be as in Def. \ref{def:regress_setup}. The minimum-norm interpolator ($\theta_{MN}$) is:
\begin{align*}
    \theta_{MN}&=\arg\min\{\norm{\theta}: X\theta=Y\}\\
    &=X^\dagger Y=X^T(XX^T)^{-1}X\beta+X^\dagger\vec\eps
    \\&=\Pi_X\beta+X^\dagger\vec{\eps}
\end{align*}
where $A^\dagger$ is the pseudo-inverse of $A$ and $\Pi_A$ is the orthogonal projection onto $A$ (that is $\Pi_X=X^T(XX^T)^{-1}X$). 
\end{definition}
We seek to characterize the generalization error.
\begin{definition}[Generalization Error]\label{def:gen_and_risk_error}
We define the generalization error, $G(\cdot)$, as the expectation (over all randomness)  of the $\ell_2$ risk, $R(\cdot)$:
\begin{equation}\label{equ:risk_and_gen_error}
\begin{aligned}
    R(\theta(X))&=\mathbb{E}_{x}\left[(\theta(X)^Tx-\beta^Tx)^2\right]\\ 
    G(\theta(\cdot))&=\mathbb{E}_{X,\eps}\left[R(\theta(\cdot))\right]  
\end{aligned}
\end{equation}
\end{definition}

We discussed the Inflation Property briefly in the Introduction. We formally define it as follows:
\begin{definition}[Inflation Property]\label{def:inflation_prop}
    Let $G$ be as in Def. \ref{def:gen_and_risk_error} and $\theta_{MN}$ be as in Def. \ref{def:mn_interpolator}. The Inflation Property occurs when,
    \begin{equation*}
    \exists\,c>1 \text{ s.t. } G(c\theta_{MN})<G(\theta_{MN})
    \end{equation*}
\end{definition}

Although we focus on the existence of some constant where the Inflation Property holds, we are also interested in the optimal constant, which is denoted as follows. 

\begin{definition}\label{def:c_opt}
Let $c_{opt}=\arg\min_{c}G(c\theta_{MN})$.  
\end{definition}
The generalization error is quadratic in $c$ and from the quadratic formula, we find that $c_{opt}=\frac{\E{\theta_{MN}^T\Sigma\beta}}{\E{\theta_{MN}^T\Sigma\theta_{MN}}}$. We study when the numerator is bigger than the denominator.

To facilitate stating our results, we will use the following definitions.
\begin{definition}~\label{def:mean_concentration}
\begin{enumerate}
    \item We write $\lambda_1\geq...\geq \lambda_d$ as the ordered eigenvalues of $\Sigma$ (from Def. \ref{def:regress_setup}) and write $v_i$ as the eigenvector corresponding to the $i$th eigenvalue of $\Sigma$.
    \item We defined the signal-to-noise ratio or SNR as: $\frac{\beta^T\Sigma\beta}{\E{\eps^2}}$
    \item We say a sequence of random variables $X_n\geq 0$ concentrates around its mean if $\frac{\sqrt{\Var[X_n]}}{\E{X_n}}=o(1)$.
    \item Let $r(n)=\frac{tr(\Sigma)^2}{tr(\Sigma^2)}$ and $\sigma^2=\E{tr(\Sigma X^T(XX^T)^{-1}\Lambda(X)(XX^T)^{-1}X)}/\E{tr(\Sigma X^T(XX^T)^{-2}X)}$ where $\Lambda(X)=\text{diag}(\E{\eps_1^2\,|\,X},...,\E{\eps_d^2\,|\,X})$.
    \item If $f(n)=o(g(n))$, then $f(n)/g(n)\to 0$. If $f(n)=\Theta(g(n))$, then $\lim\sup _n f(n)/g(n)\in(0,\infty)$. If $f(n)=O(g(n))$, then $\lim\sup_n f(n)/g(n)\in[0,\infty)$. If $f(n)=\Omega(g(n))$, then $\lim\inf_nf(n)/g(n)\in(0,\infty]$. Finally, if $f(n)=\omega(g(n))$, then $f(n)/g(n)\to\infty$.
\end{enumerate}    
\end{definition}

\subsection{Additive Improvement in Generalization Error}\label{sec:results_theory}
We first provide sufficient conditions on $\Sigma$, $\beta$, and $\P_{\eps}$ such that the inflated estimator provides a strictly better improvement of the generalization error. This result is similar to the shrinkage-type improvements found by \cite{james_estimation_1961} and the series of papers relating to James--Stein estimation that followed. 

We present a few sets of assumptions in Section \ref{sec:results} that can be understood as placing constraints in three key drivers: (1) the effective dimension of the feature covariance, (2) the strength of the alignment between $\beta$ and the top directions of $\Sigma$, and (3) the spectral structure of $\Sigma$ (e.g., the presence of a strong signal subspace).

\begin{assumption}\label{assump:additive} We assume that for all $n$ (where we note that $\beta$, $\Sigma$, $\sigma^2$ depend on $n$) the following holds: 
\begin{enumerate}
    \item $\beta^T\Sigma\beta=1$ and $tr(\Sigma)=d$
    \item $\lambda_1\leq \frac{1}{8}\frac{d}{n}$, $\lambda_i\geq \exp(-n^{1/2})/d$
    \item $\frac{n}{d}\beta^T\Sigma^2\beta-(1+\sigma_{max}^2)\frac{n}{d^2}tr(\Sigma^2)> 0$
\end{enumerate}
\end{assumption}
\begin{theorem}[Additive Improvement]\label{thm:add_improve}
Under Assumption \ref{assump:additive}, whenever $\frac{n}{d}\lambda_1\leq \frac{1}{32}$ for all $n$ and $n$ is sufficiently large (depending on the relationship between $d$ and $n$),
\begin{equation}\label{equ:additive_gen_err_1}
G(\theta_{MN})>G(c\theta_{MN})\; \forall c\in(1,2c_{opt}-1),\;c_{opt}>1    
\end{equation}
 Further, when $\lambda_1=o(\frac{d}{n})$ and $\sigma_{max}^2=\sigma^2$, Assumption \ref{assump:additive}.3 is also a necessary condition for Equ. \ref{equ:additive_gen_err_1} to hold. 
\end{theorem}

We describe this result as ``additive improvement'' because essentially $G(\theta_{MN})-G(c_{opt}\theta_{MN})=\xi_n$ where $\xi_n$ can vanish in the limit. This contrasts to what we call ``multiplicative improvement'', where the improvement does not vanish in the limit and scales $G(\theta_{MN})$ by a constant less than one.

Before presenting stronger assumptions which allow us to obtain a multiplicative improvement to the generalization error, we briefly discuss the assumptions to Theorem \ref{thm:add_improve}.

The role of Assumption \ref{assump:additive}.1 is simply to normalize the problem as a function of $n$ and is done without loss of generality. In the next sub-section, we discuss in depth the purpose of Assumption \ref{assump:additive}.2 in its relation to the effective rank of $\Sigma$. Assumption \ref{assump:additive}.3 requires that $\beta$ align with $\Sigma$ more strongly than its expected alignment with an isotropic random vector in $\mathbb{R}^{d}$ (on a noise adjusted basis). The assumption on the minimum eigenvalue is technical but fairly innocuous given the average eigenvalue is approximately order 1.
\subsection{Multiplicative Improvement}\label{sec:multiplicative_improvement}
To have a hope of learning in this high-dimensional setting beyond just additive improvement, we need the effective dimension of $\Sigma$ to be on the order of $n$ rather than on the order of $d$. In the isotropic case, when $d\gg n$, the asymptotic generalization error approaches the risk of simply predicting zero. Indeed, under Assumption \ref{assump:additive}.1, if $\lambda_1=o(\frac{d}{n})$, then we have the same result where asymptotically we are indifferent to predicting zero (i.e. the additive improvement shrinks to zero).  

The generalization error refers to how well we can predict the output of a fresh draw, $x\sim N(0,\Sigma)$. As such, the prediction risk is mostly dependent on $\Sigma^2$. To have a sufficiently small effective rank, there should be an order $n$ collection of eigen-directions of $\Sigma^2$ whose eigenvalues are large (relative to the tail) and $\beta$ maintain a non-vanishing projection onto this dominant subspace. We note that the tail of $\Sigma$ can still play an out-sized role in determining the distribution of $x$.

With this in mind, we state the following assumptions that are sufficient to obtain a multiplicative improvement in the generalization error:
\begin{assumption}[Strong Multiplicative Assumptions]\label{assumpt:strong_canonical_case} Given some constants $\alpha_{min}\in (0,1)$, $\alpha'\in (0,\infty)$, $\alpha_{noise}\in (0,\infty)$:
\begin{enumerate}
    \item $tr(\Sigma)=d$, $\beta^T\Sigma\beta=1$. 
    \item $\lambda_1= q_{max}\frac{d}{n}$ for $q_{max}\in\left(0,\frac{1}{8}\right]$ where $q_{max}$ is independent of $n$, $\lambda_i\geq \exp(-n^{1/2})/d$.
    \item $1\leq\text{card}(K_n)=O(n)$ where $K_n:=\{ i : \; \forall n'\geq n \;\; \alpha_{min}\,q_{max}\leq \frac{n'}{d(n')}\lambda_i(\Sigma_{n'})\leq q_{max}\}$.
    \item $\forall n\geq 1\;\sum_{i\in K_n}\lambda_i\beta^Tv_iv_i^T\beta>\alpha'>0$.
    \item $(1+\sigma^2)nr(n)\leq \alpha_{noise}q_{max}^2$.
\end{enumerate}
\end{assumption}
We note that Assumption \ref{assumpt:strong_canonical_case} does not imply Assumption \ref{assump:additive}. However, for fixed thresholds $\alpha_{min},\alpha',\alpha_{noise}$, we do get that Assumption \ref{assumpt:strong_canonical_case} implies Assumption \ref{assump:additive} when $q_{max}$ is sufficiently small. 

Before we discuss the assumptions themselves, we provide a few examples of sequences of eigenvalues which meet the above assumptions (ignoring the assumptions on $\beta$ and $\P_\eps$). We show in Prop. \ref{prop:examples_shown} that these examples satisfy Assumption \ref{assumpt:strong_canonical_case}.
\begin{example}\label{example:eigen_decay}~
\begin{enumerate}
    \item $\Sigma=\begin{bmatrix}
q\frac{d}{n}I_{n} & 0 \\
0 & \epsilon I_{d-n} 
\end{bmatrix}$ where $\epsilon=\frac{d}{d-n}(1-q)$, $0<q\leq \frac{1}{8}$, and $\alpha_{noise}>1$ are fixed quantities. 
    \item Suppose for some fixed $\alpha\in(0,1)$, we have $d=\gamma n^{1+\alpha}$ where $\gamma\geq 8$ is a fixed quantity. Consider the spectra $\lambda_i=\frac{1}{\gamma} \frac{d}{n}i^{-a}$ with some fixed $\alpha_{min}$ and $n\gg 1$. 
    \item Let $0<q\leq \frac{1}{8}$, $\alpha_{min}\in [0,1), \alpha_0\geq 0, \alpha_1\in(0,1),\alpha_2\in(0,1)$ be fixed for all $n$ and $d=n^{1-\alpha_2+\alpha_1}/n$. Consider the following spectra for $n\gg 1$:
    \begin{equation*}
        \lambda_i\propto\begin{cases} 
      q\frac{d}{n}(\alpha_{min}+i^{-\alpha_0}) & i\leq n \\
      (\frac{d}{n})^{-\alpha_1}(i-n+1)^{-\alpha_2} & i>n \\
   \end{cases}
    \end{equation*}
    Where $\propto$ refers to the use of constants to ensure that the spectrum sums to $d$.
    \end{enumerate}

\end{example}

Example \ref{example:eigen_decay}.1 helps show the import of the $d/n\to\infty$ assumption. Consider instead that $d/n\to\gamma\in (1,\infty)$. Then, if $q= \frac{1}{\gamma}$, then the covariance matrix would have to isotropic (and these results would not hold) and $q<\frac{1}{\gamma}$ could not occur (because of our eigenvalue ordering and trace condition). Similarly, without the $d/n\to\infty$ condition, the results might not hold for all $n,d$ large.

Assumptions \ref{assumpt:strong_canonical_case}.2 and \ref{assumpt:strong_canonical_case}.5 relate to two notions of effective rank of $\Sigma\;$--- (1) $\;tr(\Sigma)/\lambda_1$ and (2) $tr(\Sigma)^2/tr(\Sigma^2)$. These are the same notions that \cite{bartlett2020benign} characterized has being essential for determining whether benign overfitting occurs. 

Both of these notions require sufficient over-parameterization. On the one hand, Assumptions \ref{assumpt:strong_canonical_case}.2-3 imposes the condition that the effective rank of the order $n$ times some large constant. On the other hand, Assumption \ref{assumpt:strong_canonical_case}.5 requires that the effective rank be over-parameterized in a similar manner when adjusted for noise. In other words, Assumption \ref{assumpt:strong_canonical_case}.5 places a restriction on the SNR. We can also see the model can handle order 1 noise and when $\text{card}(K_n)=o(n)$, the model can handle vanishing a SNR. 

Assumption \ref{assumpt:strong_canonical_case}.3 relates to the separation of the data into two classes of eigenvalues. Assumption \ref{assumpt:strong_canonical_case}.3 can be thought of as looking at the set of eigenvalues that are uniformly in the ``strong signal'' block. These assumptions also require that at least one eigenvalue is growing at $d/n$ and Assumption \ref{assumpt:strong_canonical_case}.4 requires that $\beta$ be aligned with the directions of the ``strong signal'' subspace in a non-vanishing manner.

Assumptions \ref{assumpt:strong_canonical_case}.3-4 can be replaced with a less intuitive but more general assumption that allows us to take advantage of more in-depth interactions between $\Sigma$ and $\beta$ to achieve various technical results. 
\begin{assumption}[Weak Multiplicative Assumptions]\label{assumpt:weak_canonical_case}
For some given constants $C_{noise}>0$, $C_1>0$, 
\begin{enumerate}
    \item $tr(\Sigma)=d$ and $\beta^T\Sigma\beta=1$.
    \item $\lambda_1\leq \frac{1}{8}\frac{d}{n}$ and $\lambda_i\geq \exp(-n^{1/2})/d$.
    \item $\frac{n^2}{d^2}\beta^T\Sigma^{3}\beta\leq C_1q^2$ where $q:=\frac{n}{d}\beta^T\Sigma^2\beta$ and $q$ is constant with respect to $n$.
    \item $(1+\sigma^2)nr(n)\leq C_{noise}q^2$.
\end{enumerate}
\end{assumption}
As seen in Prop. \ref{prop:weak_implies_strong}, Assumption \ref{assumpt:weak_canonical_case} implies Assumption \ref{assumpt:strong_canonical_case}. We also show in Prop. \ref{prop:simple_canon_multiplicative} that an analogous result to Theorem \ref{theorem:canon_multiplicative} occurs under Assumption \ref{assumpt:strong_canonical_case}. With Assumption \ref{assumpt:weak_canonical_case} in mind, we can state the following result:
\begin{theorem}\label{theorem:canon_multiplicative}
    Consider Assumption \ref{assumpt:weak_canonical_case}. There exists an $\alpha\in (0,1)$ (not depending on $n$ or any particular $\Sigma$ but possibly on $C_1$ or $C_{noise}$) such that when $q<\frac{1}{216(C_1+C_{noise})}$ and $n$ sufficiently large (depending on the relationship between $d$ and $n$), $\alpha G(\theta_{MN})\geq G(c_{opt}\theta_{MN})$ and $c_{opt}>1$.  Further, when $r(n)=o(1/n)$, then, $\alpha$ only depends on $C_1$ and the restriction on $q$ becomes $q<\frac{1}{216C_1}$. 
\end{theorem}

As a result of Theorem \ref{theorem:canon_multiplicative}, we can see that the inflation method can create a multiplicative improvement in the generalization error. Although it would be helpful to know whether $q<\frac{1}{C_1+C_{noise}}$ \textit{ a priori}, Prop. \ref{prop:c_estimate_mult} allows us to construct an estimator achieves an improvement analogous to $\alpha$ when the assumptions of Theorem \ref{theorem:canon_multiplicative} are met and only results in marginally worse generalization error in other cases, as compared to $G(\theta_{MN})$ (see Prop. \ref{prop:ds_MN_close_expect}).

As discussed above regarding Example \ref{example:eigen_decay}.1, the general allowance of $q$ to be sufficiently small may not hold if $d/n$ does not go to infinity.
\subsection{Characterizing \texorpdfstring{$c_{opt}$}{c\_{opt}} in the Multiplicative Case}
Given the importance of $c_{opt}$ (see Def. \ref{def:c_opt}), we look at how it relates to $\beta$ and $\Sigma$.

\begin{proposition}\label{prop:weak_c_opt_bounds}
    Consider Assumption \ref{assumpt:weak_canonical_case} and $(1+\sigma^2)r(n)=o(1/n)$. Then, there are universal constants $k_1$ and $k_2$ such that the following is true,
    \begin{equation*}
    \frac{\frac{n}{d}\beta^T\Sigma^2\beta}{\frac{n^2}{d^2}\beta^T\Sigma^3\beta}(1+k_1q) \leq c_{opt}\leq \frac{\frac{n}{d}\beta^T\Sigma^2\beta}{\frac{n^2}{d^2}\beta^T\Sigma^3\beta}(1+k_2q)       
    \end{equation*}

\end{proposition}

When $q$ is sufficiently small, $c_{opt}\approx \frac{\frac{n}{d}\beta^T\Sigma^2\beta}{\frac{n^2}{d^2}\beta^T\Sigma^3\beta}$. This has a few interpretations:
\begin{enumerate}
    \item[\textbullet] We can write $c_{opt}\approx\|\hat{\beta}\|_I^2\;/\;\|\hat{\beta}\|_{(\frac{n}{d}\Sigma)}^2$ where $\hat{\beta}=\frac{n}{d}\Sigma\beta$. Under this viewpoint, $c_{opt}$ is adjusting $\beta$ for the differences in its standard Euclidean geometry and its geometry under $\Sigma$. Returning to the Johnson--Lindenstrauss Lemma interpretation, the minimum-norm interpolant is projecting onto a space that contracts the lengths and $c_{opt}$ is inflating those lengths to bring them back to what they should be under an isotropic projection.
\item[\textbullet] Let $\lambda$ be a random variable distributed such that $\P(\lambda=\frac{n}{d}\lambda_i)=\frac{n\lambda_i}d(\beta^Tv_i)^2$. Then, $c_{opt}=\E{\lambda}/\E{\lambda^2}$. This lens describes the minimum-norm interpolant as being overly dependent on the alignment of $\beta$ and $\Sigma^3$ and $c_{opt}$ removing that dependence in favor of the stronger dependence between $\beta$ and $\Sigma^2$.    
    \item[\textbullet] As discussed previously, $\frac{n}{d}\beta^T\Sigma^2\beta-(1+\sigma^2)nr(n)>0$ means that $\beta$ is better aligned with $\Sigma$ than its average alignment with a random isotropic vector. Since $(1+\sigma^2)nr(n)=o(1)$, $\beta$ is markedly better aligned to $\Sigma$ than a random isotropic vector. Because in this case $\frac{n^2}{d^2}\beta^T\Sigma^3\beta<\frac{n}{d}\beta^T\Sigma^2\beta$, we know that $c_{opt}>1$ whenever $\beta$'s alignment with $\Sigma$ dominates a random isotropic vector. 
\end{enumerate}
\subsection{Obtaining a Consistent Estimator Through Data Splitting}\label{sec:data_splitting}
In order to obtain a consistent estimator, we make use of data splitting\footnote{With additional assumptions on the dependence between $\eps$ and a draw of the data, it is possible to obtain concentration for $\theta_{MN}$ when $(1+\sigma^2)nr(n)=o(1)$.}. Although cross-validation type approaches are common in ridge regression techniques \citep{liu2020ridgeregressionstructurecrossvalidation, Delaney01041986}, we use the data splitting primarily to create an estimator $\theta_{ds}$ for the unknown vector $\beta$ by taking advantage of an averaging like effect which allows us to obtain consistent estimators and concentration. We do use a hold-out estimation technique to estimate the optimal inflation constant. We define the data splitting procedure and calculation of $\theta_{ds}$ as follows:

\begin{definition}\label{def:data_splitting} Let $X$, $Y$, and $\vec\eps$ be from Def. \ref{def:regress_setup}.
\begin{enumerate}
    \item Let $N\in\mathbb{N}$ to be the number of data splits.
    \item For $1\leq i\leq N-1$, $X^{(i)}:=(x_{(i-1)\lfloor \frac{n}{N}\rfloor+1},...,x_{i\lfloor \frac{n}{N}\rfloor})$.
    \item $X^{(N)}:=(x_{(N-1)\lfloor \frac{n}{N}\rfloor+1},...,x_{n})$.
    \item Let $\vec{\eps}^{(i)}$ and $Y^{(i)}$ be the elements of $\vec\eps$ (resp. $Y$) corresponding to the elements in $X^{(i)}$.
    \item Let $\theta^{(i)}_{MN}:=(X^{(i)})^T((X^{(i)})(X^{(i)})^T)^{-1}Y^{(i)}$.
    \item Let $\theta_{ds}:=\sum_{i=1}^{N-1}\theta_{MN}^{(i)}$
    \item Let $c^*=\arg\min_c G(c\theta_{ds})$
\end{enumerate} 
\end{definition}
At first glance, the use of data splitting to obtain better results in this anti-regularization setup seems contradictory. Much of the literature has focused on the regularization properties of data splitting \citep{pmlr-v151-muecke22a, 10.1162-neco.1997.9.5.1163, bagging_regularizes}. In the context of data splitting in kernel ridge regression, \cite{pmlr-v30-Zhang13} found that each ``local'' estimator (the estimator on the subset) is anti-regularized and required a positive ridge penalty to be optimal. 

At first glance, it may seem that we need to divide by $1/N$ in the definition of $\theta_{ds}$. But because of the high dimensionality, $\E{\theta_{MN}^{(i)}}\approx \frac{n/N}{d}\Sigma\beta$ (see Prop. \ref{prop:proj_expect}) and so $\theta_{ds}$ already reflects an implicit average. 

Here, when we split data into a growing number of partitions, we are able to use the averaging like effect to create an estimator ($\theta_{ds}$) which performs very similarly to $\theta_{MN}$ but is consistent and allows us to obtain better concentration results. 

We first off note that the generalization error between $\theta_{ds}$ and $\theta_{MN}$ are comparable. For avoidance of doubt, $c_{opt}$ remains as defined in Def. \ref{def:c_opt}. 
\begin{proposition}\label{prop:ds_MN_close_expect}
Consider Assumption \ref{assumpt:weak_canonical_case} and $N\to\infty$. Let $\rho:=\frac{n}{d}\lambda_1$. Then, for $c>0$, 
\begin{align*}
|G(c\theta_{ds})-G(c\theta_{MN})|\leq 3c\rho q+15(C_1+C_{noise})c^2\rho q^2    
\end{align*}. 

Further, as $c_{opt}\leq \frac{4}{3(C_1+C_{noise})q}$, 
\begin{align*}
|G(c_{opt}\theta_{ds})-G(c_{opt}\theta_{MN})|\leq 30\rho    
\end{align*}

Further,
\begin{align*}
    |G(c^*\theta_{ds})-G(c_{opt}\theta_{MN})|\leq 3q
\end{align*}
\end{proposition}
\begin{proposition}
    Consider Assumption \ref{assumpt:weak_canonical_case}, $\lambda_1\leq\rho\frac{d}{n}$. Then, 
    \begin{align*}
        |c^*-c_{opt}|\leq\kappa \rho
    \end{align*}
    Where $\kappa$ is universal. 
\end{proposition}

We can think of $G(c_{opt}\theta_{MN})$ and $G(c^*\theta_{ds})$ to be $\Theta(1)$. However, using $\theta_{ds}$, we are able to have cases where $ G(c^*\theta_{ds})=o(1)$ but $G(\theta_{ds})$ is pretty close to one.
\begin{assumption}[Rate Improvement]\label{assumpt:rate_improve}~
\begin{enumerate}
    \item Assumption \ref{assumpt:weak_canonical_case} holds.
    \item $\frac{n^2}{d^2}\beta^T\Sigma^3\beta\leq q^2(1+o(1))$.
    \item $(1+\sigma^2)nr(n)=o(1)$.
    \item $\sigma_{max}^2=\exp(o(n^{1/3}))$
\end{enumerate}
\end{assumption}
We note these are fairly restrictive, and we basically require that $\Sigma\beta$ be very heavily aligned, with an eigenspace of $\Sigma$ that has a restricted range of eigenvalues. Using the probability interpretation of the distribution ($\lambda$), we can think of $\lambda$ as having low variance, but this does not mean that $\Sigma$ is isotropic (an isotropic covariance matrix does satisfy Assumption \ref{assumpt:rate_improve}.2-4). This can be seen with the following example:
\begin{align*}
    \Sigma=\begin{bmatrix}
q\frac{d}{n}I_{n} & 0 \\
0 & \epsilon I_{d-n}
\end{bmatrix} \text{ and } \beta= \begin{bmatrix}
\beta_{align} \\
0
\end{bmatrix}
\end{align*}

In other words, $\beta$ is well aligned with the top-eigenspace of $\Sigma$. We note that $\beta^T\Sigma\beta=1$ which implies that $\norm{\beta}^2=\frac{n}{d}q$. Thus, $\frac{n}{d}\beta^T\Sigma^2\beta =q$ and $\frac{n^2}{d^2}\beta^T\Sigma^3\beta=q^2$.

Additionally, for Assumption \ref{assumpt:rate_improve}.3-4, we require that the effective rank when adjusted for noise to be $o(n)$. We note that the restriction on $\sigma_{max}^2$ is relatively benign given that $\sigma^2$ is limited to being order $n^2$.

\begin{proposition}
    Suppose Assumption \ref{assumpt:rate_improve} holds. Then, $G(\theta_{ds})=\Theta(1)$ and $$\min_cG(c\theta_{ds})=o(1)$$
\end{proposition}

From here, we look at showing that $R(c\theta_{ds})$ concentrates around its mean (see Def. \ref{def:mean_concentration}).
\begin{proposition}\label{prop:concentration_estimators}
Consider Assumption \ref{assumpt:weak_canonical_case}. Let $c>0$ be constant and assume $n$ sufficiently large (depending on the relationship between $d$ and $n$). Let $M$ be some large number that is fixed. Suppose one of the following is true:
\begin{enumerate}
    \item $\sigma_{max}^2=o(\sqrt{n})$; or
\item $\sigma^2nr(n)=o(1)$, $\sigma^2=o(n/\log(n)^2)$ and $\sigma_{max}^2/\sigma^2=o(N^{1/2})$;

\end{enumerate}

For any $M>c>0$, $\Var[c\theta_{ds}]=o(1)$. As a result, when $\min_c G(c\theta_{ds})=\omega(1)$, $R(c\theta_{ds})$ concentrates around its mean (as in Def. \ref{def:mean_concentration}).

\end{proposition}

With additionally regularity on the functional dependence between $\eps$ and a draw of the data, we can obtain concentration of $G(c\theta_{ds})$ when it (or a subsequence) is $o(1)$.

Now, using the concentration of $R(c\theta_{ds})$, we are able create an estimator for $c^*$ such that $\hat{c}^*\theta_{ds}$ is able to achieve the generalization error of $G(c^*\theta_{ds})$ in probability. To estimate $\hat{c}^*$, we use a hold-out estimation technique.  Essentially, the way we constructed $\theta_{ds}$, there is one partition of data left over. We use this left over data to help construct an estimator of $c^*$ and the remaining data is used to construct $\theta_{ds}$ as described in Def. \ref{def:data_splitting}.6. Specifically, we construct our estimator of $c^*$, as follows:
\begin{align*}
    \hat{c}^*=\frac{\sum_{i=1}^k\theta_{ds}^Tz_i\eta _i}{\sum_{i=1}^k(\theta_{ds}^Tz_i)^2}
\end{align*}

Where $k$ is the size of $X^{(N)}$ and $z_i,\eta_i$ represent the corresponding $i$-th element of $X^{(N)}$ and $Y^{(N)}$.

The use of $\theta_{ds}$ allows us to use the data in an efficient manner. Because in high-dimensions we see every direction basically once, it is important that only a vanishingly small amount of data is held out and this allows us to hold out a small amount of data while also simultaneously leveraging the whole data set. 

\begin{proposition}\label{prop:c_estimate_mult}
We assume that Assumption \ref{assumpt:weak_canonical_case} holds and $\sigma_{max}^2=o(\sqrt{n})$. Then, we can construct a specific number of splits ($N$) to create estimators $\theta_{ds}$,  $\hat{c}^*$, and $\delta_n=o(1)$ such that,
    \begin{align*}
        \mathbb{P}(|\hat{c}^*-c^*|>\delta_n)=o(1)\\
        \mathbb{P}(|R(\hat{c}^*\theta_{ds})-G(\hat{c}\theta_{ds})|>\delta_n)=o(1)\\
    \end{align*}
When the assumptions of Theorem \ref{theorem:canon_multiplicative} hold, then,
\begin{align*}
    \mathbb{P}(R(\hat{c}^*\theta_{ds})>\alpha G(\theta_{MN}))\to 0
\end{align*}
Where $\alpha$ is as constructed in Theorem \ref{theorem:canon_multiplicative}.
\end{proposition}
We note the foregoing proposition shows that when the Inflation Property occurs we attain the improvement and otherwise, the generalization error is vanishingly close to the generalization error of the optimally scaled, $\theta_{ds}$. Further, the generalization error of our estimator is relatively close to that of the optimally scaled $\theta_{MN}$ when $\rho$ is small.
\section{Experimental Results} In this section, we provide empirical results to highlight some of the theoretical results found in Section \ref{sec:results_theory} with simulated data. Using the simulated data, we are able to expand upon the results by showing how the Inflation Property performs against (negative) ridge regression and see how it performs asymptotically. We, then, provide empirical simulations that suggest that the Inflation Property can hold, even when the distribution of the data ($x_i$) is non-Gaussian. Finally, we show for a real-world data sets that the Inflation Property holds. For most of these results, we also provide a comparison to negative ridge regression. Because, for a fixed negative penalty, the matrix $(XX^T+\lambda I_n)^{-1}$ can be indefinite, positive, or negative definite depending on the draws of the data, studying negative ridge is not analytically tractable using the tools in this paper. Similarly, the theoretical results in \cite{JMLR:v21:19-844} were limited to showing that the penalty must be negative but did not study further.
\subsection{Simulation Setup}\label{sec:sim_setup}
We begin our simulations by following the setup in the latter portion of Section 2.3 of \cite{JMLR:v21:19-844}. That is, we consider $\sigma^2=.1$, $\Sigma=(1-\rho)I_d+\rho\,\mathbf{1}\mathbf{1}^T$ where $\rho=.1$, and $\beta=(b,...,b)$ with $b=\sqrt{1/(d(1-\rho+d\rho))}$. 
\begin{figure}[H]
  \centering

  \begin{subfigure}[t]{0.465\textwidth}
    \centering
    \resizebox{\linewidth}{!}{\input{figure-sim-diff-1.pgf}}
    \caption{Generalization error of the inflated minimum-norm interpolator versus ridge regression across sample sizes, with $\rho=5/n$ and $d=n\log(n)$.}
    \label{fig:fig_sim_diffy}
  \end{subfigure}
  \hfill
  \begin{subfigure}[t]{0.465\textwidth}
    \centering
    \resizebox{\linewidth}{!}{\input{figure-sim-SNR.pgf}}
    \caption{Generalization error of the Inflation Property estimator versus ridge regression across signal-to-noise ratios, varying $\sigma$ to change the noise level.}
    \label{fig:fig_sim_snr}
  \end{subfigure}

  \caption{Comparison of the Inflation Property estimator and ridge regression across sample-size and signal-to-noise-ratio regimes. The dashed lines represent the Inflation Property.}
  \label{fig:fig_sim_combined}
\end{figure}
\begin{figure}[H]
  \centering

  \begin{subfigure}[t]{0.465\textwidth}
    \centering
    \resizebox{\linewidth}{!}{\input{figure-sim10.pgf}}
    \caption{The optimal inflation constant ($c_{opt}$) is plotted as $n$ grows for various $d/n$ ratios. }
    \label{fig:fig_sim_asym_2}
  \end{subfigure}
  \hfill
  \begin{subfigure}[t]{0.465\textwidth}
    \centering
    \resizebox{\linewidth}{!}{\input{figure-sim4-10.pgf}}
    \caption{This figure plots the probability (as a function of $n$) the empirical risk for the estimator $\hat{c}^*\theta_{ds}$ (where $\hat{c}^*$ is constructed in the same manner as Prop. \ref{prop:c_estimate_mult}) is less than the generalization error of the optimally scaled $\theta_{MN}$.}
    \label{fig:fig_sim_asym_1}
  \end{subfigure}

  \caption{These plots look at how the Inflation Property performs as $n$ grows. In Figure \ref{fig:fig_sim_asym_2}, the Inflation Property holds even for small $n$. As seen in Figure \ref{fig:fig_sim_asym_1}, we see that empirically the convergence in Prop. \ref{prop:c_estimate_mult} occurs very quickly (which shows that the Inflation Property is relevant even in relatively small data sets).  }
  \label{fig:fig_sim_combined_asymp}
\end{figure}

\subsection{Non-Gaussian Design}
We consider the same setup as in Section \ref{sec:sim_setup} with the exception of the distribution of $x$. Here, we consider when $x$ is distributed as Chi-Squared, Laplace, and Exponential (where we adjust the distribution to be mean zero and adjusting the variance of each variate to be one).

As seen in Figure \ref{fig:fig_sim_1}, the Inflation Property occurred when the assumptions were relaxed to consider distributions beyond Gaussian distributions. From theoretical side, it is unclear to what extent it is possible to generalize our theorems beyond Gaussian (or nearly Gaussian) distributions. The key propositions that allow us to understand expectations of and properties of the matrix $(XX^T)^{-1}$ require the use of Stein's Lemma (see e.g. Props. \ref{prop:proj_expect}, \ref{prop:exp_sig_exp}). Understanding $(XX^T)^{-1}$ is non-trivial when $X$ does not have i.i.d. standard Normal entries. One might think that a Lindenberg method (of slowly replacing variates with Gaussian distributions and tracking the differences) could yield theoretical results. This does not yield results in a tractable manner due to the matrix inversion operations and empirically, it looks like if it would work, it would only work for certain distributions (Figure \ref{fig:fig_sim_1} shows a clear difference between generalization error with different input distribution).

Figure \ref{fig:fig_sim_1} also reveals a complicated picture about how negative ridge regression and the Inflation Property operate on different distributions. In this case, negative ridge regression under-performs the Inflation Property for all distributions. It is notable that negative ridge regression performs significantly better on the Exponential distribution compared to other distributions. Here, the Exponential distribution has comparable tails to the Laplace distribution---this indicates the heaviness of the tails is not necessarily responsible for the difference. 

\begin{figure}[H]
  \centering 
  \input{figure-sim1-3.pgf}
  \captionsetup{justification=centering}
  \caption{This figure compares negative ridge regression (solid lines) and inflation $c\theta_{MN}$ (dashed lines) across Gaussian, centered $\chi^2$, Laplace distribution ($a=0$, and $b=\frac{1}{\sqrt{2}}$), and the Exponential distribution (with scale being one). Each distribution has been normalized to have a mean zero and a variance of one. The legend colors are consistent across each pair of inflation and negative ridge regression curves.}

  \label{fig:fig_sim_1}
\end{figure}

\subsection{Real Data}
We repeat the setup described in Section 2.6 of \cite{JMLR:v21:19-844}\footnote{In Section 2.1 of \cite{JMLR:v21:19-844}, the authors study a different data set but ultimately conclude positive ridge regression was helpful. Likewise, the Inflation Property does not hold on this data set.}, except (1) we transformed the labels and the features to be mean zero (under the training data) and (2) after generating the features, we rescaled the features such that the trace of its covariance matrix is equal to the number of parameters ($d=2,000$). We end up finding that the Inflation Property holds on this data set. We note that the expected generalization error associated with the optimal inflationary constant is slightly lower than the expected generalization error arising from the optimal (negative) ridge penalty. This difference was statistically significant (p-value: 0.0299), and with probability one, the inflated/negative ridge estimator outperformed the standard $\theta_{MN}$ estimator.
\begin{figure}[H]
  \centering 
  \input{figure-kobak-1.pgf}
  \captionsetup{justification=centering}
\caption{In this figure, the test error for the set up described in \cite{JMLR:v21:19-844} is compared between the Inflation Property and ridge regression.}
  \label{fig:fig_2}
\end{figure}
In addition to the data from \cite{JMLR:v21:19-844}, we consider two high dimensional, low sample size data sets. For each data set, we employ the following procedure. First, we use a bootstrap procedure to take several separate random samples (with the remaining data for each random sample being used as a test set). We split the data seventy-five percent for the training sample and the balance going to the test set. Second, we normalize each bootstrap sample such that the sample covariance matrix of the features has a trace equal to the number of the features in the dataset and the mean is zero. The normalization used for the training sample is used on the corresponding test set. Third, the test error is calculated for each bootstrap sample for a given estimator and is then averaged over the bootstrapped samples. As seen in Figure \ref{fig:fig_3}, negative ridge regression is not always inferior to the Inflation Property.

\begin{figure}[H]
  \centering

  \begin{subfigure}[t]{0.48\textwidth}
    \centering
    \resizebox{\linewidth}{!}{\input{figure-newset-3.pgf}}
    \caption{Data set used in \cite{Yeoh:2002wg}, a classification problem with six classes, $n=248$, and $d=12625$. This data set consists of samples of pediatric acute lymphoblastic leukemia, with features given by gene-expression measurements from oligonucleotide microarrays.}
    \label{fig:fig_3}
  \end{subfigure}
  \hfill
  \begin{subfigure}[t]{0.48\textwidth}
    \centering
    \resizebox{\linewidth}{!}{\input{figure-newset-2.pgf}}
    \caption{Data set used in \cite{Sorlie:2001kr}, a classification problem with five classes, $n=85$, and $d=456$. This data set consists of eighty-five breast tissue samples, with features reflecting micro-array signals from a variety of complementary DNA clones.}
    \label{fig:fig_4}
  \end{subfigure}

  \captionsetup{justification=centering}
  \caption{This figure reflects two real-world data set and the effects of (negative) ridge regression and inflation on the test error.}
  \label{fig:fig_combined_datasets}
\end{figure}

\section{Proof Sketch}\label{sec:proof-sketch}
In this section, we provide a proof sketch for some of the key concepts used in the results described in Section \ref{sec:results}. From elementary matrix operations on $G(c\theta_{MN})$, we can see that the key objects of study are $\Pi_X=X^T(XX^T)^{-1}X$ and $X^T(XX^T)^{-1}$ both of which ultimately revolve around $(XX^T)^{-1}=:A^{-1}$. 

Specifically, we will use these objects to show a proof sketch for why $\E{\Pi_X}\approx \frac{n}{d}\Sigma$ and why the key elements of Prop. \ref{prop:c_estimate_mult} hold. See Def. \ref{def:mn_interpolator}.
\begin{proposition}[Informal]\label{prop:informal_pi}
    Assume $\lambda_{max}(\Sigma)\leq q\frac{1}{k+1}\frac{d}{n}$ where $q$ is small. Then, $\E{\Pi_X}\approx \frac{n}{d}\Sigma$.
\end{proposition}

We note this is a simplification of the bounds in Prop. \ref{prop:proj_expect}.

We note that $A\stackrel{d}{=}\sum_{i=1}^d\lambda_i(\Sigma)a_ia_i^T$ where $a_i\sim N(0,I_n)$ i.i.d. and that,
\begin{align*}
    v_i(\Sigma)^T\Pi_Xv_j(\Sigma)\stackrel{d}{=}\sqrt{\lambda_i}\sqrt{\lambda_j}a_i^TA^{-1}a_j
\end{align*}
By independence and mean zero property, to calculate $\E{\Pi_X}$ we just need to consider the diagonal entries (in the basis of the eigenvectors of $\Sigma$). From Sherman-Morrison,
\begin{align}\label{equ:aAInva}
    a_i^TA^{-1}a_i=a_i^TA_i^{-1}a_i-\frac{\lambda_i(a_i^TA_i^{-1}a_i)^2}{1+\lambda_ia_i^TA_i^{-1}a_i}
\end{align}
Where for a set of indexes, $\alpha$, we say that $A_\alpha:=\sum_{i\notin \alpha, i\leq d} \lambda_ia_ia_i^T$.

As a result, we need to determine the expectation of $tr(A^{-1})$ and $tr(A^{-2})$ (although we sketch for an arbitrary $k$ for $A^{-k}$). We note that this is a simplified form of Prop. \ref{prop:tr_A_bounds}.
\begin{proposition}[Informal]
    Assume $\lambda_{max}(\Sigma)\leq q\frac{1}{k+1}\frac{d}{n}$ where $q$ is small. Then,
    \begin{align*}
    \E{tr(A^{-k})}\approx \frac{\E{tr(A^{-k+1})}}{d}\approx \frac{n}{d^k}    
    \end{align*}

    Where $A^0=I_n$.
\end{proposition}
\begin{proof}
    
    Fix some $i\in\N, i\leq d$. $F(a_i):=A^{-k}a_i$. 
\begin{align*}
    \sum_{j=1}^n\frac{\partial F(a_i)}{\partial a_{ij}}&=\sum_{j=1}^n\partial a_{ij}e_j^TA^{-k}a_i= \sum_{j=1}^ne_jA^{-1}e_j+e_j^T(\partial a_{ij} A^{-k})a_i
    \\&=tr(A^{-k})-\lambda_i\left(\sum_{r=0}^{k-1}a_i^TA^{-k+r}a_itr(A^{-r-1})+ka_i^TA^{-k-1}a_i\right)
\end{align*}

Putting this together and using Stein's Lemma (Lemma \ref{lemma:steins}), we note that,
\begin{align*}
    \E{tr(A^{-k+1})}&=\sum_{i=1}^d\lambda_i\E{a_i^TF(a_i)}=\sum_{i=1}^d\lambda_i\sum_{j=1}^n\E{\frac{\partial F(a_i)}{\partial a_{ij}}}
    \\&=\sum_{i=1}^d\lambda_i\E{tr(A^{-k})}-\sum_{i=1}^d\lambda_i^2\left(\sum_{r=0}^{k-1}\E{a_i^TA^{-k+r}a_itr(A^{-r-1})+a_i^TA^{-k-1}a_i}\right)
    \\&\geq d\E{tr(A^{-k})}(1-q)
\end{align*}

The result comes from our use our bound on $\lambda_1$ and the fact that:
\begin{align*}
    \sum_{i=1}^d\lambda_i\left(\sum_{r=0}^{k-1}\E{a_i^TA^{-k+r}a_itr(A^{-r-1})+a_i^TA^{-k-1}a_i}\right)\approx \E{tr(A^{-k})}(1-q)
\end{align*}

As a result, $\E{tr(A^{-k})}\geq \frac{n}{d^k(1-q)^k}$.

From the above equations, we can see that since $A$ is symmetric, positive definite and we can show that $a_i^TA^{-k+r}a_i\geq 0$ for any relevant $r$. As a result,
\begin{align*}
    \sum_{i=1}^d\lambda_i^2\left(\sum_{r=0}^{k-1}\E{a_i^TA^{-k+r}a_itr(A^{-r-1})+a_i^TA^{-k-1}a_i}\right)\geq 0
\end{align*}
Thus, using our original set of equations from Stein's lemma, when $k>1$,
\begin{align*}
   \E{tr(A^{-k+1})}\geq  d\E{tr(A^{-k})}
\end{align*}

Further, we note that from properties of trace matrices and Jensen's Inequality, 
\begin{align*}
\E{tr(A^{-1})}\geq \E{\frac{n}{tr(A)}}\geq \frac{n}{d}    
\end{align*}
As a result, $\E{tr(A^{-k})}\approx \frac{n}{d^k}$ when $q$ is small.
\end{proof}
We now turn to completing the informal proof of Prop. \ref{prop:informal_pi}.
\begin{proof}
Using this informal proposition, we can see that,
\begin{align*}
    \E{\frac{\lambda_i(a_i^TA_i^{-1}a_i)^2}{1+\lambda_i^TA_i^{-1}a_i}}\leq \E{a_i^TA_i^{-1}a_i}=\E{tr(A_i^{-1})}\leq \frac{n}{d(1-q)}
\end{align*}
Where we keep with the notion that $\lambda_{max}(\Sigma)\leq \frac{q}{k+1}\frac{d}{n}$.

Conversely, for a lower bound, we use Jensen Inequality,
\begin{align*}
    \E{\frac{\lambda_i(a_i^TA_i^{-1}a_i)^2}{1+\lambda_ia_i^TA_i^{-1}a_i}}\geq \frac{\lambda_i\E{a_i^TA_i^{-1}a_i}^2}{1+\lambda_i\E{a_i^TA_i^{-1}a_i}}\geq\frac{\lambda_i\frac{n^2}{d^2}}{1+\lambda_i\frac{n}{d}}\geq \frac{n}{d}\frac{q}{1+q}
\end{align*}

Thus, $\E{\Pi_X}\approx \frac{n}{d}\Sigma$.
\end{proof}
We next sketch a proof of the variances of $\theta_{ds}^T\Sigma\theta_{ds}$ and $\theta_{ds}^T\Sigma\beta$. This essentially uses a simplified form of Props. \ref{prop:exp_sig_exp}, \ref{prop:var_proj_sig_proj}, \ref{prop:var_exp_beta}  in the context of data splitting. We note that each $\theta_{MN}^{(i)}$ consists of approximately $n/N$ data points and so $\E{\theta_{MN}^{(i)}}\approx \frac{n/N}{d}$. 
\begin{proposition} [Informal]
Assume that Assumption \ref{assumpt:weak_canonical_case} holds and $\eps=0$. Then,
\begin{align*}
\E{\theta_{ds}^T\Sigma\beta}&\approx\frac{n}{d}\beta^T\Sigma^2\beta\\
\E{\theta_{ds}^T\Sigma\theta_{ds}}&\approx nr(n)+\frac{n^2}{d^2}\beta^T\Sigma^3\beta\\
    \Var\left[\beta^T\Sigma\theta_{ds}\right]&=o(1)\\
    \Var[\theta_{ds}^T\Sigma\theta_{ds}]&=o(1)
\end{align*}
\end{proposition}
\begin{proof}
For the first term, we can easily use Prop. \ref{prop:informal_pi} as follows:
\begin{align*}
   \E{\theta_{ds}^T\Sigma\beta}&\approx(N-1)\beta^T\E{\Pi_{X^{(1)}}}\Sigma\beta\approx \frac{n}{d}\beta^T\Sigma^2\beta\\ 
\end{align*}

For the second term, similarly,
\begin{align*}
    \E{\theta_{ds}^T\Sigma\theta_{ds}}\approx \frac{n^2}{d^2}\beta^T\Sigma^3\beta+\sum_{i=1}^d\E{\beta^T\Pi_{X^{(i)}}\Sigma\Pi_{X^{(i)}}\beta}
\end{align*}

We handle the term $\E{\beta^T\Pi_{X^{(i)}}\Sigma\Pi_{X^{(i)}}\beta}$, as follows, where we let $A:=\sum_{i=1}^d\lambda_ia_ia_i^T$ and $a_i$ is i.i.d. $N(0,I_{\lfloor n/N \rfloor})$,
\begin{align*}
    v_j^T\Pi_{X^{(i)}}\Sigma\Pi_{X^{(i)}}v_k\stackrel{d}{=}\sqrt{\lambda_k\lambda_j}\sum_{\ell=1}^da_j^TA^{-1}a_\ell a_\ell^TA^{-1}a_k
\end{align*}

We note that the non-diagonal terms (in the basis of eigenvectors of $\Sigma$) are mean zero. As result, considering diagonal terms and employing a similar procedure as in Prop. \ref{prop:informal_pi}, we get:
\begin{align*}
    \E{\Pi_{X^{(i)}}\Sigma\Pi_{X^{(i)}}}\approx \frac{n/N}{d^2}\left(\sum_{\ell=1}^d\lambda_\ell^2\right)\Sigma+\frac{n^2/N^2}{d^2}\Sigma^3
\end{align*}

Here, we can see the power of data splitting in removing lower order terms. As such,
\begin{align*}
    \E{\theta_{ds}^T\Sigma\theta_{ds}}=\frac{n}{d^2}tr(\Sigma^2)\beta^T\Sigma\beta+\frac{n^2}{d^2}\beta^T\Sigma^3\beta
\end{align*}

For our variance terms, we begin by noting that,
\begin{align*}
    \Var[\theta_{ds}^T\Sigma\beta]&=\sum_{i=1}^N\Var[\beta^T\Sigma\theta_{MN}^{(i)}]\leq \sum_{i=1}^N\E{(\theta_{MN}^{(i)})^T\Sigma\beta\beta^T\Sigma\theta_{MN}^{(i)}}
\end{align*}

Essentially, we use that $\theta_{MN}^{(i)}$ is on the order of $\frac{n}{N}\Sigma\beta$. Because we are essentially taking the ``square'' and summing over $N$ elements, we end up with an extra division by $N$. This occurs in a related sense for $\Var[\theta_{ds}^T\Sigma\theta_{ds}]$. 

\end{proof}

We now turn to a proof sketch of part of Prop. \ref{prop:c_estimate_mult}.
\begin{proposition}[Informal] Assume that Assumption \ref{assumpt:weak_canonical_case} holds and $\eps=0$. Then we can construct a specific number of splits ($N$) to create estimators $\theta_{ds}$ and $\hat{c}^*$ such that for any $\delta>0$,
\begin{align*}
    \mathbb{P}(|\hat{c}^*-c^*|>\delta)&=o(1)\\
    \mathbb{P}(|R(\hat{c}^*\theta_{ds})-G(c^*\theta_{ds})|>\delta)&=o(1)
\end{align*}
    
\end{proposition}
\begin{proof}
    We note that our goal is to estimate $c^*$ (recall Def. \ref{def:data_splitting}), which can be described at, 
    \begin{align*}
        c^*=\frac{\E{\theta_{ds}^T\Sigma\beta}}{\E{\theta_{ds}^T\Sigma\theta_{ds}}}=\frac{\frac{n}{d}\beta^T\Sigma^2\beta}{\frac{n^2}{d^2}\beta^T\Sigma^3\beta+nr(n)}(1+o(1))
    \end{align*}

    We consider the samples within, $X^{(N)}$. That is, $x_{(N-1)\lfloor \frac{x}{N}\rfloor},...,x_n$. For notational convenience, we denote them as $z_1,...,z_k$ and the corresponding outputs as $\eta_1,...,\eta_k$ where $k$ is the size of $X^{(N)}$. We note that $k\approx \frac{n}{2N}$. Since $\theta_{ds}$ is constructed from the $X^{(1)},...,X^{(N-1)}$, these samples are held out from our original estimator. 

    Let $\hat{q}:=\frac{1}{k}\sum_{i=1}^k\theta_{ds}^Tz_i\eta_i$ and $\hat{r}:=\frac{1}{k}\sum_{i=1}^k(\theta_{ds}^Tz_i)^2$. Let $\hat{c}^*:=\hat{q}/\hat{r}$.

    We note that using some properties of Gaussian distributions, $\E{\hat{q}}=\frac{n}{d}\beta^T\Sigma^2\beta(1+o(1))$ and $\E{\hat{r}}=nr(n)(1+o(1))+\frac{n^2}{d^2}\beta^T\Sigma^3\beta(1+o(1))$.

    By the conditional variance on all the samples except $X^{(N)}$ (which we will denote $X^s$),
    \begin{align*}
        \Var\left[\left.\frac{1}{k}\sum_{i=1}^k\theta_{ds}^Tz_i\eta_i\right|X^s\right]&=\frac{1}{k}\theta_{ds}^T\Var[z_1\eta_1]\theta_{ds}\leq \frac{1}{k}\theta_{ds}^T\Sigma\theta_{ds}+\frac{1}{k}(\theta_{ds}^T\Sigma\beta)^2\\
        \Var\left[\left.\frac{1}{k}\sum_{i=1}^k(\theta_{ds}^Tz_i)^2\right|X^s\right]&=\frac{1}{k}\Var[\theta_{ds}^Tzz^T\theta_{ds}]\leq \frac{2}{k}(\theta_{ds}^T\Sigma\theta_{ds})^2
    \end{align*}

    Now, turning to the law of total variance,
    \begin{align*}
        \Var\left[\frac{1}{k}\sum_{i=1}^k\theta_{ds}^Tz_i\eta_i\right]&\leq \E{\frac{1}{k}\theta_{ds}^T\Sigma\theta_{ds}+\frac{1}{k}(\theta_{ds}^T\Sigma\beta)^2}+\Var\left[\E{\left.\theta_{ds}^Tz_1\eta_1\right|X^s}\right]
        \\&=O\left(\frac{1}{k}\right)+o(1)
    \end{align*}
    Where we note that $\Var\left[\E{\left.\theta_{ds}^Tz_1\eta_1\right|X^s}\right]=\Var[\theta_{ds}^T\Sigma\beta]=o(1)$.
    We take $N$ to be such that $k$ is growing to infinity and, additionally, $N\to\infty$. Essentially, we need that $n/N\to\infty$ and $N\to\infty$. With the presence of noise, we need to perform a more complex balancing act. 
    
    By a similar procedure we can show analogously $\Var\left[\frac{1}{k}\sum_{i=1}^k\theta_{ds}^Tz_i\eta_i\right]=o(1)$.

    Let $\delta>0$ be fixed for all $n$. Now, by the Continuous Mapping Theorem we know that $\hat{c}^*\to c^*$ in probability and 
    \begin{align*}
        \mathbb{P}(|R(\hat{c}\theta_{ds})-G(c^*\theta_{ds})|>\delta)=o(1)
    \end{align*}
\end{proof}
\section{Proofs}
\subsection{Results from Linear Algebra}
\begin{definition}[Loewner Order]
    Let $A,B\in\mathbb{R}^{n\times n}$. Suppose $B-A$ is symmetric positive semi-definite (that is, $x^T(B-A)x\geq 0$ for all $x\in\mathbb{R}^n$) and $A,B$ have the same eigenvectors, then we say that $A\preceq B$. 
\end{definition}

\begin{proposition}\label{prop:cauchy_tr_tr_inverse}
If $C\in \mathbb{R}^{n\times n}$ is symmetric positive definite, then $\frac{n^2}{\tr(A)}\leq \tr(A^{-1})$  
\end{proposition}
\begin{proof}
    By Cauchy-Schwartz, for any strictly positive sequence ($c_1,...,c_n$), $n^2\leq (\sum_{i}c_i)(\sum_{i}\frac{1}{c_i})$. As such, $n^2\leq \tr(A)\tr(A^{-1})$ and the result follows.  
\end{proof}
\begin{definition}\label{def:same_eigenstruct}
    Suppose $A$ and $B$ share the same eigenvectors. We say that matrices $A$ and $B$ are simultaneously diagonalizable with the same ordering if $\lambda_i(A)$ and $\lambda_i(B)$ can correspond to the same eigenvectors (up to permutations for identical eigenvalues).   
\end{definition}
\begin{proposition}[Chebyshev's Sum Inequality]\label{prop:chebyshev_sum_inequal}
Suppose $a_1\geq...\geq a_n$ and $b_1\geq ...\geq b_n$. Then, $\left(\frac{1}{n}\sum_{k=1}^na_k\right)\left(\frac{1}{n}\sum_{k=1}^nb_k\right)\leq \left(\frac{1}{n}\sum_{k=1}^na_kb_k\right)$
\end{proposition}
A proof of this is available on pg. 43 of \cite{hardy1988inequalities}.
\begin{proposition}\label{prop:tr(AB)_leq_ntr(A)tr(B)}
Suppose $A,B\in\mathbb{R}^{n\times n}$ and meet Def. \ref{def:same_eigenstruct}, then $\tr(A)\tr(B)\leq n\tr(AB)$.
\end{proposition}
\begin{proof}
Let $\lambda_1(A)\geq...\geq\lambda_n(A)$ be the eigenvalues of $A$ and $\lambda_1(B)\geq...\geq\lambda_n(B)$ be the eigenvalues of $B$. 

By the simultaneously diagonalizable condition of $A$ and $B$, $AB$ has the following (non-ordered) eigenvalues  $\lambda_1(A)\lambda_1(B),...,\lambda_n(A)\lambda_n(B)$. With that, we use Chebyshev's Sum Inequality (Prop. \ref{prop:chebyshev_sum_inequal}) to obtain:
\begin{align*}
    \tr(AB)=\sum_k\lambda_k(A)\lambda_k(B)\geq \frac{1}{n}\left(\sum_{k}\lambda_k(A)\right)\left(\sum_{k}\lambda_k(B)\right)=\frac{1}{n}\tr(A)\tr(B)
\end{align*}
\end{proof}
\begin{proposition}\label{prop:arb_mult_matrix_trace}
Let $B_1,...,B_s\in\mathbb{R}^{n\times n}$ and meets the requirements of Def. \ref{def:same_eigenstruct} in a pairwise manner. Then $\Pi_{k=1}^s\tr(B_k)\leq n^{s-1}\tr(\Pi_{k=1}^sB_k)$    
\end{proposition}
\begin{proof}

For this proof we will use an inductive method to obtain the result. Consider the base case when $s=1$, this trivially holds as $tr(B_1)\leq tr(B_1)$.
Fix $s\in\mathbb{N}, s>1$. Assume that the result holds for all $s'\leq s$. Then, by Prop. \ref{prop:tr(AB)_leq_ntr(A)tr(B)},
\begin{align*}
\tr(\Pi_{k=1}^sB_k)=\tr(B_s\Pi_{k=1}^{s-1}B_k)\leq n\tr(B_s)\tr(\Pi_{k=1}^{s-1}B_k)\leq nn^{s-2}\Pi_{k=1}^s\tr(B_k)=n^{s-1}\Pi_{k=1}^s\tr(B_k)
\end{align*}

\end{proof}
\begin{proposition}\label{prop:bounds_effective_ranks}
Let $\Sigma$ be a positive definite matrix which is symmetric and $\lambda_1$ be the top eigenvalue of $\Sigma$. Then $\frac{tr(\Sigma)}{\lambda_1}\leq \frac{tr(\Sigma)^2}{tr(\Sigma^2)}\leq \frac{tr(\Sigma)^2}{\lambda_1^2}$    
\end{proposition}
The proof follows from Lemma 5 of \cite{bartlett2020benign}.
\begin{proposition}\label{prop:A_diffs}
    Let $A$ be a positive definite matrix, symmetric, and be differentiable. Then, $d(A^r)=\sum_{k=0}^{r-1}A^{k}d(A)A^{r-k-1}$ and $d(A^{-1})=-A^{-1}d(A)A^{-1}$.    
\end{proposition}
\begin{proof}
We note that matrix product rule says, $\frac{d}{dx}D(x)E(x)=D(x)E'(x)+D(x)E'(x)$.

By properties of matrix inverse, the result inverse matrix identity result follows from noting that:
\begin{align*}
    0=\frac{d}{dx}I_n&=\frac{d}{dx}AA^{-1}=A(A^{-1})'+A'A^{-1}
\end{align*}

Turning to the first expression, $d(A^1)=d(A)$ and as such the base case is satisfied.

Now fix some $r\in\mathbb{N}$ and assume that for all $r'<r$, $d(A^{r'})=\sum_{k=0}^{r'-1}A^{k}d(A)A^{r'-k-1}$.

We note that,
\begin{align*}
    d(A^r)=d(AA^{r-1})&=Ad(A^{r-1})+d(A)A^{r-1}\\
    &=A\sum_{k=0}^{r-2}A^{k}d(A)A^{r-k-2}+d(A)A^{r-1}\\
    &=\sum_{k=0}^{r-2}A^{k+1}d(A)A^{r-k-2}+d(A)A^{r-1}\\
    &=\sum_{k=1}^{r-1}A^kd(A)A^{r-k-1}+A^0d(A)A^{r-0-1}\\
    &=\sum_{k=0}^{r-1}A^kd(A)A^{r-k-1}
\end{align*}

\end{proof}
\begin{proposition}
    Let $A=\sum_{i=1}^d\lambda_ia_ia_i^T$ and $r\in\mathbb{N}$. Then, $$\sum_ja_i^T\partial_{a_{ij}}(A^{-r})e_j=-\lambda_i\left(\sum_{k=0}^{r-1}a_i^TA^{-r+k}a_itr(A^{-k-1})+ra_i^TA^{-r-1}a_i\right)$$
\end{proposition}
\begin{proof}
We use Prop. \ref{prop:A_diffs} to find the following:
\begin{align*}
    a_i^T\partial_{a_{ij}}(A^{-r})e_j&=a_i^T\sum_{k=0}^{r-1}-A^{-r+1}A^k\partial_{a_{ij}}(A)A^{r-k-1}A^{-r}e_j=\lambda_ia_i^T\left(-\sum_{k=0}^{r-1}A^{-r+k}(a_ie_j^T+e_ja_i^T)A^{-k-1}\right)e_j\\&=\lambda_i\left(-\sum_{k=0}^{r-1}a_i^TA^{-r+k}a_ie_j^TA^{-k-1}e_j-\sum_{k=0}^{r-1}a_i^TA^{-r+k}e_ja_i^TA^{-k-1}e_j\right)
\end{align*}
As such, summing over $j$, the result follows.
\end{proof}
\begin{proposition}[Weyl's Inequality for Eigenvalues]\label{prop:weyl_inequality}
Let $\lambda_1(X),...,\lambda_n(X)$ be the ordered eigenvalues of a matrix $X$. Let $X,Y\in\mathbb{R}^{n\times n}$ Then,    
\begin{align*}
    \lambda_{j+k-1}(X+Y)\leq \lambda_{j}(X)+\lambda_k(Y)
\end{align*}
Where $j,k\in\mathbb{N}$ and $j+k-1\leq n$,
\end{proposition}
This follows from pgs. 239-240 of \cite{Horn_Johnson_2012}. 
\begin{proposition}\label{prop:pd_moments}
    Suppose $A$ is positive definite and symmetric. Then $A^k$ is positive definite and symmetric for any $k\in \mathbb{Z}$.
\end{proposition}
\begin{proof}
    For $k>0$, this follows trivially by the eigenvalue decomposition of $A$.
    
    When $k=0$, then $A^k=I$ which is trivially positive definite.

    By noting that $A^{-1}$ is also symmetric and positive definite, the result follows.
\end{proof}
\begin{proposition}[Sherman-Morrison Formula]
    Let $A\in\mathbb{R}^{n\times n}$, $A$ is invertible, and $u^TAv\neq-1$ where $u,v\in\mathbb{R}^n$. Then, $(A+uv^T)^{-1}=A^{-1}-\frac{A^{-1}uv^TA^{-1}}{1+u^TAv}$.
\end{proposition}
This follows directly from \cite{10.1214-aoms-1177729893}.
\begin{proposition}\label{prop:A_higher_moment_psd}
    Let $A:=\sum_ib_ib_i^T$,$A_i:=A-b_ib_i^T$ and $k\in\mathbb{N}$. Then, 
    \begin{enumerate}
        \item $A-A_i$ is positive semi-definite and $A_i^{-1}-A^{-1}$ is positive semi-definite.
        \item $b_i^TA_i^kb_i\leq b_i^TA^kb_i$ and $b_i^TA^{-k}b_i\leq b_iA_i^{-k}b_i$
        \item $\tr(A_i^k)\leq \tr(A^k)$ and $\tr(A^{-k})\leq \tr(A_i^{-k})$
        \item $A^{-1}b_ib_i^TA^{-1}\preceq A_i^{-1}b_ib_i^TA_i^{-1}$
    \end{enumerate}
    
\end{proposition}
\begin{proof}

Fix $x$ as a vector. $x^TA-A_ix=x^Tb_ib_i^Tx\geq 0$. Additionally, the following:
\begin{align*}
    x^T(A_i^{-1}-A^{-1})x=x^T\left(A_i^{-1}-A_i^{-1}+\frac{A_i^{-1}b_ib_i^TA_i^{-1}}{1+b_i^TA_i^{-1}b_i}\right)x=\frac{1}{1+b_i^TA_i^{-1}b_i}(x^TA_ib_i)^2\geq 0
\end{align*}

For the second part (non-inverse case), we want to show that $b_i^TEb_i\geq 0$ where $E$ is made from an arbitrary multiplication of $b_ib_i^T$ and $A_i$. We will prove this by induction.

Let $r\in\mathbb{N}$, $E_0:=A_i$, $E_1:=b_ib_i^T$, and $E:=E_{i_1}...E_{i_r}$ where $i\in \{0,1\}^r$.

Suppose $r=1$, if $i_1=0$, then $b_i^TA_ib_i\geq 0$ by p.s.d. of $A_i$. Otherwise, noting that $b_i^Tb_ib_i^Tb_i\geq 0$ concludes the base case.

Suppose $r>1$ and $b_i^TEb_i\geq 0$ for any $E=E_{i_1}...E_{i_{r'}}$ where $r'<r$. If $i_l=0\;\forall \, l\leq r$, then $b_i^TEb_i=b_i^TA_i^rb_i\geq 0$. Otherwise, $\exists \, l\leq r$ such that $i_l=1$. Now, 
\begin{align*}
    b_i^TEb_i=b_i^T(E_{i_1}...E_{i_{l-1}}b_ib_i^TE_{i_{l+1}}...E_{i_r})b_i=(b_i^TE_{i_1}...E_{i_{l-1}}b_i)(b_i^TE_{i_{l+1}}...E_{i_{r}}b_i)\geq 0
\end{align*}
Where the last inequality comes from the inductive hypothesis.

Thus, as $b_i^TA^kb_i=b_i^T(A_i+b_ib_i^T)^kb_i\geq b_i^TA_i^kb_i$ where we note the remaining terms in the binomial expansion are of the form $E=E_{i_1}...E_{i_{k}}$.

We now turn to the inverse case, let $C(s):=A_i+sb_ib_i^T$ and $B(s):=b_i^T(C(s))^{-k}b_i$ where we note that $B(0)=b_i^TA_i^{-k}b_i$ and $B(1)=b_i^TA^{-k}b_i$. We then take the derivative and it is is negative on $(0,1)$,
\begin{align*}
    \frac{d}{ds}B(s)&=b_i^T\partial_s((A_i+sb_ib_i^T)^{-k})b_i=-b_i^T\sum_{r=0}^{k-1}C(s)^{r-k}\partial_s(C(s))C(s)^{k-r-1+k}\\& =-\sum_{r=0}^{k-1}b_i^TC(s)^{r-k}b_ib_i^TC(s)^{-r-1}b_i \leq 0
\end{align*}
Where we use Prop. \ref{prop:A_diffs} and we note that $C(s)$ is positive definite on $s\in[0,1]$ (see Prop. \ref{prop:pd_moments}).

As a result, $b_i^TA^{-k}b_i\leq b_i^TA_i^{-k}b_i$. 

For the third part. We note that if $E:=-b_ib_i^T$, then $A_i=A+E$ and $E$'s eigenvalues are $0$ with multiplicity of $n-1$ and $-b_i^Tb_i$.

Using Prop. \ref{prop:weyl_inequality}, with $k=1$, we see that $\lambda_j(A_i)=\lambda_j(A+E)\leq \lambda_j(A)+\lambda_1(E)=\lambda_j(A)$. As a result, $\lambda_j(A_i)^k\leq \lambda_j(A)^k$ and the trace results follow directly.

For the fourth item, we note the following expansion (where we let $\gamma:=A_i^{-1}b_ib_i^TA_i^{-1}$):
\begin{align*}
    A^{-1}b_ib_i^TA^{-1}&=\left(A_i^{-1}-\frac{A_i^{-1}b_ib_i^TA_i^{-1}}{1+b_i^TA_i^{-1}b_i}\right)b_ib_i^T\left(A_i^{-1}-\frac{A_i^{-1}b_ib_i^TA_i^{-1}}{1+b_i^TA_i^{-1}b_i}\right)\\
    &=\gamma-2\frac{\gamma b_i^TA_i^{-1}b_i}{1+b_i^TA_i^{-1}b_i}+\frac{A_i^{-1}b_i(b_i^TA_i^{-1}b_i)^2b_i^TA_i^{-1}}{(1+b_i^TA_i^{-1}b_i)^2}
    \\&=A_i^{-1}b_ib_i^TA_i^{-1}\left(1-2\frac{b_i^TA_i^{-1}b_i}{1+b_i^TA_i^{-1}b_i}+\frac{(b_i^TA_i^{-1}b_i)^2}{(1+b_i^TA_i^{-1}b_i)^2}\right)
    \\&\leq A_i^{-1}b_ib_i^TA_i^{-1}
\end{align*}

The result follows from noting the following: (1) both, $A^{-1}b_ib_i^TA^{-1}$ and $A_i^{-1}b_ib_i^TA_i^{-1}$ are positive definite, and (2)
\begin{align*}
 \left(1-2\frac{b_i^TA_i^{-1}b_i}{1+b_i^TA_i^{-1}b_i}+\frac{(b_i^TA_i^{-1}b_i)^2}{(1+b_i^TA_i^{-1}b_i)^2}\right)=\frac{1}{(1+b_i^TA_i^{-1}b_i)^2}   
\end{align*}

\end{proof}
\subsection{Results from Probability Theory}
\begin{proposition}[Continuous Mapping Theorem]\label{prop:cts_mapping_theorem}
    Suppose $X_n$ and $Y_n$ converge in probability to $X$ and $Y$ respectively. Then, for some continuous function, $F$ constant (with respect to $n$) and non-random, $F(X_n,Y_n)$ converges in probability to $F(X,Y)$. Further, suppose $\P[Y\leq 0]=0$, then, $X_n/Y_n$ converge in probability to $X/Y$. 
\end{proposition}
This follows directly from \cite{Durrett_2019}, Theorems 2.3.4 and 3.2.10. 
\begin{proposition}\label{prop:total_law_probability}
Suppose $X,Y$ are random variables with $\Var[Y]<\infty$. Then,
\begin{align*}
    \Var[Y]=\E{\Var[Y\,|\,X]}+\Var[\E{Y\,|\,X}]
\end{align*}
\end{proposition}
\begin{proof}
The result follows from noting that,
\begin{align*}
    \Var[Y]=\E{Y^2}-\E{Y}^2=\E{\E{Y^2\,|\,X}}-\E{\E{Y\,|\,X}^2}+\E{\E{Y\,|\,X}^2}-\E{\E{Y\,|\,X}^2}
\end{align*}
\end{proof}
\begin{proposition}\label{prop:gen_cauchy_scwartz}
Let $X_1,...,X_s$ be a sequence of random variables. Then,
\begin{align*}
    \E{\Pi_{k=1}^sX_k}\leq \Pi_{k=1}^s\E{X_k^{2^s}}^{1/{2^s}}
\end{align*}
\end{proposition}
\begin{proof}
Suppose $s=1$, $\E{X_1}\leq \E{X_1^2}^{1/2}$ by Cauchy-Schwartz.

Fix $s>1$. Suppose that $\E{\Pi_{k=1}^{s-1}X_k}\leq \Pi_{k=1}^{s-1}\E{(X_k)^{2^{s-1}}}^{1/{2^{s-1}}}$. Then,
\begin{align*}
    \E{(\Pi_{k=1}^{s-1}X_k)X_{s}}&\leq \E{(\Pi_{k=1}^{s-1}X_k)^2}^{1/2}\E{X_s^2}^{1/2}\leq \E{X_s^2}^{1/2}\Pi_{k=1}^{s-1}\E{(X_k)^{2s}}^{1/{2s}}
    \\&\leq \E{X_s^{2^s}}^{1/2^s}\Pi_{k=1}^{s-1}\E{(X_k)^{2s}}^{1/{2s}}
\end{align*}

And the result follows.

\end{proof}
\begin{proposition}\label{prop:fubini_derivative_interchange}
Let $f(x,y)$ be a differentiable function with respect to $y$ around $y_0$. Assume further that $\exists\,\delta>0, C>0 \;\forall\, y\in(y_0-\delta,y_0+\delta)$, $\E{|f(x,y)|}<C<\infty$ where the expectation is over $x$. Then,
\begin{align*}
    \frac{\partial}{\partial y}\E{f(x,y)}\biggr|_{y=y_0}=\E{\frac{\partial}{\partial y}f(x,y)\biggr|_{y=y_0}}
\end{align*}
\end{proposition}
This proposition follows by Fubini-Tonelli Theorem (Theorem 2.37 of \cite{follandrealanalysis}).
\begin{proposition}\label{prop:norm_decomp}
    Let $x\sim N(0,\Sigma)$ where $\Sigma\in\mathbb{R}^{d\times d}$ is a positive definite, symmetric covariance matrix. Then, $x=\sum_{i=1}^da_i\sqrt{\lambda_i}v_i$ where $a_i\sim N(0,1)$.
\end{proposition}
This result is seen in \cite{16223}.
\begin{lemma}[Hanson-Wright Inequality]\label{lemma:hanson-wright}
Let $z\sim N(0,I_k)$. Let $B\in\mathbb{R}^{k\times k}$ which is positive semi-definite and symmetric. Then, for every $t\geq 0$,
\begin{align*}
    \mathbb{P}\left(\left|z^TBz - \E{z^TBz}\right|>t\right)\leq 2\exp\left[-\eta\min\left(\frac{t^2}{tr(B^2)},\frac{t}{\norm{B}_2}\right)\right]
\end{align*}
Where $\eta$ is a universal constant. 
If $B$ is not positive semi-definite and symmetric, then, for every $t\geq 0$,
\begin{align*}
    \mathbb{P}\left(|z^TBz - \E{z^TBz}|>t\right)\leq 2\exp\left[-\eta\min\left(\frac{t^2}{tr(B^T B)},\frac{t}{\sqrt{\norm{B^TB}_2}}\right)\right]    
\end{align*}
\end{lemma}
The proof follows from Theorem 1.1 of \cite{rudelson2013hanson}.

\begin{proposition}\label{prop:general_chi_square_moment}
Let $z\sim N(0,I_k)$. Let $0\preceq B$ with a rank of $\ell$ and $g(k)\to 0$. Assume $\frac{\norm{B}_2}{tr(B)}=\omega(g(k))$, $\frac{tr(B)^2}{tr(B^2)}=\omega(g(k))$, and $\lambda_{min}\geq \omega(\exp(-g(k))/\ell)$. Fix $q\in \mathbb{N}$, $\alpha\in (0,1)$, then, 
\begin{align*}
    \E{\frac{1}{(z^TBz)^q}}\leq \frac{1}{tr(B)^q(1-\alpha)^q}(1+o(1))
\end{align*}
\end{proposition}
\begin{proof}
Let $Q=z^TBz$ and $E=\{Q<tr(B)(1-\alpha)\}$. Using Cauchy-Schwartz and Markov Inequality, we note that:
\begin{align*}
    \E{\frac{1}{Q^q}}&=\E{\frac{1}{Q^q}\mathbf{1}(E)}+\E{\frac{1}{Q^q}\mathbf{1}(E^c)}\leq \frac{1}{tr(B)^q(1-\alpha)^q}+\E{\frac{1}{Q^{2q}}}^{1/2}\mathbb{P}(E^c)^{1/2}
    \\&\leq\frac{1}{tr(B)^q(1-\alpha)^q}+\frac{1}{\lambda_{min}^q}\E{\frac{1}{(\chi^2_\ell)^{2q}}}^{1/2}\mathbb{P}(E^c)^{1/2}
    \\&\leq \frac{1}{tr(B)^q(1-\alpha)^q}+\frac{\Theta(1)}{\lambda_{min}^q\ell^{q}}\mathbb{P}(E^c)^{1/2}
\end{align*}

The conclusion follows from our assumptions and Lemma \ref{lemma:hanson-wright},
\begin{align*}
    \mathbb{P}(E^c)\leq \mathbb{P}(|z^TBz-\E{z^TBz}|>\alpha tr(B))\leq 2\exp\left(-\eta\min\left(\frac{\alpha ^2tr(B)^2}{tr(B^2)},\alpha\frac{tr(B)}{\norm{B}_2}\right)\right)\leq \exp(-g(k))
\end{align*}

\end{proof}
\begin{proposition}\label{prop:general_chi_square_moment_taught}
Let $z\sim N(0,I_k)$. Let $0\preceq B$ with a rank of $\ell$ and $g(k)\to 0$. Further assume that, $\frac{\norm{B}_2}{tr(B)}=\Omega(g(k))$, $\frac{tr(B)^2}{tr(B^2)}=\Omega(g(k))$, and $\lambda_{min}=\Omega(1)$. Fix $q\in \mathbb{N}$, then, 
\begin{align*}
    \E{\frac{1}{(z^TBz)^q}}\leq \frac{1}{tr(B)^q(1-1/h(k))^q}(1+o(1))
\end{align*}
Where $h(k)=o(g(k))$
\end{proposition}
\begin{proof}
    We essentially reuse Prop. \ref{prop:general_chi_square_moment} and note that in the calculation of $\mathbb{P}(E^c)$, the term $\alpha^2\frac{tr(B)^2}{tr(B^2}=e^{-\omega(1)}$ when $\alpha^2\ll \frac{1}{g(k)}$. The same reasoning applies for the other term.
\end{proof}
\begin{lemma}[Stein's Lemma]\label{lemma:steins}
Let $a$ be a vector that is $N(0,I_n)$ and $F:\mathbb{R}^n\to \mathbb{R}^n$, then $\E{a^TF(a)}=\E{\sum_{j=1}^n \frac{\partial F_j(a)}{\partial a_j}}$. We also impose the following regularization conditions: $\E{a_jF_j(a)} < \infty$ and $\E{| \frac{\partial F_j(a)}{\partial a_j}|} < \infty$.
\end{lemma}
\begin{proof}
The classical Stein's Lemma (see Theorem 1 of \cite{LANDSMAN2008912}) is for $g:\mathbb{R}\to\mathbb{R}$, $x\sim N(0,1)$, and subject to $\E{|xg(x)|}+\E{|g'(x)|}<\infty$,
\begin{align*}
\E{g(x)x}=\E{g'(x)}    
\end{align*}

Now we apply the Tower Property and Classical Stein's Lemma as follows, \begin{align*}
 \E{a^TF(a)}=\sum_{j=1}^n\E{a_jF_j(a)}=\sum_{j=1}^n\E{\E{a_jF_j(a) | \;a_1,...a_{j-1},a_{j+1},...,a_n}}=\sum_{j=1}^n\E{\frac{\partial F_j(a)}{\partial a_j}}   
\end{align*}  
\end{proof}
\begin{proposition}\label{prop:stein_satsified}
Let $A=\sum_{i=1}^d\lambda_ia_ia_i^T$ where $\lambda_i>0$ and $a_i\sim N(0,I_n)$ (i.i.d.). Consider a fixed $i$ and $k$ and let $F(x)=A^{-k}x$. Then, the regularity conditions of Lemma \ref{lemma:steins} are satisfied.   
\end{proposition}
\begin{proof}
Let $a_i^c:=(a_1,...,a_{i-1},a_{i+1},...,a_d)$. Further, let $A_i:=A-\lambda_ia_ia_i^T$,

We note that, $$\E{a_i^TF(a_i)\;|\;a_i^c}=\E{a_i^TA^{-k}a_i\;|\;a_i^c}\leq \E{a_i^TA_i^{-k}a_i\;|\;a_i^c}=tr(A^{-k}_i)<\infty$$
Where the first inequality comes from Prop. \ref{prop:A_higher_moment_psd} and the second comes from the invertibility of $A_i$. By Jensen's Inequality, this is sufficient for the first regularity condition of Lemma \ref{lemma:steins}.

We now prove the second part of the proposition. First, we derive an expression for $\sum_j \E{\left.|\frac{\partial A^{-k}}{\partial a_{ij}}|\right|\; a_i^c}$ for some $k\in\mathbb{N}$.
\begin{align*}
    \sum_{j}\left|\frac{\partial F_j(a_j)}{\partial a_{ij}}\right|&=\sum_{j}|\partial_{a_{ij}}e_j^TA^{-k}a_i|=\sum_{j}|e_j^TA^{-k}e_j+e_j^T(\partial_{a_{ij}}A^{-k})a_i|
    \\&=\sum_{j}\left|e_j^TA^{-k}e_j-\lambda_i\sum_{k=0}^{r-1}\left(a_i^TA^{-r+k}a_ie_j^TA^{-k-1}e_j+a_i^TA^{-r+k}e_je_j^TA^{-k-1}a_i\right)\right|
    \\&\leq \sum_{j}e_j^TA^{-k}e_j+\lambda_i\sum_{k=0}^{r-1}\left(a_i^TA_i^{-k+r}a_itr(A_i^{-r-1})+ka_i^TA_i^{-k-1}a_i\right)
    \\&\leq tr(A_i^{-k})+\lambda_i\sum_{k=0}^{r-1}\left(a_i^TA_i^{-k+r}a_itr(A_i^{-r-1})+ka_i^TA_i^{-k-1}a_i\right)=:\alpha_*
\end{align*}

Where the third inequality comes from Prop. \ref{prop:A_diffs} and the inequality comes from Prop. \ref{prop:A_higher_moment_psd} (as well as properties of positive semi-definite matrices).

As $\alpha_*\geq 0$ and by invertibility of $A_i$ (and hence arbitrary powers of $A_i$),  
\begin{align*}
   \sum_j \E{\left.|\frac{\partial A^{-k}}{\partial a_{ij}}|\right|\; a_i^c}\leq \E{\alpha_*\;|\;\ a_i^c}=tr(A_i^{-k})+2\lambda_i\left(\sum_{r=0}^{k-1}tr(A_i^{-k+r})tr(A_i^{-r-1})+ktr(A_i^{-k-1})\right)<\infty
\end{align*}
\end{proof}

\begin{proposition}\label{prop:moments_z_quadratics}
Let $B_1,...,B_s\in\mathbb{R}^{n\times n}$ and the $B_k$ matrices pairwise satisfy Def. \ref{def:same_eigenstruct}. Further, let $z\sim N(0,I_n)$. Then $\E{\Pi_{k=1}^sz^TB_kz}=\left(n^{s-1}+O(n^{s-2})\right)tr(\Pi_{k=1}^sB_k)$
\end{proposition}
\begin{proof}
    We note by Def. \ref{def:same_eigenstruct}, any product of $B_k$ matrices will commute. Let us denote the common set of eigenvectors as $\{v_i\}_{i=1}^n$ and the corresponding eigenvalues of $B_k$ to be $\lambda_i(B_k)$.

    Now, we note that $z^TB_kz=\sum_{l}\lambda_l(B_k)(z^Tv_l)^2=\sum_l\lambda_l(B_k)y_l^2$ where $y_i\sim N(0,1)$ and i.i.d. by orthogonality of $v_i$.
    
    As such, 
    \begin{align*}\label{equ:expanded_sum_z_B_z}
        \Pi_{k=1}^sz^TB_kz=\sum_{l_1,...,l_s}\Pi_{k=1}^s\lambda_{l_k}(B_k)y_{l_k}^2
    \end{align*}

Let $f:[n]^s\to{0,1}$ such that $f(p)=1$ if for any $i\neq j$, $p_i\neq p_j$ and is zero otherwise.

Then, 
\begin{align*}
    \E{\Pi_{k=1}^sz^TB_kz}=\E{\sum_{\stackrel{l=(l_1,...,l_s)}{f(l)=1}}\Pi_{k=1}^s\lambda_{l_k}(B_k)y_{l_k}^2+\sum_{\stackrel{l=(l_1,...,l_s)}{f(l)=0}}\Pi_{k=1}^s\lambda_{l_k}(B_k)y_{l_k}^2}
\end{align*}
Now we note that, 
\begin{align*}
n^{s-1}tr(\Pi_{k=1}^sB_k)&=\Pi_{k=1}^str(B_k)=\sum_{l_1,...,l_s}\Pi_{k=1}^s\lambda_{l_k}(B_k)=
\sum_{l_1,...,l_s}\Pi_{k=1}^s\lambda_{l_k}(B_k)\E{y_{l_k}^2}\\&\geq\E{\sum_{\stackrel{l=(l_1,...,l_s)}{f(l)=1}}\Pi_{k=1}^s\lambda_{l_k}(B_k)y_{l_k}^2}    
\end{align*}

Turning to the $f(l)=0$ case, we note that by Jensen's Inequality,
\begin{align*}
    \Pi_{|r_1|+...+|r_s|=s}\E{(\chi_1^2)^{|r_i|}}\leq \E{(\chi_1^2)^s}^s=:C
\end{align*}

Thus,
\begin{align*}
   \E{\sum_{\stackrel{l=(l_1,...,l_s)}{f(l)=0}}\Pi_{k=1}^s\lambda_{l_k}(B_k)y_{l_k}^2}\leq  C\sum_{\stackrel{l=(l_1,...,l_s)}{f(l)=0}}\Pi_{k=1}^s\lambda_{l_k}(B_k)
\end{align*}

Now, we seek to exploit that for each element in the sum, (at least) two indices must match. As a result,
\begin{align*}
    \sum_{\stackrel{l=(l_1,...,l_s)}{f(l)=0}}\Pi_{k=1}^s\lambda_{l_k}(B_k)&=\sum_{q,r}\sum_{\stackrel{l_1,...,l_s}{l_q=l_r}}\Pi_{k=1}^s\lambda_{l_k}(B_k)=\sum_{q,r}tr(B_q^2)\Pi_{k\neq q,k\neq r}tr(B_k)\\&
    \leq \sum_{q,r}n^{s-3}tr(\Pi_{k\neq q,k\neq r}B_k)tr(B_q^2)
    \\&\leq \sum_{q,r}n^{s-2}tr(\Pi_{k=1}^sB_k)\leq s^2n^{s-2}tr(\Pi_{k=1}^sB_k)
\end{align*}

Where the last and penultimate inequalities use Prop. \ref{prop:arb_mult_matrix_trace}. 
As such, the result follows.

\end{proof}
\begin{corollary}
    Let $z\sim N(0,I_n)$ and $B\in \mathbb{R}^{n\times n}$ be a positive definite matrix. Let $s\in \mathbb{N}$ and $k_1,...,k_s\in \mathbb{N}$. Then, $\E{\Pi_{l=1}^sz^TB^{k_l}z}=(n^{s-1}+O(n^{s-2}))tr(B^{k_1+...+k_s})$
\end{corollary}
\begin{proof}
    We note that any power of $B$ has the same eigenvectors as any other power of $B$ and that the ordering is maintained because $x^r$ is an increasing function for any $r\geq 0$. As such, the result concludes by Prop. \ref{prop:moments_z_quadratics}.
\end{proof}
\begin{proposition}\label{prop:second_third_quadratic_forms}
Let $z\sim N(0,I_n)$. Let $B,C,D\in\mathbb{R}^{n\times n}$. Then,
\begin{enumerate}
    \item $\E{z^TBzz^TCz}=tr(BC)+tr(CB)+tr(B)tr(C)$  
    \item $\E{z^TBzz^TCzz^TDz}=tr(B)tr(C)tr(D)+2tr(B)tr(CD)+2tr(C)tr(BD)+2tr(D)tr(BC)+8tr(BCD)$
\end{enumerate}  
\end{proposition}
\begin{proof}
For the first result. We note that $\E{z^TBzz^TCz}=\sum_{i,j,k,l}\E{z_iz_jz_kz_le_i^TBe_je_k^TCe_l}$ and we use Isserlis' Theorem. 

When $i=j$ and $k=l$, we have $\E{\sum_{i,k} z_i^2z_k^2e_i^TBe_ie_k^TCe_k}=tr(B)tr(C)$.

When $i=k$ and $j=l$ or $i=l$ and $j=k$, we have $\E{\sum_{i,j} z_i^2z_j^2e_i^TBe_je_i^TCe_j}=tr(BC)$

The result follows.

Turning to the second result. We note that we have the following expression:
\begin{align*}
    \E{z^TBzz^TCzz^TDz}=\sum_{i,j,k,\ell, m,r} e_i^TBe_je_k^TCe_\ell e_m^TDe_r\E{z_iz_jz_kz_\ell z_mz_r}
\end{align*}

We use Isserlis' Theorem on $\E{z_iz_jz_kz_\ell z_mz_r}$ and note we have the following 15 cases: (1) $i=j$, $k=\ell$ and $m=r$, (2) $i=j$, $k=m$, and $\ell =r$, (3) $i=j$, $k=r$, and $\ell=m$, (4) $i=k$, $j=m$, $\ell=r$, (5) $i=k$, $j=m$, and $\ell=r$, (6) $i=k$, $j=r$, and $\ell=m$, (7) $i=\ell$, $j=k$, and $m=r$, (8) $i=\ell$, $j=r$, and $k=m$, (9) $i=\ell$, $j=m$, and $k=r$, (10) $i=\ell$, $j=r$, and $k=m$, (11) $i=m$, $j=\ell$, and $k=r$, (12) $i=m$, $j=r$, and $k=\ell$, (13) $i=r$, $j=k$, and $\ell=m$, (14) $i=r$, $j=\ell$, and $k=m$, and (15) $i=r$, $j=m$, and $k=\ell$.

We note the first case results in the sum being $tr(B)tr(C)tr(D)$.

We have six self pairings (two self pairings for $B$, $C$, $D$). A self pairing for $B$ occurs if $i=j$, a self pairing for $C$ occurs if $k=\ell$, and a self pairing for $D$ occurs if $m=r$. It is easy to see that when a self pairing occurs, the sum becomes the trace of the self paired matrix times the trace of the product of the remaining two matrices---e.g. $tr(B)tr(CD)$.

The remaining eight cases we have no self pairings. In this case, it is easy to see that by reordering the indices (WLOG), we obtain the following:
\begin{align*}
    \sum_{i',j',k'}B_{i'j'}C_{j'k'}D_{k'i'}=tr(BCD)
\end{align*}

As such, the result follows.

\end{proof}    
\begin{proposition}
    Let $z\in\mathbb{R}^k$ be a random variable with i.i.d. entries such that we have $E[z_1^4]\leq C<\infty$ and $\E{z_1}=0$. Further, let $A\in\mathbb{R}^{k\times k}$ be a fixed matrix. Then, 
    \begin{align*}
        \E{(z^TAz)^2}\leq Ctr(A)^2+2Ctr(A^2)
    \end{align*}
    
\end{proposition}
\begin{proof}
Consider $\E{z_iz_jz_kz_\ell}$. We begin by noting that,
\begin{enumerate}
    \item when $i=j=k=\ell$, $\E{z_iz_jz_kz_\ell}=\E{z_1^4}\leq C$.
    \item When $i=j$, $\ell=k$, and $i\neq k$, $\E{z_iz_jz_kz_\ell}=\E{z_1^2}^2\leq C$.
    \item When $i=k$, $j=\ell$, and $i\neq j$, $\E{z_iz_jz_kz_\ell}=\E{z_1^2}^2\leq C$.
    \item When $i=\ell$, $j=k$, and $i\neq j$,
    $\E{z_iz_jz_kz_\ell}=\E{z_1^2}^2\leq C$.
    \item Otherwise, $\E{z_iz_jz_kz_\ell}=0$.
\end{enumerate}
We note the following:
\begin{align*}
\E{(z^TAz)^2}&=\sum_{i,j,k,\ell}A_{ij}A_{k\ell}\E{z_iz_jz_kz_\ell}\\&\leq 2\sum_{i}A_{ii}^2\E{z_1^4}+\sum_{\stackrel{i=j,k=\ell}{i\neq k}}A_{ii}A_{kk}\E{z_1^2}^2+\sum_{\stackrel{i=k,\ell=j}{i\neq j}}A_{ij}^2\E{z_1^2}^2+\sum_{\stackrel{i=\ell,j=k}{i\neq j}}A_{ij}^2\E{z_1^2}^2
\\&\leq\sum_{\stackrel{i=j,k=\ell}{i\neq k}}A_{ii}A_{kk}\E{z_i^2z_k^2}+\sum_{i=k,\ell=j}A_{ij}^2\E{z_i^2z_j^2}+\sum_{i=\ell,j=k}A_{ij}^2\E{z_i^2z_j^2}\leq Ctr(A)^2+2Ctr(A^2)
\end{align*}

Where the last inequality follows from,
\begin{align*}
\sum_{i=j,k=\ell}A_{ii}A_{kk}\E{z_i^2z_k^2}= \sum_{i=j,k=\ell}A_{ii}A_{kk}\E{z_i^2}^2-\sum_{i}A_{ii}^2\E{z_i^2}^2=tr(A)^2\E{z_i^2} -\sum_{i}A_{ii}^2\E{z_i^2}^2\leq Ctr(A)^2  
\end{align*}
\begin{align*}
    \E{(z^TAz)^2}&=\sum_{i,j,k,\ell}A_{ij}A_{k\ell}\E{z_iz_jz_kz_\ell}\leq 3C\sum_{i}A_{ii}^2+C\sum_{\stackrel{i=j,k=\ell}{i\neq k}}A_{ii}A_{kk}+C\sum_{\stackrel{i=k,\ell=j}{i\neq j}}A_{ij}^2+C\sum_{\stackrel{i=\ell,j=k}{i\neq j}}A_{ij}^2
    \\&\leq Ctr(A)^2+2Ctr(A^2)
\end{align*}
\end{proof}
\begin{proposition}\label{prop:prod_quads_bound_constant}
Let $z\sim N(0,I_n)$. Let $B_1,...,B_s\in\R^{n\times n}$ and each $B_i$ is symmetric positive semi-definite. 
Then,
\begin{align*}
    \E{\Pi_{k=1}^sz^TB_iz}=\Theta(1)\Pi_{k=1}^str(B_i)
\end{align*}
\end{proposition}
\begin{proof}

Let $v_i^{k}$ be the $i$-th eigenvector for matrix $B_k$ and $\lambda_i^{k}$ be the corresponding eigenvalue (this is well defined by diagonalizability of $B_k$). Let $x_i^k=z^Tv_i^k$.

As a result,
\begin{align*}
    \E{\Pi_{i=1}^kz^TB_iz}&=\E{\Pi_{k=1}^s\sum_{i_k=1}^n\left(x^k_{i_k}\right)^2\lambda_{i_k}^k}=\E{\sum_{(i_1,...,i_s)}\Pi_{k=1}^s\left(x^k_{i_k}\right)^2\lambda_{i_k}^k}
    \\&\leq \E{\sum_{(i_1,...,i_s)}\Pi_{k=1}^s\left(x^k_{i_k}\right)^2\lambda_{i_k}^k}\leq C\sum_{(i_1,...,i_s)}\Pi_{k=1}^s\lambda_{i_k}^k=C\Pi_{k=1}^s\sum_{i_k}\lambda_{i_k}^k=C\Pi_{k=1}^str(B_k)
\end{align*}

The upper bound comes from noting that by Prop. \ref{prop:gen_cauchy_scwartz},
\begin{align*}
    \E{\Pi_{k=1}^s\left(x^k_{i_k}\right)^2}\leq \Pi_{k=1}^s\E{\left(x^k_{i_k}\right)^{2^{s+1}}}^{1/2^s}\leq C^s
\end{align*}

Where we note that $x^k_{i_k}\sim N(0,1)$, the finiteness of moments of normal distributions have finite moments, and: 
\begin{align*}
\E{\left(x^k_{i_k}\right)^4}=\E{(z^T(v_{i_k}^k)(v_{i_k}^k)^Tz)^2}\leq tr\left( (v_{i_k}^k)(v_{i_k}^k)^T\right)^2+2tr\left((v_{i_k}^k)(v_{i_k}^k)^T(v_{i_k}^k)(v_{i_k}^k)^T\right) \leq 3 
\end{align*}
\end{proof}
\subsection{Preparatory Results}
\begin{proposition}\label{prop:var_x_y}
    Suppose $x\sim N(0,\Sigma)$, $\eps$ is a random variable such that $\E{\eps \,|\, x}=0$ and $\E{\eps^2\,|\,x}\leq \sigma_{max}^2$ $x$-a.s. where and $\Sigma$ is symmetric p.d. matrix. Let $y=x^T\beta+\eps$ for some $\beta$. Further, let $z$ be a constant vector. Then,
    \begin{align*}
        \Var[z^Txy]&\leq z^T\Sigma z(\sigma_{max}^2+\beta^T\Sigma\beta)+z^T\Sigma\beta\beta^T\Sigma z
        \\\Var[z^Txx^Tz]&=2(z^T\Sigma z)^2
    \end{align*}
\end{proposition}
\begin{proof}
Let $(\lambda_1,v_1),...,(\lambda_d,v_d)$ be the eigen-decomposition of $\Sigma$, further $c_k:=\beta^Tv_k$.

We note that $\Var[z^Txy]=\E{(z^Txy)^2}-\E{z^Txy}^2=z^T\E{xx^Ty^2}z-z^T\E{xy}\E{xy}^Tz$.

We begin by calculating $v_i^T\E{xx^T\beta^Txx^T\beta}v_j$ where $i\neq j$, we note that $v_k^Tx\sim N(0,\lambda_k)$ and if $k\neq \ell$, $v_k^Tx\perp v_\ell^Tx$,
\begin{align*}
\E{v_i^Txx^Tv_j\beta^Txx^T\beta}&=\E{v_i^Txv_j^Tx\sum_{k,\ell}c_kc_\ell v_k^Txv_\ell^Tx}=2\E{(v_i^Tx)^2(v_j^Tx)^2}c_ic_j=2\lambda_i\lambda_jc_ic_j
\end{align*}

In the case where $i=j$, we have the following:
\begin{align*}
 \E{v_i^Txx^Tv_i\beta^Txx^T\beta}&=\E{(v_i^Tx)^2\sum_{k=1}^dc_k^2 v_k^Txv_k^Tx}=\lambda_i\E{\sum_{k\neq i}c_k^2(v_k^Tx)^2}+c_k^2\E{(v_i^Tx)^4}
 \\&= \lambda_i\beta^T\Sigma\beta+2c_i^2\lambda_i^2
\end{align*}

As a result, $\E{xx^T\beta^Txx^T\beta}=\Sigma\beta^T\Sigma\beta+\Sigma\beta\beta^T\Sigma$ and we have that

\begin{align*}
    \E{xx^Ty^2}=\E{xx^T\beta^Txx^T\beta+\eps ^2xx^T}\leq \Sigma(\sigma_{max}^2+\beta^T\Sigma\beta)+2\Sigma\beta\beta^T\Sigma
\end{align*}
Turning to the first moment squared, 
\begin{align*}
    \E{xy}\E{xy}^T=\Sigma\beta\beta^T\Sigma
\end{align*}

The first result, thus follows.

For the second result, we start off calculating $\E{xx^Tzz^Txx^T}$. Consider $i\neq j$ and let $\gamma_k:=z^Tv_k$.

\begin{align*}
    v_i^T\E{xx^Tzz^Txx^T}v_j=\E{v_i^Txv_j^Tx\sum_{k,\ell}\gamma_k\gamma_\ell v_k^Txv_\ell^Tx}=2\E{(v_i^Tx)^2(v_j^Tx)^2}\gamma_i\gamma_j=2\lambda_i\lambda_j\gamma_i\gamma_j
\end{align*}

When $i=j$, we obtain the following:
\begin{align*}
    v_i^T\E{xx^Tzz^Txx^T}v_i&=\E{(v_i^Tx)^2\sum_{k,\ell}\gamma_k\gamma_\ell v_k^Txv_\ell^Tx}=\E{(v_i^Tx)^2\sum_{k}\gamma_k^2 (v_k^Tx)^2}
    \\&=\sum_{i\neq k}\gamma_k^2\E{(v_i^Tx)^2(v_k^Tx)^2}+\E{(v_i^Tx)^4}\gamma_i^2
    \\&=\sum_{i\neq k}\gamma_k^2\lambda_i\lambda_k+3\lambda_i^2\gamma_i^2=\lambda_iz^T\Sigma z+2\lambda_i^2\gamma_i^2
\end{align*}

As a result,
\begin{align*}
    z^T\E{xx^Tzz^Txx^T}z&=\sum_{k,\ell}\gamma_k\gamma_\ell v_k^T\E{xx^Tzz^Txx^T}v_\ell=\sum_{k=1}^d\gamma_k^2\lambda_iz^T\Sigma z+2\sum_{k,\ell}\gamma_k^2\gamma_\ell^2 \lambda_i\lambda_j
    \\&=(z^T\Sigma z)^2+2 (z^T\Sigma z)^2
\end{align*}

The result follows from noting that, $\E{z^Txx^Tz}^2=(z^T\Sigma z)^2$.

\end{proof}
\begin{definition}
Let $A:=\sum_{i=1}^{d}\lambda_ia_ia_i^T$ where $a_i\sim N(0,I_n)$. We define the following for an index set $\alpha$,
\begin{align*}
A_{\alpha}=\sum_{i=1,i\notin \alpha}^d\lambda_ia_ia_i^T    
\end{align*}
\end{definition}
\begin{proposition}\label{prop:tr_A_bounds}
Let $A:=\sum_{i=1}^{d}\lambda_ia_ia_i^T$ where $a_i\sim N(0,I_n)$, $L:=\sum_i\lambda_i$, $k\in\mathbb{N}$. Assume $k+2<n<d-1$ and $n+1<L$. Further, assume that $\lambda_{max}\leq \frac{1}{k+1}\frac{d}{n}$. Then, 
\begin{enumerate}
    \item $\E{tr(A^{-k})}=\Theta\left(\frac{\E{tr(A^{-k+1})}}{L}\right)=\Theta(\frac{n}{L^k})$
    \item $\E{tr(A^{-k})}\leq \frac{n}{L^k}(1+o(1))+\frac{1}{(1-(k+2)\rho)^{k+1}}\frac{n^2\,r(n)}{L^k}$ where $r(n):=\frac{\sum_{i=1}^d\lambda_i^2}{L^2}$
    \item $\frac{n}{L}+\frac{n^2\,r(n)(1-\Theta(1)\rho)}{L}\leq \E{tr(A^{-1})}\leq\frac{n}{L}+\frac{n^2\,r(n)}{L(1-2\rho)^2}$ where $\rho:=\frac{n}{d}\lambda_{max}$
    \item $\frac{n}{L^2}(1+o(1))+\frac{2n^2\,r(n)}{L^2}\leq \E{tr(A^{-2})}\leq \frac{n}{L^2}(1+o(1))+\frac{2n^2\,r(n)}{L^2(1-3\rho)^3}$
\end{enumerate}
Where we note that each $o(1)$ is $O(\frac{1}{tr(\Sigma)}\lambda_1)$.
\end{proposition}
\begin{proof}
    Let $F(a_i):=A^{-k}a_i$. From Prop. \ref{prop:stein_satsified}, the regularity conditions of Stein's Lemma (Lemma \ref{lemma:steins}) are satisfied. We start off by noting,
    \begin{align*}
        \sum_{j=1}^n\frac{\partial F_j(a_i)}{\partial a_{ij}}&=\partial_{a_{ij}}e_j^TA^{-k}a_i=\sum_{j=1}^ne_j^TA^{-k}e_j+e_j^T\left(\partial_{a_{ij}}A^{-k}\right)a_i\\&=tr(A^{-k})-\lambda_i\left(\sum_{r=0}^{k-1}a_i^TA^{-k+r}a_itr(A^{-r-1})+k\,a_i^TA^{-k-1}a_i\right)
    \end{align*}
    Where we take $A^0=I_n$.
    
    Second we note that,
    \begin{align*}
        \sum_{i}\lambda_i\E{a_i^TF(a_i)}=\sum_i\lambda_i\E{a_i^TA^{-k}a_i}=\E{tr(A^{-k+1})}
    \end{align*}
    Now, by Stein's Lemma,
    \begin{equation}\label{equ:stein_lemma_expanded}
    \begin{aligned}
\sum_{i=1}^d\lambda_i\E{a_i^TF(a_i)}&=\sum_{i=1}^d\lambda_i\sum_{j=1}^n\E{\frac{\partial F_j(a_i)}{\partial a_{ij}}}\\&= \E{\sum_{i=1}^d\left(\lambda_i tr(A^{-k})-\lambda_i^2\left(\sum_{r=0}^{k-1}a_i^TA^{-k+r}a_itr(A^{-r-1})+k\,a_i^TA^{-k-1}a_i\right)\right)}\\&=L\E{tr(A^{-k})}-\lambda_{max}\E{\sum_{r=0}^{k-1}tr(A^{-r-1})\sum_{i=1}^dtr(A^{-k+r}\lambda_ia_ia_i^T)+k\sum_{i=1}^dtr(A^{-k-1}\lambda_ia_ia_i^T)}\\&=L\E{tr(A^{-k})}-\lambda_{max} \E{\sum_{r=0}^{k-1}tr(A^{-k+r+1})tr(A^{-r-1})+k\,tr(A^{-k})}\\&\geq L\E{tr(A^{-k})}-\lambda_{max}\E{tr(A^{-k})}(kn+k)
\end{aligned}
    \end{equation}
As a result, $\E{tr(A^{-k})}\leq \frac{\E{tr(A^{-k+1})}}{L-\lambda_{max}(kn+k)}=\frac{\E{tr(A^{-k+1})}}{L}\frac{1}{1-(k+1)\frac{n}{L}\lambda_{max}}$.

Noting that $\E{tr(A^{-1})}\leq \frac{n}{L}\frac{1}{1-2\frac{n}{L}\lambda_{max}}$, we can apply induction and obtain the following result:
\begin{align*}
\E{tr(A^{-k})}\leq \frac{n}{L^k}\left(\Pi_{r=1}^k\frac{1}{1-(r+1)\frac{n}{L}\lambda_{max}}\right)
\end{align*}
We note from Equ. \ref{equ:stein_lemma_expanded} and Prop. \ref{prop:cauchy_tr_tr_inverse}, $\E{tr(A^{-k})}\geq \frac{\E{tr(A^{-k+1})}}{L}\geq \frac{n}{L^k}$  

For the second result, we start back with Equ. \ref{equ:stein_lemma_expanded}.

We note that, $0\leq \sum_{i=1}^d\lambda_i^2\E{a_i^TA^{-k-1}a_i}\leq \lambda_{max}\E{tr(A^{-k})}=o(L)\E{tr(A^{-k})}$. Thus, we can focus on the middle set of terms. For some $0\leq r\leq k-1$, we have the following:
\begin{align*}
    \sum_{i=1}^d\lambda_i^2\E{a_i^TA^{-k+r}a_itr(A^{-r-1})}&\leq \sum_{i=1}^d\lambda_i^2\E{a_i^TA_i^{-k+r}a_itr(A_i^{-r-1})}\\&=
    \sum_{i=1}^d\lambda_i^2\E{tr(A_i^{-k+r})tr(A_i^{-r-1})}\leq \sum_{i=1}^d\lambda_i^2n\E{tr(A_i^{-k-1})}\\&\leq \frac{1}{(1-(k+2)\rho)^{k+1}}\frac{n}{(L-o(1))^{k+1}}n\sum_{i=1}^d\lambda_i^2\\&=\frac{1}{(1-(k+2)\rho)^{k+1}}\frac{n}{L^{k-1}}r(n)
\end{align*}

Where we note that each $o(1)$ is $O(\frac{1}{tr(\Sigma)}\lambda_1)$. As such, $$\E{tr(A^{-k})}\leq \frac{n/L^{k-1}+O(\frac{n^2}{L^{k-1}})r(n)}{L}=\frac{n}{L^k}+O(n^2\frac{r(n)}{L^k})$$.

Turning to the third statement, we start off with the $k=1$ case. First, the upper bound is obtained as follows:
\begin{align*}
\sum_{i=1}^d\lambda_i^2\E{a_i^TA^{-1}a_itr(A^{-1})}=\sum_{i=1}^d\lambda_i^2\E{a_i^TA_i^{-1}a_itr(A_i^{-1})}\leq \frac{n^2r(n)}{(1-2\rho)^2}(1+o(1))
\end{align*}

The lower bound comes directly from the following:
\begin{align*}
    \E{a_i^TA^{-1}a_itr(A^{-1})}&=\E{\left(a_i^TA_i^{-1}a_i-\lambda_i\frac{(a_i^TA_i^{-1}a_i)^2}{1+a_i^TA_i^{-1}a_i}\right)tr\left(A_i^{-1}-\lambda_i\frac{A_i^{-1}a_ia_i^TA_i^{-1}}{1+\lambda_ia_i^TA_i^{-1}a_i}\right)}\\&\geq
    \E{a_i^TA_i^{-1}a_itr(A_i^{-1})-\lambda_ia_i^TA_i^{-2}a_ia_i^TA_i^{-1}a_i-\lambda_i(a_i^TA_i^{-1}a_i)^2tr(A_i^{-1})}\\&=\E{tr(A_i^{-1})^2}-\lambda_i\E{2tr(A_i^{-3})+(tr(A_i^{-2})+tr(A_i^{-1})^2+2tr(A_i^{-2}))tr(A_i^{-1})}\\&\geq \frac{n^2}{L^2}(1+o(1))-\lambda_i\Theta\left(\frac{n^3}{L^3}\right)(1+o(1))=\frac{n^2}{L^2}(1-\Theta(1)\rho)(1+o(1))
\end{align*}

When $k=2$, we obtain the upper bound result from the following bounds:
\begin{align*}
\sum_{i=1}^d\lambda_i^2\E{a_i^TA^{-2}a_itr(A^{-1})}&\leq \sum_{i=1}^d\lambda_i^2\E{a_i^TA_i^{-2}a_itr(A_i^{-1})}\leq \sum_{i=1}^d\lambda_i^2\E{tr(A_i^{-1})tr(A_i^{-2})}\\&\leq \frac{n^2r(n)}{L(1-4\rho)^3}(1+o(1))    \\
\sum_{i=1}^d\lambda_i^2 \E{a_i^TA^{-1}a_itr(A^{-2})}&\leq \sum_{i=1}^d\lambda_i^2\E{a_i^TA_i^{-1}a_itr(A_i^{-2})}\leq \sum_{i=1}^d\lambda_i^2\E{tr(A_i^{-1})tr(A_i^{-2})}\\&\leq \frac{n^2r(n)}{L(1-4\rho)^3}(1+o(1))
\end{align*}

For the lower bound, we note the following (where $\Delta_\ell:=a_i^TA_i^{-\ell}a_i$):
\begin{align*}
    \E{a_i^TA^{-2}a_itr(A^{-1})}&=\E{\left(\Delta_2-\lambda_i\frac{2\Delta_2\Delta_1}{1+\lambda_i\Delta_1}+\lambda_i^2\frac{\Delta_2\Delta_1^2}{(1+\lambda_i\Delta_1)^2}\right)\left(tr(A_i^{-1})-\lambda_i\frac{\Delta_2}{1+\lambda_i\Delta_1}\right)}\\&\geq\E{\Delta_2 tr(A_i^{-1})-2\lambda_i\Delta_2\Delta_1tr(A_i^{-1})-\lambda_i\Delta_2^2-\lambda_i^3\Delta_2^2\Delta_1^2}\\&
    \geq \E{tr(A_i^{-2})tr(A_i^{-1})-\Omega\left(\lambda_intr(A_i^{-3})tr(A_i^{-1})-\lambda_intr(A_i^{-4})-\lambda_i^3n^3tr(A_i^{-6})\right)}
    \\&\geq \frac{n^2}{L^3}(1+o(1))-\Omega\left(\lambda_i\frac{n^3}{L^4}+\lambda_i^3\frac{n^4}{L^6}\right)=\frac{n^2}{L^3}(1+o(1))
\end{align*}
\begin{align*}
\E{a_i^TA^{-1}a_itr(A^{-2})}&=\E{\left(\Delta_1-\lambda_i\frac{\Delta_1^2}{1+\lambda_i\Delta_1}\right)\left(tr(A_i^{-2})-2\lambda_i\frac{\Delta_3\Delta_1}{1+\lambda_i\Delta_1}+\lambda_i^2\frac{\Delta_2^2}{(1+\lambda_i\Delta_2)^2}\right)}\\&\geq \E{\Delta_1tr(A_i^{-2})-2\lambda_i\frac{\Delta_3\Delta_1^2}{1+\lambda_i\Delta_1}-\lambda_i\frac{\Delta_1^2tr(A_i^{-2})}{1+\lambda_i\Delta_1}-\lambda_i^3\frac{\Delta_2^4}{(1+\lambda_i\Delta_2)^3}}\\&\geq \frac{n^2}{L^3}(1+o(1))-\Omega\left(\lambda_i\frac{n^3}{L^5}+\lambda_i\frac{n^3}{L^4}+\lambda_i^3\frac{n^4}{L^8}\right)=\frac{n^2}{L^3}(1+o(1))
\end{align*}
Where we note that $\E{tr(A_i^{-1})tr(A_i^{-2})}\geq \E{tr(A_i^{-1})tr(A_i^{-1})^2/n}\geq \frac{n^2}{L^3}(1+o(1))$
\end{proof}
\begin{proposition}\label{prop:proj_expect}
    Let $X\in\mathbb{R}^{n\times d}$ (where $1<n<d-4$) be a matrix such that each row is distributed i.i.d. $N(0,\Sigma)$ where $\Sigma$ is a covariance which has eigenvalues such that $\lambda_1\leq \frac{1}{4}\frac{d}{n}$ and $tr(\Sigma)>n+1$. Then, $\E{\Pi_X}=\Theta\left(\frac{n}{tr(\Sigma}\right)\Sigma$ and in particular,
    \begin{align*}
\frac{n}{tr(\Sigma)}\Sigma(1+o(1))+\frac{n^2\,r(n)}{tr(\Sigma)}\Sigma-\frac{n^2}{tr(\Sigma)^2}\Sigma^2-2\frac{n^3r(n)}{tr(\Sigma)^2}\preceq \E{\Pi_X}
\end{align*}
\begin{align*}
\E{\Pi_X}\preceq  \frac{n}{tr(\Sigma)}\Sigma(1+o(1))+\frac{n^2\,r(n)}{tr(\Sigma)(1-2\rho)^2}\Sigma -  \frac{n^2(1+nr(n)(1-\Theta(1))^2}{tr(\Sigma)^2(1+\frac{n}{d}\lambda_1)}\Sigma^2
    \end{align*}
Where the $o(1)$ term in the second expression is positive and $O(\frac{1}{n})$.
    
\end{proposition}
\begin{proof}
We start off by noting that by Prop. \ref{prop:norm_decomp}, $X=\sum_{i=1}^da_i\sqrt{\lambda_i}v_i^T$ where $a_i\sim N(0,I_n)$. As such,
\begin{align*}
    XX^T&=(\sum_{i=1}^da_i\sqrt{\lambda_i}v_i^T)(\sum_{i=1}^dv_ia_i^T\sqrt{\lambda_i})=\sum_{i=1}^d\lambda_ia_ia_i^T=:A\\
    \Pi_X&=X^T(XX^T)^{-1}X=\sum_{ij}\sqrt{\lambda_i\lambda_j}a_i^TA^{-1}a_jv_iv_j^T
\end{align*}
As, $a_i\stackrel{d}{=}-a_i$ and $A$ is an even function for any $a_i$, $\E{a_i^TA^{-1}a_j}=0$ when $i\neq j$.  

Thus, $\E{\Pi_X}=\sum_{i=1}^d\lambda_i\E{a_i^TA^{-1}a_i}v_iv_i^T$. 

Now we note that,
\begin{align*}
a_i^TA^{-1}a_i=a_i^T\left(A_i^{-1}-\lambda_i\frac{A_i^{-1}a_ia_i^TA_i^{-1}}{1+\lambda_ia_i^TA_i^{-1}a_i}\right)a_i=a_i^TA_i^{-1}a_i-\lambda_i\frac{(a_i^TA_i^{-1}a_i)^2}{1+\lambda_ia_iA_i^{-1}a_i}    
\end{align*}

For our lower bound,
\begin{align*}
    \E{a_i^TA^{-1}a_i}\geq \E{a_i^TA^{-1}_ia_i}-\lambda_i\E{(a_i^TA^{-1}_ia_i)^2}
\end{align*}

For our upper bound, we use Jensen's inequality on $a_i^TA_i^{-1}a_i$ and noting that $\frac{z^2}{1+\lambda z}$ is convex:
\begin{align*}
\E{a_i^TA^{-1}a_i}&\leq \E{a_i^TA_i^{-1}a_i}  -\frac{\lambda_i}{1+2\frac{n}{tr(\Sigma)}\lambda_i}\E{(a_i^TA_i^{-1}a_i)^2} 
\end{align*}

As a result, we have the following upper and lower bounds on $\E{a_i^TA^{-1}a_i}$ from which the result follows:
\begin{align*}
    \E{a_i^TA^{-1}a_i}&\leq \E{tr(A_i^{-1})}\leq \frac{n}{tr(\Sigma)}(1+\frac{2}{n})+\frac{n\,r(n)}{tr(\Sigma)(1-2\rho)^2}-\frac{n^2(1+nr(n)(1-\Theta(1)\rho)^2}{tr(\Sigma)^2(1+\rho)}\lambda_i^2\\
    \E{a_i^TA^{-1}a_i}&\geq \frac{n}{tr(\Sigma)}(1+o(1))+\frac{n\,r(n)}{tr(\Sigma)}-\lambda_i\frac{n^2}{tr(\Sigma)^2}-2\frac{n^3r(n)}{tr(\Sigma)^2}\lambda_i^2
\end{align*}

\end{proof}
\begin{proposition}
Consider the same assumptions and notation as in Prop. \ref{prop:proj_expect}. Then,
\begin{align*}
   \Var[\beta^T\Pi_X\Sigma\eps]=\Theta\left(\frac{1}{n}(1+\sigma_{max}^2)\right)
\end{align*}
\end{proposition}
\begin{proof}
Let $c_k:=\beta^Tv_k$ (i.e. consider the eigenvectors of $\Sigma$ as a basis in $\R^d$). We begin by noting that, $\Sigma\beta\beta^T\Sigma=\sum_{k,l}c_kc_l\lambda_k\lambda_lv_kv_l^T$. As a result for $i\neq j$,
\begin{align*}
    \beta^T\Pi_X\Sigma\beta\beta^T\Sigma\Pi_X\beta=\sum_{k,l}\sqrt{\lambda_i\lambda_j\lambda_k\lambda_l}\lambda_k\lambda_lc_ic_jc_kc_la_i^TA^{-1}a_ka_l^TA^{-1}a_j
\end{align*}

We have the following options that are non-zero, when $i=k$ and $\ell=j$ or when $i=\ell$ and $k=j$. 

For the first case,
\begin{align*}
    \sum_{i,j}c_i^2c_j^2\lambda_i^2\lambda_j^2\E{a_i^TA^{-1}a_ia_j^TA^{-1}a_j}&\leq \sum_{i,j}c_i^2c_j^2\lambda_i^2\lambda_j^2\E{a_i^TA_{ij}^{-1}a_ia_j^TA_{ij}^{-1}a_j}
    \\&\leq \sum_{i,j}c_i^2c_j^2\lambda_i^2\lambda_j^2\E{tr(A_{ij}^{-1})^2}
    \\&\leq \Theta(1)\sum_{i,j}c_i^2c_j^2\lambda_i^2\lambda_j^2n\E{tr(A_{ij}^{-2})}\leq \Theta\left(\frac{n^2}{d^2}\right)(\beta^T\Sigma^2\beta)^2
\end{align*}

For the second case,
\begin{align*}
    \sum_{i,j}c_i^2c_j^2\lambda_i^2\lambda_j^2\E{a_i^TA^{-1}a_ja_j^TA^{-1}a_i}\leq \sum_{i,j}c_i^2c_j^2\lambda_i^2\lambda_j^2\E{a_i^TA_{ij}^{-2}a_i}\leq \Theta\left(\frac{n}{d^2}\right)(\beta^T\Sigma^2\beta)^2
\end{align*}

\end{proof}

\begin{proposition}\label{prop:var_Sig_proj}
Consider the same assumptions and notation as in Prop. \ref{prop:proj_expect}. Then, $\Var[\beta^T\Pi_X\Sigma\beta]=\Theta\left(\frac{n^3\rho r(n)}{tr(\Sigma)^2}\right)(\beta^T\Sigma^2\beta)^2+\Theta\left(\frac{n}{tr(\Sigma)^2}\right)\beta^T\Sigma\beta\beta^T\Sigma^3\beta+O\left(\frac{n^3}{tr(\Sigma)^3}\right)\beta^T\Sigma^3\beta\beta^T\Sigma^2\beta$    
\end{proposition}
\begin{proof}
 Let $c_k:=\beta^Tv_k$ (i.e. consider the eigenvectors of $\Sigma$ as a basis in $\R^d$). We begin by noting that, $\Sigma\beta\beta^T\Sigma=\sum_{k,l}c_kc_l\lambda_k\lambda_lv_kv_l^T$. As a result for $i\neq j$,
\begin{align*}
    v_i^T\Pi_X\Sigma\beta\beta^T\Sigma\Pi_Xv_j=\sqrt{\lambda_i\lambda_j}\sum_{k,l}\sqrt{\lambda_k\lambda_l}\lambda_k\lambda_lc_kc_la_i^TA^{-1}a_ka_l^TA^{-1}a_j
\end{align*}
The terms of the sum are mean zero unless $k=i$ and $l=j$ (or visa-versa). 
We note the following fact (and assuming $i\neq k$),
\begin{equation}\label{equ:offset_index_aI}
\begin{aligned}
a_i^TA^{-1}a_ka_k^TA^{-1}a_i&=\left(a_i^T\left(A_i^{-1}-\lambda_i\frac{A_i^{-1}a_ia_i^TA_i^{-1}}{1+\lambda_ia_i^TA_i^{-1}a_i}\right)a_k\right)^2\\&=\left(a_i^TA_i^{-1}a_k\left(1-\lambda_i\frac{a_i^TA_i^{-1}a_i}{1+\lambda_ia_i^TA_i^{-1}a_i}\right)\right)^2\leq (a_i^TA_i^{-1}a_k)^2    
\end{aligned}
\end{equation}

Thus, 
\begin{align*}
\E{v_i^T\Pi_X\Sigma\beta\beta^T\Sigma\Pi_Xv_j}&=c_ic_j\lambda_i^2\lambda_j^2\left(\E{a_i^TA^{-1}a_ia_j^TA^{-1}a_j}+\E{a_i^TA^{-1}a_ja_i^TA^{-1}a_j}\right)\\&\leq c_ic_j\lambda_i^2\lambda_j^2(\E{a_i^TA_j^{-1}a_ia_j^TA_j^{-1}a_j}+\E{a_i^TA_j^{-1}a_ja_j^TA_j^{-1}a_i})\\&\leq c_ic_j\lambda_i^2\lambda_j^2\E{a_i^TA_j^{-2}a_i}+c_ic_j\lambda_i^2\lambda_j^2\E{a_i^TA_{ij}^{-1}a_ia_j^TA_{ij}^{-1}a_j}
\\&\leq c_ic_j\lambda_i^2\lambda_j^2\Theta\left(\frac{n}{tr(\Sigma)^2}\right)+c_ic_j\lambda_i^2\lambda_j^2\left(\frac{n^2(1+2r(n)/(1-2\rho)^2)}{tr(\Sigma)^2}\right)(1+o(1))
\end{align*}

When $i=j$,
\begin{align*}
\E{v_i^T\Pi_X\Sigma\beta\beta^T\Sigma\Pi_Xv_i}&=\lambda_i\sum_{k,l}c_kc_l\sqrt{\lambda_k\lambda_l}\lambda_k\lambda_l\E{a_i^TA^{-1}a_ka_l^TA^{-1}a_i}\\&=\lambda_i\sum_{k}c_k^2\lambda_k^3\E{a_i^TA^{-1}a_ka_k^TA^{-1}a_i}\\&\leq\lambda_i^4c_i^2\E{(a_i^TA_i^{-1}a_i)^2}+\lambda_i\sum_{k\neq i}c_k^2\lambda_k^3\E{a_i^TA^{-1}a_ka_k^TA^{-1}a_i}\\&\leq \lambda_i^4c_i^2\E{(a_i^TA_i^{-1}a_i)^2}+\lambda_i\sum_{k\neq i}c_k^2\lambda_k^3\E{a_i^TA_k^{-1}a_ka_k^TA_k^{-1}a_i}\\&\leq \lambda_i^4c_i^2\E{(a_i^TA_i^{-1}a_i)^2}+\lambda_i\sum_{k\neq i}c_k^2\lambda_k^3\E{a_i^TA_k^{-2}a_i}
\\&\leq \lambda_i^4c_i^2\E{(a_i^TA_i^{-1}a_i)^2}+\lambda_i\E{a_i^TA_{ik}^{-2}a_i}\sum_{k}c_k^2\lambda_k^3
\\&\leq \lambda_i^4c_i^2\frac{n^2(1+2nr(n)/(1-2\rho)^2))}{tr(\Sigma)^2}(1+o(1))+\lambda_i\Theta\left(\frac{n}{tr(\Sigma)^2}\right)\beta^T\Sigma^3\beta
\end{align*}
Where we used the fact that $\E{a_i^TA^{-1}a_ka_l^TA^{-1}a_i}=0$ unless $k=l$.

Using this, we can thus obtain the second moment of $\beta^T\Pi_X\Sigma\beta$,
\begin{align*}
    \mathbb{E}[\beta^T\Pi_X\Sigma&\beta\beta^T\Sigma\Pi_X\beta]=\sum_{i\neq j}c_i^2c_j^2\lambda_i^2\lambda_j^2\left(\Theta\left(\frac{n}{tr(\Sigma)^2}\right)+\left(\frac{n^2(1+2r(n)/(1-2\rho)^2)}{tr(\Sigma)^2}\right)(1+o(1))\right)
    \\&+\sum_i\lambda_i^4c_i^4\left(\frac{n^2(1+2nr(n)/(1-2\rho)^2)}{tr(\Sigma)^2}\right)(1+o(1))+\Theta\left(\frac{n}{tr(\Sigma)^2}\right)\beta^T\Sigma\beta\beta^T\Sigma^3\beta
    \\&\leq \left(\frac{n^2(1+2nr(n)/(1-2\rho)^2)}{tr(\Sigma)^2}\right)(1+o(1))\sum_{ij}\lambda_i^2\lambda_j^2c_i^2c_j^2+\Theta\left(\frac{n}{tr(\Sigma)^2}\right)\beta^T\Sigma\beta\beta^T\Sigma^3\beta
    \\&\leq \left(\frac{n^2(1+2nr(n)/(1-2\rho)^2)}{tr(\Sigma)^2}\right)(1+o(1))(\beta^T\Sigma^2\beta)^2+\Theta\left(\frac{n}{tr(\Sigma)^2}\right)\beta^T\Sigma\beta\beta^T\Sigma^3\beta
\end{align*}

Finally, we look at the first moment squared:
\begin{align*}
    \E{\beta^T\Pi_X\Sigma\beta}^2&=\left(\frac{n(1+nr(n))}{tr(\Sigma)}\beta^T\Sigma^2\beta(1+o(1))-\Omega\left(\frac{n^2}{tr(\Sigma)^2}\right)\beta^T\Sigma^3\beta\right)^2\\&\geq \frac{n^2(1+nr(n))^2}{tr(\Sigma)^2}(\beta^T\Sigma^2\beta)^2(1+o(1))-2\Omega\left(\frac{n^3}{tr(\Sigma)^3}\right)\beta^T\Sigma^3\beta\beta^T\Sigma^2\beta
\end{align*}

Before taking the difference, we note the following difference:

\begin{align*}
    (1+2nr(n)/(1-2\rho)^2)-(1+nr(n))^2&=1+2nr(n)/(1-2\rho)^2-1-2nr(n)-n^2r(n)^2
    \\&\leq 16n\rho\,r(n)-n^2r(n)^2
\end{align*}

Thus,
\begin{align*}
    \Var[\beta^T\Pi_X\Sigma\beta]=\Theta\left(\frac{n^3\rho r(n)}{tr(\Sigma)^2}\right)(\beta^T\Sigma^2\beta)^2+\Theta\left(\frac{n}{tr(\Sigma)^2}\right)\beta^T\Sigma\beta\beta^T\Sigma^3\beta+O\left(\frac{n^3}{tr(\Sigma)^3}\right)\beta^T\Sigma^3\beta\beta^T\Sigma^2\beta
\end{align*}
\end{proof}

\begin{proposition}\label{prop:exp_sig_exp}
Consider the same assumptions and notation as in Prop. \ref{prop:proj_expect}. Then, \begin{align*}
        (1+o(1))\frac{n^2(1+nr(n))^2}{tr(\Sigma)^2}\Sigma^3-2\frac{n^3}{tr(\Sigma)^3}\Sigma^4+n(1+o(1))r(n)(1+2nr(n))(1-4\rho)\Sigma\preceq \E{\Pi_X\Sigma\Pi_X}
\end{align*}
\begin{align*}
    \E{\Pi_X\Sigma\Pi_X}\preceq (1+o(1))\frac{n^2(1+nr(n)/(1-4\rho)^3)^2}{tr(\Sigma)^2}\Sigma^3+(1+o(1))nr(n)(1+nr(n)/(1-2\rho)^2)\Sigma   
\end{align*}
\end{proposition}
\begin{proof}
We note that, $\Pi_X\Sigma\Pi_X=\sum_{i,j,k} \lambda_j^2\sqrt{\lambda_i}\sqrt{\lambda_k} a_i^TA^{-1}a_ja_j^TA^{-1}a_kv_iv_k^T$.

When $i\neq k$, the terms are mean zero. 

As such (noting Equ. \ref{equ:offset_index_aI}),
\begin{align*}
\E{\Pi_X\Sigma\Pi_X}&=\sum_{i,j}\lambda_i\lambda_j^2\E{a_i^TA^{-1}a_ja_j^TA^{-1}a_i}v_iv_i^T\\&\leq \sum_{i=1}^d\lambda_i^3\E{a_i^TA_i^{-1}a_ia_i^TA_i^{-1}a_i}v_iv_i^T+\sum_{i\neq j}\lambda_i\lambda_j^2\E{a_i^TA_j^{-1}a_ja_j^TA_j^{-1}a_i}v_iv_i^T
\\&=\sum_{i\neq j}\lambda_i\lambda_j^2\E{a_i^TA_j^{-2}a_i}v_iv_i^T+\sum_{i=1}^d\lambda_i^3\E{(a_i^TA_i^{-1}a_i)^2}
\\&\leq (1+o(1))\frac{n(1+nr(n)/(1-2\rho)^2)}{tr(\Sigma)^2}\sum_{i,j} \lambda_i\lambda_j^2v_iv_i^T+\Theta\left(\frac{n^2}{tr(\Sigma)^2}\right)\sum_{i=1}^d\lambda_i^3v_iv_i^T
\\&=(1+o(1))\frac{n(1+nr(n)/(1-2\rho)^2)}{tr(\Sigma)^2} tr(\Sigma^2)\Sigma + \Theta\left(\frac{n^2}{tr(\Sigma)^2}\right)\Sigma^3
\end{align*}

Conversely,
\begin{align*}
    \E{(a_i^TA^{-1}a_i)^2}&=\E{\left(a_i^TA_i^{-1}a_i-\lambda_i\frac{(a_i^TA_i^{-1}a_i)^2}{1+\lambda_ia_i^TA_i^{-1}a_i}\right)^2}
    \\&\geq \E{(a_i^TA_i^{-1}a_i)^2-\lambda_i(a_i^TA_i^{-1}a_i)^3}
    \\&\geq (1+o(1))\frac{n^2(1+r(n))^2}{tr(\Sigma)^2}-2\lambda_i\frac{n^3}{tr(\Sigma)^3}
\end{align*}
And for $i\neq j$
\begin{align*}
    \E{a_i^TA^{-1}a_ja_j^TA^{-1}a_i}&=\E{(a_i^TA_i^{-1}a_i)^2\left(1-\frac{\lambda_ia_i^TA_i^{-1}a_i}{1+\lambda_ia_i^TA_i^{-1}a_i}\right)^2}
    \\&\geq \E{a_j^TA_i^{-2}a_j}-2\lambda_i\E{a_j^TA_i^{-2}a_jtr(A_i^{-1})+2a_j^TA_i^{-3}a_j}
    \\&\geq \E{a_j^TA_i^{-2}a_j}-2\lambda_i(n+2)\E{tr(A_{ij})^{-3}}
    \\&\geq \frac{n(1+2nr(n))}{tr(\Sigma)^2}(1+o(1))-4\lambda_i\frac{n^2}{tr(\Sigma)^3}
\end{align*}
Thus, putting it together,
\begin{align*}
    (1+o(1))\frac{n^2(1+r(n))^2}{tr(\Sigma)^2}\Sigma^3&-\Omega\left(\frac{n^3}{tr(\Sigma)^3}\right)\Sigma^4+(1+o(1))r(n)(1+2r(n))(1-4\rho)\Sigma\\&\preceq \E{\Pi_X\Sigma\Pi_X} \preceq \\
    \Theta\left(\frac{n^2}{tr(\Sigma)^2}\right)\Sigma^3&+(1+o(1))nr(n)(1+nr(n)/(1-2\rho)^2)\Sigma
\end{align*}

\end{proof}
\begin{proposition}\label{prop:exp_noise_first_order}
Consider the same assumptions and notation as in Prop. \ref{prop:proj_expect}. Further, assume $\eps=(\eps_1,...,\eps_d)$ where $\E{\eps_i\,|\,X_i}=0$ and $\E{\eps_i^2\,|\,X_i}\leq \sigma_{max}^2$ where $\sigma_{max}^2$ does not depend on $X$. Let $\sigma^2$ be as in Def. 2.6. of the main text. Then, 
\begin{align*}
 \E{\eps^T(XX^T)^{-1}X\Sigma X^T(XX^T)^{-1}\eps}&\leq \sigma^2r(n)n\left(1+2nr(n)/(1-4\rho)^2\right)(1+o(1))\\   
  \E{\eps^T(XX^T)^{-1}X\Sigma X^T(XX^T)^{-1}\eps}&\geq \sigma^2r(n)n(1+2nr(n))(1-4\rho)
\end{align*}
\end{proposition}
\begin{proof}

We begin the upper bound proof by noting that:
\begin{align*}
    \E{v_i^TX^T(XX^T)^{-2}Xv_i}&=\lambda_i\E{a_i^TA^{-2}a_i}\leq \lambda_i\E{a_i^TA_i^{-2}a_i}\\&=\lambda_i\E{tr(A_i^{-2})}\leq \lambda_i\frac{n(1+2r(n)n/(1-4\rho)^2)}{tr(\Sigma)^2}(1+o(1))
\end{align*}

Thus, by trace properties to obtain the following when $\eps \perp X$:
\begin{align*}
    \E{\eps^T(XX^T)^{-1}X\Sigma X^T(XX^T)^{-1}\eps}&=\sigma^2\E{tr(\Sigma X^T(XX^T)^{-2}X)}=\sigma^2\sum_{i=1}^d\lambda_i^2\E{v_i^TX^T(XX^T)^{-2}Xv_i}
    \\&\leq \sigma^2\sum_{i=1}^d\lambda_i^2\frac{n(1+2nr(n)/(1-4\rho)^2)}{tr(\Sigma)^2}(1+o(1))\\&=\sigma^2r(n)n(1+2r(n)n/(1-4\rho)^2)(1+o(1))
\end{align*}

We note that for general $\eps$:
\begin{align*}
    \E{\eps^T(XX^T)^{-1}X\Sigma X^T(XX^T)^{-1}\eps}&=\E{tr(\Sigma X^T(XX^T)^{-1}\Lambda(X)(XX^T)^{-1}X)}\\&=\sigma^2\E{tr(\Sigma X^T(XX^T)^{-2}X)}
\end{align*}

For the lower bound, 
\begin{align*}
 \E{v_i^TX^T(XX^T)^{-2}Xv_i}&=\lambda_i\E{a_i^TA^{-2}a_i}\geq \E{a_i^TA_i^{-2}a_i-2\lambda_ia_i^TA_i^{-2}a_ia_i^TA_i^{-1}a_i}\\&\geq \lambda_i\frac{n(1+2nr(n))}{tr(\Sigma)^2}(1+o(1))   -4\lambda_i^2\frac{n^2}{tr(\Sigma)^3}\\&\geq \lambda_i\frac{n(1+2nr(n))}{tr(\Sigma)^2}(1+o(1))(1-4\rho)
\end{align*}

The result follows. 
\end{proof}
\begin{proposition}\label{prop:var_noise_pseudo_inverse}
Consider the same assumptions and notation as in Prop. \ref{prop:exp_noise_first_order}. Then, $\Var[\eps^T(XX^T)^{-1}X\Sigma X^T(XX^T)^{-1}\eps]\leq \sigma_{max}^4\Theta(1)r(n)^2n^2\beta^T\Sigma\beta$
\end{proposition}
\begin{proof}

\begin{align*}
    \E{(\eps^T(XX^T)^{-1}X\Sigma X^T(XX^T)^{-1}\eps)^2}\leq \sigma_{max}^4(n+2)\E{tr\left(\left((XX^T)^{-1}X\Sigma X^T(XX^T)^{-1}\right)^2\right)}
\end{align*}

We can simply the trace expression as follows,
\begin{align*}
    tr\left(\left((XX^T)^{-1}X\Sigma X^T(XX^T)^{-1}\right)^2\right)&=tr\left((XX^T)^{-1}X\Sigma X^T(XX^T)^{-2}X\Sigma X^T(XX^T)^{-1}\right)
    \\&=tr(\Sigma X^T(XX^T)^{-2}X\Sigma X^T(XX^T)^{-2}X)
\end{align*}

We note that,
    \begin{align*}
        X^T(XX^T)^{-2}X\Sigma X^T(XX^T)^{-2}X=\sum_{ijk}\sqrt{\lambda_i\lambda_k}\lambda_j^2a_i^TA^{-2}a_ja_j^TA^{-2}a_kv_iv_k^T
    \end{align*}
    As all terms except when $k=i$ are mean zero, we have that, 
    \begin{align*}
    \E{X^T(XX^T)^{-2}X\Sigma X^T(XX^T)^{-2}X}=\sum_{ij}\lambda_i\lambda_j^2\E{a_i^TA^{-2}a_ja_j^TA^{-2}a_i}v_iv_i^T    
    \end{align*}

Now, we consider the two cases when $i=j$ and when $i\neq j$. Suppose $i=j$, then
\begin{align*}
    \E{a_i^TA^{-2}a_ia_i^TA^{-2}a_i}\leq \E{a_i^TA_i^{-2}a_ia_i^TA_i^{-2}a_i}\leq (n+2)\E{tr(A_i^{-4})}
\end{align*}

Suppose $i\neq j$. We begin by noting that:
\begin{align*}
    &a_i^TA^{-2}a_ja_j^TA^{-2}a_i=(a_i^TA^{-2}a_j)^2\\&=\left(a_i^T\left(A_i^{-2}-\lambda_i\frac{A_i^{-2}a_ia_i^TA_i^{-1}}{1+\lambda_ia_i^TA_i^{-1}a_i}-\lambda_i\frac{A_i^{-1}a_ia_i^TA_i^{-2}}{1+\lambda_ia_i^TA_i^{-1}a_i}+\lambda_i^2\frac{A_i^{-1}a_ia_i^TA^{-2}_ia_ia_i^TA_i^{-1}}{(1+\lambda_ia_i^TA_i^{-1}a_i)^2}\right)a_j\right)^2
    \\&=\left(a_i^TA_i^{-2}a_j-\lambda_i\frac{a_i^TA_{i}^{-2}a_ia_i^TA_i^{-1}a_j+a_i^TA_i^{-1}a_ia_i^TA_i^{-2}a_j}{1+\lambda_ia_i^TA_i^{-1}a_i}+\lambda_i^2\frac{a_i^TA_i^{-1}a_ia_i^TA_i^{-2}a_ia_i^TA_i^{-1}a_j}{(1+\lambda_ia_i^TA^{-1}_ia_i)^2}\right)^2
    \\&=\left(a_i^TA_i^{-2}a_j\left(1-\frac{\lambda_ia_i^TA_i^{-1}a_i}{1+\lambda_ia_i^TA_i^{-1}a_i}\right)-\frac{\lambda_i a_i^TA_i^{-1}a_ja_i^TA_i^{-2}a_i}{1+\lambda_ia_i^TA_i^{-1}a_i}\left(1-\frac{\lambda_ia_i^TA_i^{-1}a_i}{1+\lambda_ia_i^TA_i^{-1}a_i}\right)\right)^2
    \\&=\left(1-\frac{\lambda_ia_i^TA_i^{-1}a_i}{1+\lambda_ia_i^TA_i^{-1}a_i}\right)^2\left(a_i^TA_i^{-2}a_j-\frac{\lambda_i a_i^TA_i^{-1}a_ja_i^TA_i^{-2}a_i}{1+\lambda_ia_i^TA_i^{-1}a_i}\right)^2\leq \left(a_i^TA_i^{-2}a_j-\frac{\lambda_i a_i^TA_i^{-1}a_ja_i^TA_i^{-2}a_i}{1+\lambda_ia_i^TA_i^{-1}a_i}\right)^2
    \\&\leq a_i^TA_{i}^{-2}a_ja_j^TA_{i}^{-2}a_i+\frac{\lambda_i^2(a_i^TA_i^{-1}a_j)^2(a_i^TA_i^{-2}a_i)^2}{(1+\lambda_ia_i^TA_i^{-1}a_i)^2}
\end{align*}

Therefore,
\begin{align*}
    \E{a_i^TA^{-2}a_ja_j^TA^{-2}a_i}&\leq \E{a_i^TA_{i}^{-2}a_ja_j^TA_{i}^{-2}a_i+\frac{\lambda_i^2(a_i^TA_i^{-1}a_j)^2(a_i^TA_i^{-2}a_i)^2}{(1+\lambda_ia_i^TA_i^{-1}a_i)^2}}
    \\&\leq \E{a_i^TA_{ij}^{-2}a_ja_j^TA_{ij}^{-2}a_i}+\lambda_1^2\E{a_i^TA_{ij}^{-1}a_ja_j^TA_{ij}^{-1}(a_i^TA_{ij}^{-2}a_i)^2}
    \\&\leq \Theta(1)\E{tr(A_{ij}^{-4})}+\lambda_1^2\E{a_i^TA_{ij}^{-1}a_ja_j^TA_{ij}^{-1}a_i(a_i^TA_{ij}^{-2}a_i)^2}
\end{align*}

Now, we note that,
\begin{align*}
    (a_i^TA^{-2}_ia_i)^2&\leq(a_i^T\left(A_{ij}^{-2}-2\lambda_j\frac{A_{ij}^{-2}a_ja_j^TA_{ij}^{-1}}{1+\lambda_j a_j^TA_{ij}^{-1}a_j}+\lambda_j^2\frac{A_{ij}^{-1}a_ja_j^TA_{ij}^{-2}a_ja_j^TA_{ij}^{-1}}{(1+\lambda_ja_j^TA_{ij}^{-1}a_j)^2}\right)a_i)^2
    \\&\leq 2(a_i^TA_{ij}^{-2}a_i)^2+4\lambda_1^2(a_i^TA_{ij}^{-2}a_i)^2(a_j^TA_{ij}^{-1}a_i)^2+2\lambda_1^4(a_i^TA_{ij}^{-1}a_j)^4(a_j^TA_{ij}^{-2}a_j)^2\\
\end{align*}

Taking expectations yields:
\begin{align*}
\E{a_i^TA_{ij}^{-1}a_ja_j^TA^{-1}_{ij}a_i(a_i^TA_{ij}^{-2}a_i)^2}&=\Theta(1)\E{a_j^TA_{ij}^{-2}a_jtr(A_{ij}^{-2})^2}\leq \Theta(n^2)\E{tr(A_{ij}^{-6})}\\
    \E{a_i^TA_{ij}^{-1}a_ja_j^TA^{-1}_{ij}a_i(a_i^TA_{ij}^{-2}a_i)^2(a_j^TA_{ij}^{-2}a_j)^2}&\leq \Theta(1)\E{a_j^TA_{ij}^{-2}a_jtr(A_{ij}^{-2})^2(a_j^TA_{ij}^{-2}a_j)^2}\\&\leq \Theta(n^4)\E{tr(A_{ij}^{-10})}\\
    \E{a_i^TA_{ij}^{-1}a_ja_j^TA^{-1}_{ij}a_i(a_i^TA_{ij}^{-1}a_j)^4(a_j^TA_{ij}^{-2}a_j)^2}&= \E{(a_i^TA_{ij}^{-1}a_ja_j^TA^{-1}_{ij}a_i)^3(a_j^TA_{ij}^{-2}a_j)^2}
    \\&\leq \Theta(1)\E{(a_j^TA_{ij}^{-2}a_j)^5}\leq \Theta(n^4)\E{tr(A_{ij}^{-10})}
\end{align*}

As a result,
\begin{align}\label{equ:A_2_diff_index}
\E{a_i^TA^{-2}a_ja_j^TA^{-2}a_i}\leq \Theta(1)\E{tr(A_{ij}^{-4})} 
\end{align}

Combining the two cases yields:
\begin{align*}
    tr&\left(\Sigma\E{X^T(XX^T)^{-2}X\Sigma X^T(XX^T)^{-2}X}\right)=\sum_{ij}\lambda_i^2\lambda_j^2\E{a_i^TA^{-2}a_ja_j^TA^{-2}a_i}
    \\&\leq \sum_{ij}\lambda_i^2\lambda_j^2\E{tr(A_{ij}^{-4})}+\Theta(n)\sum_{i}\lambda_i^4\E{tr(A_{ij}^{-4})}
    \\&\leq \Theta(\frac{n}{tr(\Sigma)^4})\left[tr(\Sigma^2)^2+ntr(\Sigma^4)\right]\leq O(1)r(n)^2n^2
\end{align*}

Where we note that by Prop. \ref{prop:bounds_effective_ranks} and Prop. \ref{prop:tr(AB)_leq_ntr(A)tr(B)},
\begin{align}\label{equ:tr_sigma_4_normalized}
    \frac{tr(\Sigma^4)}{tr(\Sigma)^4}=\frac{tr(\Sigma^4)}{tr(\Sigma^2)^2}\frac{tr(\Sigma^2)^2}{tr(\Sigma)^4}=r(n)^2\frac{tr(\Sigma^4)}{tr(\Sigma^2)^2}\leq r(n)^2\frac{\lambda_1^2}{tr(\Sigma^2)}\leq \rho^2r(n)^2\frac{tr(\Sigma)^2}{n^2}\frac{1}{tr(\Sigma)^2/n}\leq \rho^2r(n)^2/n
\end{align}
The result then follows.

\end{proof}
\begin{proposition}\label{prop:var_noise_proj}
Consider the same assumptions and notation as in Prop. \ref{prop:exp_noise_first_order}. Then, $\Var[\beta^T\Pi_X\Sigma X^T(XX^T)^{-1}\eps]\leq\sigma_{max}^2\Theta(n^2r(n)^2)\beta^T\Sigma\beta$
\end{proposition}
\begin{proof}
We first note that the first moment is zero and so we need only to focus on the second moment.
\begin{align*}
\E{\beta^T\Pi_X\Sigma X^T(XX^T)^{-1}\eps\eps^T(XX^T)^{-1}X\Sigma\Pi_X\beta}\leq \sigma_{max}^2\E{\beta^T\Pi_X\Sigma X^T(XX^T)^{-2}X\Sigma\Pi_X\beta}
\end{align*}

Now, by the mean zero property and independence of the $a_i$, we need only to consider,
\begin{align*}
    v_i^T\E{\Pi_X\Sigma X^T(XX^T)^{-2}X\Sigma\Pi_X}v_i&=\sum_{k,\ell}\lambda_i\lambda_k^2\lambda_\ell^2\E{a_i^TA^{-1}a_ka_k^TA^{-2}a_\ell a_\ell^TA^{-1}a_i}\\&\leq \sum_{k,\ell}\lambda_i\lambda_k^2\lambda_\ell^2\E{(a_i^TA^{-1}a_k)^2(a_\ell^TA^{-1}a_i)^2}^{1/2}\E{a_k^TA^{-2}a_\ell a_\ell^TA^{-2}a_k}^{1/2}
\end{align*}

We address each of the two expectations. For the first term, we use Cauchy-Schwartz again,
\begin{align*}
    \E{(a_i^TA^{-1}a_k)^2(a_i^TA^{-1}a_\ell)^2}&\leq \E{(a_i^TA^{-1}a_ka_k^TA^{-1}a_i)^2}^{1/2}\E{(a_i^TA^{-1}a_\ell a_\ell^TA^{-1}a_i)^2}^{1/2}
    \\&\leq \E{(a_i^TA_{ik}^{-1}a_ka_k^TA_{ik}^{-1}a_i)^2}^{1/2}\E{(a_i^TA_{i\ell}^{-1}a_\ell a_\ell^TA_{i\ell}^{-1}a_i)^2}^{1/2}
    \\&\leq \E{(a_i^TA^{-2}_{ik}a_i)^2}^{1/2}\E{(a_i^TA^{-2}_{i\ell}a_i)^2}^{1/2}\leq \Theta(\frac{n^2}{tr(\Sigma)^4})
\end{align*}

For the second term. Suppose $k=\ell$,
\begin{align*}
    \E{a_k^TA^{-2}a_ka_k^TA^{-2}a_k}=\Theta(n)\E{tr(A_k^{-4})}=\Theta(\frac{n^2}{tr(\Sigma)^4})
\end{align*}

Now, suppose, $k\neq \ell$, by Equ. \ref{equ:A_2_diff_index},

\begin{align*}
    \E{a_k^TA^{-2}a_\ell a_\ell^TA^{-2}a_k}\leq \Theta(1)\E{tr(A_{k\ell}^{-4})}\leq \Theta(\frac{n}{tr(\Sigma)^4})
\end{align*}

Putting it all together,
\begin{align*}
    v_i^T\E{\Pi_X\Sigma X^T(XX^T)^{-2}X\Sigma\Pi_X}v_i\leq \Theta(\frac{n^2}{tr(\Sigma)^4})\sum_{k,\ell}\lambda_i\lambda_k^2\lambda_\ell^2=\lambda_i\Theta(\frac{n^2tr(\Sigma^2)^2}{tr(\Sigma)^4})=\lambda_i\Theta(n^2r(n)^2)
\end{align*}

\end{proof}

\begin{proposition}\label{prop:var_proj_sig_proj}
Consider the same assumptions and notation as in Prop. \ref{prop:proj_expect}. Then,
\begin{align*}
\Var[\beta^T\Pi_X\Sigma\Pi_X\beta]&=  \Theta(\frac{1}{tr(\Sigma)^4})(n^2tr(\Sigma^2)(\beta^T\Sigma\beta)^2+n^3tr(\Sigma^2)\beta^T\Sigma^3\beta\beta^T\Sigma\beta\\&+n^4(\beta^T\Sigma^3\beta)^2+n^4\beta^T\Sigma^5\beta\beta^T\Sigma\beta)   
\end{align*}    
\end{proposition}
\begin{proof}
We recall that, $\Pi_X\Sigma\Pi_X=\sum_{i,j,k} \lambda_j^2\sqrt{\lambda_i}\sqrt{\lambda_k} a_i^TA^{-1}a_ja_j^TA^{-1}a_kv_iv_k^T$

Thus, $\beta^T\Pi_X\Sigma\Pi_X\beta=\sum_{i,j,k} c_ic_k\lambda_j^2\sqrt{\lambda_i}\sqrt{\lambda_k} a_i^TA^{-1}a_ja_j^TA^{-1}a_k$.

As a result,
\begin{equation}\label{equ:beta_proj_sig_pro_beta_squared}
\begin{aligned}
\beta^T\Pi_X&\Sigma\Pi_X\beta\beta^T\Pi_X\Sigma\Pi_X\beta=
\\&\sum_{i,j,k,\ell,m,r} c_ic_kc_\ell c_r \lambda_j^2\sqrt{\lambda_i}\sqrt{\lambda_k} \lambda_m^2\sqrt{\lambda_\ell}\sqrt{\lambda_r}a_i^TA^{-1}a_ja_j^TA^{-1}a_ka_\ell^TA^{-1}a_ma_m^TA^{-1}a_r        
\end{aligned}
\end{equation}

We note by symmetry of the Gaussian random variates and their mean zero property, that either (i) $i=k$ and $\ell=r$, (ii) $i=\ell$ and $k=r$, and (iii) $i=r$ and $k=\ell$. By the i.i.d. nature of $a_i$, the cases (ii) and (iii) are the same.

For each case, we will handle all the corresponding sub-cases of indexes being equal.

We note that, if $i=\ell$, then for $i\neq j$ and $i\neq m\neq j$,
\begin{align*}
    &\sum_{i\neq j\neq m, i\neq m}c_i^4\lambda_i^2\lambda_j^2\lambda_m^2\E{a_i^TA^{-1}a_ja_j^TA^{-1}a_ia_i^TA^{-1}a_ma_m^TA^{-1}a_i}
    \\&\leq \sum_{i\neq j\neq m, i\neq m}c_i^4\lambda_i^2\lambda_j^2\lambda_m^2\E{a_j^TA_i^{-2}a_ja_m^TA_i^{-2}a_m}\leq \Theta(\frac{n^2tr(\Sigma^2)}{tr(\Sigma)^4})(\beta^T\Sigma\beta)^2
\end{align*}

For $i=j\neq m$
\begin{align*}
&\sum_{i\neq m}c_i^4\lambda_i^4\lambda_m^2\E{a_i^TA^{-1}a_ia_i^TA^{-1}a_ia_i^TA^{-1}a_ma_m^TA^{-1}a_i}\\&\leq 
\sum_{i,m}c_i^4\lambda_i^4\lambda_m^2\E{a_i^TA_i^{-1}a_ia_i^TA_i^{-1}a_ia_i^TA_i^{-1}a_ma_m^TA_i^{-1}a_i}
\\&\leq \Theta(1)\sum_{i,m}c_i^4\lambda_i^4\lambda_m^2\E{tr(A_i^{-1})^2a_m^TA_i^{-2}a_m}\leq \Theta(\frac{n^3}{tr(\Sigma)^4})\beta^T\Sigma^2\beta
\end{align*}

For $i=j=m$,
\begin{align*}
&\sum_{i}c_i^4\lambda_i^6\E{a_i^TA^{-1}a_ia_i^TA^{-1}a_ia_i^TA^{-1}a_ia_i^TA^{-1}a_i}\\&\leq \Theta(1)
\sum_{i}c_i^4\lambda_i^6\leq \Theta(\frac{n^4}{tr(\Sigma)^4})(\beta^T\Sigma^3\beta)^2
\end{align*}

For $i\neq j=m$,
\begin{align*}
&\sum_{i\neq j}c_i^4\lambda_i^2\lambda_j^4\E{a_i^TA^{-1}a_ja_j^TA^{-1}a_ia_i^TA^{-1}a_ja_j^TA^{-1}a_i}
\\&\leq \Theta(1)\sum_{i,j}c_i^4\lambda_i^2\lambda_j^4\E{(a_j^TA^{-2}a_j)^2}\leq \Theta(\frac{n^2}{tr(\Sigma)^4})tr(\Sigma^4)(\beta^T\Sigma\beta)^2 
\end{align*}

If $i\neq \ell$, then we have the following distinct cases:
For $\ell \neq j\neq i$, $i\neq m\neq \ell$ (we will denote this $(i,j,\ell,m)\in E$),
\begin{align*}
    &\sum_{(i,j,\ell,m)\in E}c_i^2c_\ell^2\lambda_j^2\lambda_m^2\lambda_i\lambda_\ell \E{a_i^TA^{-1}a_ja_j^TA^{-1}a_ia_\ell^TA^{-1}a_ma_m^TA^{-1}a_\ell}
    \\&\leq \sum_{(i,j,\ell,m)\in E}c_i^2c_\ell^2\lambda_j^2\lambda_m^2\lambda_i\lambda_\ell \E{(a_i^TA^{-1}a_ja_j^TA^{-1}a_i)^2}^{1/2}\E{(a_\ell^TA^{-1}a_ma_m^TA^{-1}a_\ell)^2}^{1/2}
    \\&\leq \sum_{(i,j,\ell,m)\in E}c_i^2c_\ell^2\lambda_j^2\lambda_m^2\lambda_i\lambda_\ell\Theta(\frac{n^2}{tr(\Sigma)^4})\leq \Theta(1)n^2r(n)^2(\beta^T\Sigma\beta)^2
\end{align*}

For $i=j$, $j\neq m$ and $\ell\neq m$,
\begin{align*}
&\sum_{i,\ell, m}c_i^2c_\ell^2\lambda_i^3\lambda_{\ell}\lambda_m^2\E{a_i^TA^{-1}a_ia_i^TA^{-1}a_ia_\ell^TA^{-1}a_ma_m^TA^{-1}a_\ell} \\&\leq \sum_{i,\ell, m}c_i^2c_\ell^2\lambda_i^3\lambda_{\ell}\lambda_m^2\E{(a_i^TA^{-1}a_ia_i^TA^{-1}a_i)^2}^{1/2}\E{(a_\ell^TA^{-1}a_ma_m^TA^{-1}a_\ell)^2}^{1/2}
\\&\leq \sum_{i,\ell, m}c_i^2c_\ell^2\lambda_i^3\lambda_{\ell}\lambda_m^2\E{(a_i^TA^{-1}a_ia_i^TA^{-1}a_i)^2}^{1/2}\E{(a_\ell^TA^{-1}a_ma_m^TA^{-1}a_\ell)^2}^{1/2}\\&\leq \Theta(\frac{n^3tr(\Sigma^2)}{tr(\Sigma)^4})\beta^T\Sigma^3\beta\beta^T\Sigma\beta
\end{align*}
For $i=j$ and $\ell=m$,
\begin{align*}
&\sum_{i,\ell}c_i^2c_\ell^2\lambda_i^3\lambda_{\ell}^3\E{a_i^TA^{-1}a_ia_i^TA^{-1}a_ia_\ell^TA^{-1}a_\ell a_\ell^TA^{-1}a_\ell} \\&\leq  \sum_{i,\ell}c_i^2c_\ell^2\lambda_i^3\lambda_{\ell}^3\E{tr(A_{i\ell}^{-1})^4}\leq \Theta(\frac{n^4}{tr(\Sigma)^4})(\beta^T\Sigma^3\beta)^2 
\end{align*}
For $i=j=m$,
\begin{align*}
&\sum_{i,\ell}c_i^2c_\ell^2\lambda_i^5\lambda_\ell\E{a_i^TA^{-1}a_ia_i^TA^{-1}a_ia_\ell^TA^{-1}a_ia_i^TA^{-1}a_\ell}
\\&\leq \sum_{i,\ell}c_i^2c_\ell^2\lambda_i^5\lambda_\ell\E{tr(A_{i\ell}^{-1})^2tr(A^{-2}_{i\ell})}\leq \Theta(\frac{n^4}{tr(\Sigma)^4})\beta^T\Sigma^5\beta\beta^T\Sigma\beta
\end{align*}

As a result.
\begin{align*}
&\sum_{i,j,m,\ell} c_i^2c_\ell^2\lambda_j^2\lambda_m^2\lambda_i\lambda_\ell\E{a_i^TA^{-1}a_ja_j^TA^{-1}a_ia_\ell^TA^{-1}a_ma_m^TA^{-1}a_\ell}\\&\leq \Theta(\frac{1}{tr(\Sigma)^4})(n^2tr(\Sigma^2)(\beta^T\Sigma\beta)^2+n^3tr(\Sigma^2)\beta^T\Sigma^3\beta\beta^T\Sigma\beta+n^4(\beta^T\Sigma^3\beta)^2+n^4\beta^T\Sigma^5\beta\beta^T\Sigma\beta)    
\end{align*}

For case (ii), we do the same procedure. We begin by noting that if $i=k$ or $j=m$, then we have already covered those cases.

For $i,k,j,m$ distinct case (denoted $(i,j,k,m)\in E$),
\begin{align*}
    &\sum_{(i,j,k,m)\in E}c_i^2c_k^2\lambda_i\lambda_k\lambda_j^2\lambda_m^2\E{a_i^TA^{-1}a_ja_j^TA^{-1}a_ka_k^TA^{-1}a_ja_j^TA^{-1}a_ma_m^TA^{-1}a_i}
    \\&\leq \sum_{(i,j,k,m)\in E}c_i^2c_k^2\lambda_i\lambda_k\lambda_j^2\lambda_m^2\E{(a_i^TA^{-1}a_ja_j^TA^{-1}a_i)^2}^{1/2}\E{(a_k^TA^{-1}a_ma_m^TA^{-1}a_k)^2}^{1/2}
    \\&\leq \sum_{(i,j,k,m)\in E}c_i^2c_k^2\lambda_i\lambda_k\lambda_j^2\lambda_m^2\E{(a_i^TA^{-1}a_ja_j^TA^{-1}a_i)^2}^{1/2}\E{(a_k^TA^{-1}a_ma_m^TA^{-1}a_k)^2}^{1/2}
    \\&\leq \sum_{(i,j,k,m)\in E}c_i^2c_k^2\lambda_i\lambda_k\lambda_j^2\lambda_m^2\Theta(\frac{n^2}{tr(\Sigma)^4})\leq \Theta(1)r(n)^2n^2(\beta^T\Sigma\beta)^2
\end{align*}

For the $i=j$ ($j\neq m$ and $i\neq k$) case which we will denote $(i,j,k,m)\in E$,
\begin{align*}
    &\sum_{(i,k,m)\in E}c_i^2c_k^2\lambda_i^3\lambda_k\lambda_m^2\E{a_i^TA^{-1}a_ia_i^TA^{-1}a_ka_k^TA^{-1}a_ma_m^TA^{-1}a_i}
    \\&\leq \sum_{(i,k,m)\in E}c_i^2c_k^2\lambda_i^3\lambda_k\lambda_m^2\E{(a_iA^{-1}a_i)^2a_i^TA^{-1}a_ka_k^TA^{-1}a_i}^{1/2}\E{a_m^TA^{-1}a_ka_k^TA^{-1}a_ma_i^TA^{-1}a_ka_k^TA^{-1}a_i}^{1/2}
 \\&\leq \sum_{(i,k,m)\in E}c_i^2c_k^2\lambda_i^3\lambda_k\lambda_m^2\E{tr(A_{ik}^{-1})^2tr(A^{-2}_{ik})}^{1/2}\E{a_m^TA_k^{-2}a_ma_i^TA_{k}^{-2}a_i}^{1/2}   
  \\&\leq \sum_{(i,k,m)\in E}c_i^2c_k^2\lambda_i^3\lambda_k\lambda_m^2\Theta(\frac{n^3}{tr(\Sigma)^4})\leq \Theta(\frac{n^3tr(\Sigma^2)}{tr(\Sigma)^4})\beta^T\Sigma^3\beta\beta^T\Sigma\beta 
\end{align*}

The conclusion follows.

\end{proof}

\begin{proposition}\label{prop:noise_nN_v_n_gen}
We assume Assumption \ref{assumpt:weak_canonical_case} holds. Let $X_s$ be a sub-sample of $X$ with size $k$. Assume $\sigma_{max}^2=\exp(o(n^{1/3}\rho))$. Finally, let,
\begin{align*}
    \tilde{\sigma}^2=\E{tr((X_sX_s^T)^{-1}X_s\Sigma X_s^T(X_sX_s^T)^{-1}\Lambda(X_s))}/\E{tr(\Sigma X_s^T(X_sX_s^T)^{-2}X_s)}
\end{align*}
Then, 
\begin{align*}
\tilde{\sigma}^2&\geq (1-6\rho^2(\frac{k}{n})^{1/3})\E{\frac{x_1^T\Sigma x_1}{\norm{x_1}^4}\E{\eps_1^2|x_1}}/(r(n)(1+kr(n))\\   
\tilde{\sigma}^2&\leq \frac{1+6\rho^2(\frac{k}{n})^{1/3}}{1-3\rho(\frac{k}{n})^{1/3}}\E{\frac{x_1^T\Sigma x_1}{\norm{x_1}^4}\E{\eps_1^2|x_1}}/(r(n)(1+3kr(n))+o(1)
\end{align*}
\end{proposition}
\begin{proof}
    Let $B(X_s)=(X_sX_s^T)^{-1}X_s\Sigma X_s^T(X_sX_s^T)^{-1}$. 

We begin by noting that,
\begin{align}\label{equ:tr_decomp_noise_diff}
    \E{tr(B\Lambda(X_s))}=\sum_{i=1}^{\text{rank}(X_s)} \E{B_{ii}(X_s)\E{\eps_i^2|X_i}}=\text{rank}(X_s)\E{B_{11}(X_s)\E{\eps_1^2|X_1}}
\end{align}
This is because any permutation of the rows of $X$ is equal in distribution. Let $k:=\text{rank}(X_s)$.

Now consider, the following block matrix:
$X_s=\begin{bmatrix}
    \tilde{X}\\
    \tilde{x}
\end{bmatrix}$.

By properties of block matrices we update the following expression:
\begin{align*}
    (X_s&X_s^T)^{-1}=\begin{bmatrix}
        \tilde{X}\tilde{X}^T & \tilde{X}\tilde{x}\\
        \tilde{x}X^T & \tilde{x}^T\tilde{x} 
    \end{bmatrix}^{-1}\\&=\begin{bmatrix}
        (\tilde{X}\tilde{X}^T)^{-1}+(\tilde{X}\tilde{X}^T)^{-1} \tilde{x}(\tilde{x}^T(I-\Pi_{\tilde{X}})\tilde{x})^{-1}\tilde{x}^T\tilde{X}^T(\tilde{X}\tilde{X}^T)^{-1}&-(\tilde{X}\tilde{X}^T)^{-1}\tilde{X}\tilde{x}(\tilde{x}^T(I-\Pi_X)\tilde{x})^{-1} \\
        -(\tilde{x}^T(I-\Pi_{\tilde{X}})\tilde{x})^{-1}\tilde{x}^T\tilde{X}^T(\tilde{X}\tilde{X}^T)^{-1}& (\tilde{x}^T(I-\Pi_{\tilde{X}})\tilde{x})^{-1}
    \end{bmatrix}
\end{align*}

As a result,
\begin{align*}
    e_k^TB_{11}(X_s)e_k=\frac{\tilde{x}^T(I-\Pi_{\tilde{X}})\Sigma(I-\Pi_{\tilde{X}})\tilde{x}}{(\tilde{x}^T(I-\Pi_{\tilde{X}})\tilde{x})^{2}}
\end{align*}

We note that we have the following objects to bound: $\frac{\tilde{x}^T\Sigma\tilde{x}}{(\tilde{x}^T(I-\Pi_{\tilde{x}})\tilde{x})^2}$, $\frac{\tilde{x}^T\Pi_{\tilde{X}}\Sigma\tilde{x}}{(\tilde{x}^T(I-\Pi_{\tilde{x}})\tilde{x})^2}$, and $\frac{\tilde{x}^T\Pi_{\tilde{X}}\Sigma\Pi_{\tilde{X}}\tilde{x}}{(\tilde{x}^T(I-\Pi_{\tilde{x}})\tilde{x})^2}$. We note our goal is to get terms of the form, $\E{\frac{\tilde{x}^T\Sigma\tilde{x}}{(\tilde{x}^T\tilde{x})^2}\E{\eps(\tilde{x})^2|\tilde{x}}}$.

Starting with the first term, for its lower bound we have that,
\begin{align*}
    \E{\frac{\tilde{x}^T\Sigma\tilde{x}}{(\tilde{x}^T(I-\Pi_{\tilde{x}})\tilde{x})^2}\E{\eps(\tilde{x})^2|\tilde{x}}}\geq\E{\frac{\tilde{x}^T\Sigma\tilde{x}}{(\tilde{x}^T\tilde{x})^2}\E{\eps(\tilde{x})^2|\tilde{x}}}
\end{align*}

Turning to the upper bound, we need to use Lemma \ref{lemma:hanson-wright}. So, let $E:=\{\tilde{x}^T\Pi_{\tilde{X}}\tilde{x}\geq 2\rho b_n\norm{\tilde{x}}^2\}$ where $b_n:=\left(\frac{k}{n}\right)^{1/3}$

We note that on the event $E^c$, we have that,
\begin{align*}
\E{\frac{\tilde{x}^T\Sigma\tilde{x}}{(\tilde{x}^T(I-\Pi_{\tilde{x}})\tilde{x})^2}\E{\eps(\tilde{x})^2|\tilde{x}}\mathbf{1}(E^c)}\leq \frac{1}{\left(1-\rho b_n\right)^2}\E{\frac{\tilde{x}^T\Sigma\tilde{x}}{\norm{\tilde{x}}^4}\E{\eps(\tilde{x})^2|\tilde{x}}}
\end{align*}

For the event, $E$, we use Cauchy-Schwartz and Prop. \ref{prop:general_chi_square_moment},
\begin{align*}
\E{\frac{\tilde{x}^T\Sigma\tilde{x}}{(\tilde{x}^T(I-\Pi_{\tilde{x}})\tilde{x})^2}\E{\eps(\tilde{x})^2|\tilde{x}}\mathbf{1}(E)}&\leq \sigma_{max}^2\E{\frac{(\tilde{x}^T\Sigma\tilde{x})^2}{(\tilde{x}^T(I-\Pi_{\tilde{X}})\tilde{x})^4}}^{1/2}\mathbb{P}(E)^{1/2} 
\\&\leq \sigma^2_{max}\E{(\tilde{x}^T\Sigma\tilde{x})^4}^{1/4}\E{\frac{1}{(\tilde{x}^T(I-\Pi_{\tilde{X}})\tilde{x})^8}}^{1/4}\mathbb{P}(E)^{1/2}
\\&\leq \Theta(1)\sigma_{max}^2\frac{tr(\Sigma^2)}{d^2}\mathbb{P}(E)^{1/2}=o(r(n))
\end{align*}

Where we note that $tr(\Sigma^{1/2}(I-\Pi_{\tilde{X}})\Sigma^{1/2})\geq d-\rho d\frac{k}{n}\geq (1-\rho b_n)d$. 

For the probability term, by Lemma \ref{lemma:hanson-wright}, we have the following
\begin{align*}
    \mathbb{P}\left(\tilde{x}^T\Pi_{\tilde{X}}\tilde{x}>2\rho b_n\tilde{x}^T\tilde{x}\right)&\leq \mathbb{P}\left(\tilde{x}^T\Pi_{\tilde{X}}\tilde{x}>1.5\rho b_n d\right)+\mathbb{P}\left(|\tilde{x}^T\tilde{x}-\E{\tilde{x}^T\tilde{x}}|>.5\rho b_nd\right)
    \\&\leq \mathbb{P}\left(|\tilde{x}^T\Pi_{\tilde{X}}\tilde{x}-\E{\tilde{x}^T\Pi_{\tilde{X}}\tilde{x}}|>.5\rho  b_nd\right)+ \mathbb{P}\left(|\tilde{x}^T\tilde{x}-\E{\tilde{x}^T\tilde{x}}|>.5\rho b_nd\right)\\&
    \leq 2\exp\left[-\eta\min\left(\frac{b_n^2\rho^2d^2}{4tr(\Pi_{\tilde{X}}\Sigma\Pi_{\tilde{X}}\Sigma)},\frac{b_n\rho d}{2\norm{\Sigma^{1/2}\Pi_{\tilde{X}}\Sigma^{1/2}}_2}\right)\right]\\&+2\exp\left[-\eta \min(\frac{b_n^2\rho^2d^2}{4tr(\Sigma^2)},\frac{b_n\rho d}{2\lambda_1})\right]
    \leq 4\exp\left[-\frac{\eta}{4}n^{1/3}\rho\right]
\end{align*}

Where we note that,
\begin{align*}
    \norm{\Sigma^{1/2}\Pi_{\tilde{X}}\Sigma^{1/2}}_2&\leq \norm{\Sigma^{1/2}}_2^2=\rho\frac{d}{n}\\
    tr((\Sigma^{1/2}\Pi_{\tilde{X}}\Sigma^{1/2})^2)
    &=tr(\Sigma\Pi_{\tilde{X}}\Sigma\Pi_{\tilde{X}})\leq tr(\Sigma^2)^{1/2}tr(\Pi_{\tilde{X}}\Sigma\Pi_{\tilde{X}}\Sigma)^{1/2}=tr(\Sigma^2)^{1/2}tr((\Sigma^{1/2}\Pi_{\tilde{X}}\Sigma^{1/2})^2)^{1/2}
\end{align*}

By the same reasoning as above, for the third term, we have,
\begin{align*}
    \E{\frac{\tilde{x}^T\Pi_{\tilde{X}}\Sigma\Pi_{\tilde{X}}\tilde{x}}{(\tilde{x}^T(I-\Pi_{\tilde{X}})\tilde{x})^2}\E{\eps(\tilde{x})^2|\tilde{x}}}&\leq \E{\frac{\tilde{x}^T\Pi_{\tilde{X}}\Sigma\Pi_{\tilde{X}}\tilde{x}}{\norm{\tilde{x}}^4}\E{\eps(\tilde{x})^2|\tilde{x}}}+o(r(n))\\&\leq 3\rho^2(\frac{k^2}{n^2}+\frac{k}{n})\E{\frac{\tilde{x}^T\Sigma\tilde{x}}{\norm{\tilde{x}}^4}\E{\eps(\tilde{x})^2|\tilde{x}}}+o(r(n))
\end{align*}

Turning to the second term, because its not guaranteed to always be positive, we need to do more work and we will just use the negative of its upper bound as its lower bound. We thus, consider $F:=\{\tilde{x}^T\Pi_{\tilde{X}}\Sigma\tilde{x}>3\rho^2 b_n\tilde{x}^T\Sigma\tilde{x}\}$.

We note that,
\begin{align*}
    \mathbb{P}(F)&\leq \mathbb{P}(|\tilde{x}^T\Pi_{\tilde{X}}\Sigma\tilde{x} -tr(\Pi_{\tilde{X}}\Sigma^2)|>3\rho^2 b_n\tilde{x}^T\Sigma\tilde{x}-tr(\Pi_{\tilde{X}}\Sigma^2))
    \\&\leq\mathbb{P}(|\tilde{x}^T\Pi_{\tilde{X}}\Sigma\tilde{x} -tr(\Pi_{\tilde{X}}\Sigma^2)|>3\rho^2 b_n\tilde{x}^T\Sigma\tilde{x}-\rho^2 b_n\frac{d^2}{n})
    \\&\leq \mathbb{P}(|\tilde{x}^T\Pi_{\tilde{X}}\Sigma\tilde{x} -tr(\Pi_{\tilde{X}}\Sigma^2)|>.5\rho^2 b_n\frac{d^2}{n})+\mathbb{P}(|\tilde{x}^T\Sigma\tilde{x}-\E{\tilde{x}^T\Sigma\tilde{x}}|>.5\rho^2 b_n\frac{d^2}{n})
\end{align*}

Now, because $\lambda_1(\Sigma^2)=\rho^2\frac{d^2}{n^2}$ and $tr(\Sigma^4)\leq \rho^3\frac{d^4}{n^3}$,
\begin{align*}
\mathbb{P}(|\tilde{x}^T\Sigma\tilde{x}-\E{\tilde{x}^T\Sigma\tilde{x}}|>.5\rho^2 b_n\frac{d^2}{n})&\leq 2\exp\left(-\eta\min(\frac{\rho^4d^4/n^2b_n^2}{4\rho^4d^4/n^3},\frac{\rho^2b_nd^2/n}{2\rho^2d^2/n^2})\right)
\\&\leq 2\exp(-\frac{\eta}{4}\rho n) 
\end{align*}

For the remaining term, we start off by noting the following:
\begin{align*}
\norm{(\Sigma^{1/2}\Pi_{\tilde{X}}\Sigma^{3/2})^T(\Sigma^{1/2}\Pi_{\tilde{X}}\Sigma^{3/2})}_2&= \norm{\Sigma^{3/2}\Pi_{\tilde{X}}\Sigma\Pi_{\tilde{X}}\Sigma^{3/2}}_2\leq \norm{\Sigma^{3/2}}_2^2\norm{\Pi_{\tilde{X}}\Sigma\Pi_{\tilde{X}}}_2\\&\leq \lambda_1^{3}\norm{\Pi_{\tilde{X}}}_2^2\norm{\Sigma}_2\leq \rho^4\frac{d^4}{n^4}\\
tr((\Sigma^{1/2}\Pi_{\tilde{X}}\Sigma^{3/2})^T(\Sigma^{1/2}\Pi_{\tilde{X}}\Sigma^{3/2}))&=tr(\Sigma^{3}\Pi_{\tilde{X}}\Sigma\Pi_{\tilde{X}})\leq tr(\Sigma^6)^{1/2}tr(\Pi_{\tilde{X}}\Sigma\Pi_{\tilde{X}}\Sigma)^{1/2}
\\&\leq tr(\Sigma^6)^{1/2}tr(\Sigma^2)^{1/2}\leq (\rho^5\frac{d^6}{n^5}\rho\frac{d^2}{n})^{1/2}=\rho^3\frac{d^4}{n^3}
\end{align*}

As a result,
\begin{align*}
 \mathbb{P}(|\tilde{x}^T\Pi_{\tilde{X}}\Sigma\tilde{x} -tr(\Pi_{\tilde{X}}\Sigma^2)|>.5\rho^2 b_n\frac{d^2}{n})&\leq 2\exp\left[-\eta\min(\frac{\rho^4b_n^2d^4/n^2}{4\norm{\Sigma^{1/2}\Pi_{\tilde{X}}\Sigma^{3/2}}_F^2},\frac{\rho^2 b_nd^2/n}{2\norm{\Sigma^{1/2}\Pi_{\tilde{X}}\Sigma^{3/2}}_2^{1/2}})\right]   
 \\&\leq 2\exp\left[-\eta\min(\frac{1}{4}\rho b_n^2n, \frac{1}{2}b_nn^2)\right]\leq 2\exp\left[-\frac{\eta}{4}\rho n^{1/3}\right]
\end{align*}

Thus, on the event $E^c\cap F^c$,
\begin{align*}
    \E{\frac{\tilde{x}^T\Pi_{\tilde{X}}\Sigma\tilde{x}}{(\tilde{x}^T(I-\Pi_{\tilde{X}})\tilde{x})^2}\E{\eps(\tilde{x})^2|\tilde{x}}\mathbf{1}(E^c\cap F^c)}\leq 3\frac{\rho^2b_n}{(1-b_n\rho)^2}\E{\frac{\tilde{x}^T\Sigma\tilde{x}}{\norm{\tilde{x}}^4}\E{\eps(\tilde{x})^2|\tilde{x}}}
\end{align*}

For $(E^c\cap F^c)^{c}$, we first note that $\mathbb{P}((E^c\cap F^c)^{c})\leq 2\mathbb{P}(E)+2\mathbb{P}(F)$. Thus, 
\begin{align*}
&\E{\frac{\tilde{x}^T\Pi_{\tilde{X}}\Sigma\tilde{x}}{(\tilde{x}^T(I-\Pi_{\tilde{X}})\tilde{x})^2}\E{\eps(\tilde{x})^2|\tilde{x}}\mathbf{1}\left((E^c\cap F^c)^c\right)}\\&\leq \sigma_{max}^2\E{(\tilde{x}^T\Pi_{\tilde{X}}\Sigma\tilde{x})^4}^{1/2}\E{\frac{1}{(\tilde{x}^T(I-\Pi_{\tilde{X}})\tilde{x})^8}}^{1/4}(\mathbb{P}((E^c\cap F^c)^{c}))^{1/4}
\\&\leq 16\sigma_{max}^2\exp[-\frac{c}{16}n^{1/3}\rho]\frac{tr(\Sigma^2)}{d^2}(1+o(1))=o(r(n))
\end{align*}

Where we note that, $\E{(\tilde{x}^T\Pi_{\tilde{X}}\Sigma\tilde{x})^2}\leq $

This comes from the following. We note that by Prop. \ref{prop:second_third_quadratic_forms},
\begin{align*}
    \E{(\tilde{x}^T\Pi_{\tilde{X}}\Sigma\tilde{x})^2}&\leq \E{tr(\Sigma\Pi_{\tilde{X}}\Sigma)^2+2tr(\Sigma^{1/2}\Pi_{\tilde{X}}\Sigma^2\Pi_{\tilde{X}}\Sigma^{1/2})}
    \\&\leq \Theta(\frac{n^2}{d^2})\rho^2\frac{d^2}{n^2}tr(\Sigma^2)+2\E{tr(\Sigma^2\Pi_{\tilde{X}}\Sigma\Pi_{\tilde{X}})}
    \\&\leq \Theta(\frac{n^2}{d^2})tr(\Sigma^2)^2+4\frac{n^2}{d^2}tr(\Sigma^2\Sigma^3)+4nr(n)tr(\Sigma^3)
    \\&\leq \Theta(\frac{n^2}{d^2}+\frac{n}{d}+nr(n))tr(\Sigma^2)^2
\end{align*}

Where we use the fact that,
\begin{align*}
    \E{tr(\Pi_{\tilde{X}}\Sigma)^2}&\leq \sum_{i,j}\lambda_i^2\lambda_j^2\E{a_i^TA^{-1}a_ia_j^TA^{-1}a_j}
    \\&\leq \sum_{i}\lambda_i^4\E{a_i^TA^{-1}a_ia_i^TA^{-1}a_i}+\sum_{i\neq j}\lambda_i^2\lambda_j^2\E{tr(A^{-1})^2}
    \\&\leq \sum_i \lambda_i^4\Theta(\frac{n^2}{d^2})+\Theta(\frac{n^2}{d^2})tr(\Sigma^2)^2
    \\&\leq \Theta(\frac{n^2}{d^2})tr(\Sigma^2)^2
\end{align*}

The result follows when we note that by the proof of Prop. \ref{prop:exp_noise_first_order},
\begin{align*}
    \E{tr(\Sigma X^T(XX^T)^{-2}X)}&\leq r(n)k(1+3kr(n))(1+o(1))\\
    \E{tr(\Sigma X^T(XX^T)^{-2}X)}&\geq r(n)k(1+kr(n))(1+o(1))    
\end{align*}

\end{proof}

\begin{proposition}\label{prop:noise_nN_v_n}
Consider the assumptions and notation of Prop. \ref{prop:noise_nN_v_n_gen}.
Then, $|\sigma^2-\tilde{\sigma}^2|\leq 10\rho\sigma^2$

\end{proposition}
\begin{proof}
Using the inequalities in Prop. \ref{prop:noise_nN_v_n_gen} and noting that $nr(n)\leq \rho^2$ by Assumption \ref{assumpt:weak_canonical_case}.4.
\begin{align*}
    \tilde{\sigma}^2-\sigma^2&\leq \E{\frac{x_1^T\Sigma x_1}{\norm{x_1}^4}\E{\eps_1^2|x_1}}(8\rho)\\
    \sigma^2-\tilde{\sigma}^2&\geq \E{\frac{x_1^T\Sigma x_1}{\norm{x_1}^4}\E{\eps_1^2|x_1}}(-8\rho)
\end{align*}

As a result,
\begin{align*}
    |\sigma^2-\tilde{\sigma}^2|\leq 8\rho\E{\frac{x_1^T\Sigma x_1}{\norm{x_1}^4}\E{\eps_1^2|x_1}}\leq 10\rho\sigma^2 
\end{align*}
    
\end{proof}
\begin{proposition}\label{prop:noise_nN_v_n_o_1}
Consider the assumptions and notation of Prop. \ref{prop:noise_nN_v_n_gen}. Let, $X^*$ be a second sub-sample of $X$ of $k'$ data points. Let,
\begin{align*}
    \tilde{\sigma}_*^2=\E{tr\left((X^*(X^*)^T)^{-1}X^*\Sigma (X^*)^T(X^*(X^*)^T)^{-1}\Lambda(X^*)\right)}/\E{tr(\Sigma (X^*)^T(X^*(X^*)^T)^{-2}X^*)}
\end{align*}
Then, when $k=o(n)$ and $k'=o(n)$, $|\tilde{\sigma}^2-\tilde{\sigma}_*^2|= o(\sigma^2)$.
\end{proposition}
\begin{proof}
By Prop. \ref{prop:noise_nN_v_n_gen},
\begin{align*}
    \tilde{\sigma}^2&=\E{\frac{x_1^T\Sigma x_1}{\norm{x_1}^4}\E{\eps_1^2|x_1}}(1+o(1))+o(1)\\
    \tilde{\sigma}_*^2&=\E{\frac{x_1^T\Sigma x_1}{\norm{x_1}^4}\E{\eps_1^2|x_1}}(1+o(1))+o(1)
\end{align*}

The result follows.
\end{proof}
\begin{proposition}\label{prop:var_exp_beta}
 
 We assume Assumption \ref{assumpt:weak_canonical_case} holds. Let $X_s$ be a sub-sample of $X$ with size $k$. Let $b_n$ be an increasing sequence and assume $$\sigma_{max}^2=\min\left(o(b_n^{-\eta/12}),\exp(-o(n^{1/3})\right)$$ where $\eta$ is the constant from Lemma \ref{lemma:hanson-wright}. Then,
 \begin{align*}
     \Var[\beta^T\Sigma X^T(XX^T)^{-1}\eps]\leq \Theta\left((\sigma^2+o(1))\frac{k\log(b_n)^2}{tr(\Sigma)^2}\right)\beta^T\Sigma^3\beta
 \end{align*}
\end{proposition}
\begin{proof}
 We can attain an upper bound in the following manner:
\begin{align*}
    \E{(\beta^T\Sigma X^T(XX^T)^{-1}\eps\eps^T(XX^T)^{-1}X\Sigma\beta)}&=\E{(\beta^T\Sigma X^T(XX^T)^{-1}\Lambda(X)(XX^T)^{-1}X\Sigma\beta)}
    \\&= n \E{e_n^T(XX^T)^{-1}X\Sigma\beta\beta^T\Sigma X^T(XX^T)^{-1}e_n\E{\eps_n^2|X_n}}
\end{align*}

Let $X=\begin{bmatrix}
    \tilde{X}\\ X_k
\end{bmatrix}$. For appropriate $\tilde{X}$ and we define $\tilde{x}=X_k$

Now, using the technique in the proof of Prop. \ref{prop:noise_nN_v_n_gen}, we note that,
\begin{align*}
  e_n^T(XX^T)^{-1}X\Sigma\beta\beta^T\Sigma X^T(XX^T)^{-1}e_n=\frac{\tilde{x}^T(I-\Pi_{\tilde{X}})\Sigma\beta\beta^T\Sigma (I-\Pi_{\tilde{X}})\tilde{x}}{(\tilde{x}^T(I-\Pi_{\tilde{X}})\tilde{x})^2}  
\end{align*}

Using event $E$ in the proof Prop. \ref{prop:noise_nN_v_n_gen}, we can see that,
\begin{align*}
   \E{e_n^T(XX^T)^{-1}X\Sigma\beta\beta^T\Sigma X^T(XX^T)^{-1}e_n\E{\eps_n^2|X_n}}\leq  \E{\frac{\tilde{x}^T\Sigma\beta\beta^T\Sigma \tilde{x}}{\norm{\tilde{x}}^4}\E{\eps_n^2|X_n}}+o(r(n))
\end{align*}

For the term, $\frac{\tilde{x}^T\Sigma\beta^T\beta^T\Sigma\tilde{x}}{\norm{\tilde{x}}^4}$. Consider the event,
\begin{align*}
    E:=\left\{\tilde{x}^T\Sigma\beta\beta^T\Sigma\tilde{x}>2
    \log(b_n)^2\frac{\beta^T\Sigma^3\beta}{tr(\Sigma^2)}\tilde{x}^T\Sigma\tilde{x}\right\}
\end{align*}

On the event $E^c$,
\begin{align*}
\E{\frac{\tilde{x}^T\Sigma\beta\beta^T\Sigma\tilde{x}}{\norm{\tilde{x}}^4}\E{\eps_k^2|X_k}\mathbf{1}(E^c)}&= \E{\frac{\tilde{x}^T\Sigma\beta\beta^T\Sigma\tilde{x}}{\tilde{x}^T\Sigma\tilde{x}}\frac{\tilde{x}^T\Sigma\tilde{x}}{\norm{\tilde{x}}^4}\E{\eps_k^2|X_k}\mathbf{1}(E^c)}
\\& \leq 2\log(b_n)^2\frac{\beta^T\Sigma^3\beta}{tr(\Sigma^2)}\E{\frac{\tilde{x}^T\Sigma\tilde{x}}{\norm{\tilde{x}}^4}}\leq 4\sigma^2\frac{\beta^T\Sigma^3\beta}{d^2}\log(b_n)^2
\end{align*}

Now, we note that,
\begin{align*}
    \mathbb{P}(E)&=\mathbb{P}\left(\tilde{x}^T\Sigma\beta\beta^T\Sigma\tilde{x}-\E{\tilde{x}^T\Sigma\beta\beta^T\Sigma\tilde{x}}>2\log(b_n)\frac{\beta^T\Sigma^3\beta}{tr(\Sigma^2)}\tilde{x}^T\Sigma\tilde{x}-\beta^T\Sigma^3\beta\right)
    \\&\leq \mathbb{P}\left(|\tilde{x}^T\Sigma\beta\beta^T\Sigma\tilde{x}-\E{\tilde{x}^T\Sigma\beta\beta^T\Sigma\tilde{x}}|>\frac{1}{2}\log(b_n)\beta^T\Sigma^3\beta\right)\\&+\mathbb{P}(|\tilde{x}^T\Sigma\tilde{x}-\E{\tilde{x}^T\Sigma\tilde{x}}|>\log(b_n)tr(\Sigma^2))
    \leq 4b_n^{-\eta/2}
\end{align*}

Where we note that by Lemma \ref{lemma:hanson-wright}
\begin{align*} 
 \mathbb{P}\left(|\tilde{x}^T\Sigma\beta\beta^T\Sigma\tilde{x}-\E{\tilde{x}^T\Sigma\beta\beta^T\Sigma\tilde{x}}|>\frac{\log(b_n)\beta^T\Sigma^3\beta}{2}\right)&\leq  2\exp(-\frac{\eta}{2}\log(b_n))\leq 2b_n^{-\eta/2}\\
 \mathbb{P}(|\tilde{x}^T\Sigma\tilde{x}-\E{\tilde{x}^T\Sigma\tilde{x}}|>\log(b_n)tr(\Sigma^2))&\leq 2\exp(-\eta\min(\frac{\log(b_n)^2tr(\Sigma^2)^2}{tr(\Sigma^4)},\frac{\log(b_n)tr(\Sigma^2)}{\lambda_1^2}))
 \\&\leq 2 \exp(-c\log(b_n))=2b_n^{-\eta}
\end{align*}

Where we use that $\min(\frac{\log(b_n)^2(\beta^T\Sigma^3\beta)^2}{4(\beta^T\Sigma^3\beta)^2},\frac{\log(b_n)\beta^T\Sigma^3\beta}{2\beta^T\Sigma^3\beta})\geq \log(b_n)/2$.

We note on the event $E$ by Cauchy-Schwartz, we have as follows,
\begin{align*}
\E{\frac{\tilde{x}^T\Sigma\beta\beta^T\Sigma\tilde{x}}{\norm{\tilde{x}}^4}\E{\eps_k^2|X_k}\mathbf{1}(E)}&\leq \mathbb{P}(E)^{1/4}\E{\frac{1}{\norm{\tilde{x}}^8}}^{1/4}\E{(\tilde{x}^T\Sigma\beta\beta^T\Sigma\tilde{x})^2}^{1/2}
\\&\leq \Theta(1)\sigma^2_{max}\frac{b_n^{-\eta/8}}{d^2}(1+o(1))\beta^T\Sigma^3\beta=o(\frac{\beta^T\Sigma^3\beta}{d^2})
\end{align*}

We note that applying the same technique to, the remaining terms, we get that,
\begin{align*}
    \E{\frac{\tilde{x}^T\Pi_{\tilde{X}}\Sigma\beta\beta^T\Sigma\tilde{x}}{(\tilde{x}^T(I-\Pi_{\tilde{X}})\tilde{x})^2}\E{\eps(\tilde{x})^2|\tilde{x}}}&\leq 8\rho^2(\frac{k}{n})^{1/3}\log(b_n)^2\frac{\beta^T\Sigma^3\beta}{tr(\Sigma^2)}\E{\frac{\tilde{x}^T\Sigma\tilde{x}}{\norm{\tilde{x}}^4}}\leq 8\rho^2(\frac{k}{n})^{1/3}\log(b_n)^2\frac{\beta^T\Sigma^3\beta}{d^2}\sigma^2\\
    \E{\frac{\tilde{x}^T\Pi_{\tilde{X}}\Sigma\beta\beta^T\Sigma\Pi_{\tilde{X}}\tilde{x}}{(\tilde{x}^T(I-\Pi_{\tilde{X}})\tilde{x})^2}} &\leq 8\rho^2\frac{k}{n}\sigma^2\log(b_n)^2\frac{\beta^T\Sigma^3\beta}{d^2}
\end{align*}

For the $\tilde{x}^T\Pi_{\tilde{X}}\Sigma\Pi_{\tilde{X}}\tilde{x}$ term, we apply its bounds the proof of Prop. \ref{prop:straight_var_in_ds} as follows:
\begin{align*}
    \E{\frac{\tilde{x}^T\Pi_{\tilde{X}}\Sigma\beta\beta^T\Sigma\Pi_{\tilde{X}}\tilde{x}}{(\tilde{x}^T(I-\Pi_{\tilde{X}})\tilde{x})^2}\E{\eps(\tilde{x})^2|\tilde{x}}}\leq \E{\frac{\tilde{x}^T\Pi_{\tilde{X}}\Sigma\Pi_{\tilde{X}}\tilde{x}}{(\tilde{x}^T(I-\Pi_{\tilde{X}})\tilde{x})^2}\E{\eps(\tilde{x})^2|\tilde{x}}}\beta^T\Sigma\beta\leq 6\rho^2\frac{k}{n}\sigma^2r(n)
\end{align*}

\end{proof}
\begin{proposition}
    We assume Assumption \ref{assumpt:weak_canonical_case} holds. Then,
    \begin{align*}
        c_{opt}=\frac{\frac{n}{d}\beta^T\Sigma^2\beta}{\frac{n^2}{d^2}\beta^T\Sigma^3\beta+(1+\sigma^2)r(n)n}(1+O(q))
    \end{align*}
\end{proposition}
\begin{proof}
We note that $c_{opt}=\frac{\E{\theta_{MN}^T\Sigma\beta}}{\E{\theta_{MN}^T\Sigma\theta_{MN}}}$.

We begin by bounding the denominator.
\begin{align*}
    \E{\theta_{MN}^T\Sigma\theta_{MN}}&\leq (1+\sigma^2)r(n)n(1+2r(n))/(1-\Theta(1)\rho)+(1+o(1))\frac{n^2(1+r(n)/(1-4\rho)^3)^2}{d^2}\beta^T\Sigma^2\beta
    \\&\leq (1+\sigma^2)r(n)n(1+\Theta(1)q^2)+\frac{n^2}{d^2}\beta^T\Sigma^3\beta(1+\Theta(1)q^2)
    \\\E{\theta_{MN}^T\Sigma\theta_{MN}}&\geq (1+\sigma^2)r(n)n+\frac{n^2}{d^2}\beta^T\Sigma^3\beta(1+O(1)q^2)-\Omega\left(\frac{n^3}{d^3}\right)\beta^T\Sigma^4\beta\\&\geq (1+\sigma^2)r(n)n+\frac{n^2}{d^2}\beta^T\Sigma^3\beta(1+O(1)q^2) -\Omega(1)q^3
\end{align*}

As a result,
\begin{align*}
   \E{\theta_{MN}^T\Sigma\theta_{MN}}=(1+\sigma^2)r(n)+\frac{n^2}{d^2}\beta^T\Sigma^3\beta(1+\Theta(1)q^2)+\Theta(1)q^3
\end{align*}

Similarly, for the numerator, we have the following,
\begin{align*}
    \E{\theta_{MN}^T\Sigma\beta}&\leq \frac{n}{d}\beta^T\Sigma^2\beta(1+o(1))+\Theta(1)r(n)\frac{n}{d}\beta^T\Sigma^2\beta\leq \frac{n}{d}\beta^T\Sigma^2\beta (1+\Theta(1)q)
    \\\E{\theta_{MN}^T\Sigma\beta}&\geq \frac{n}{d}\beta^T\Sigma^2\beta(1+o(1))+r(n)\frac{n}{d}\beta^T\Sigma^2\beta-\Omega(\frac{n^2}{d^2})\beta^T\Sigma^3\beta\geq \frac{n}{d}\beta^T\Sigma^2\beta(1+\Theta(1)q)
\end{align*}

As such,
\begin{align*}
    c_{opt}=\frac{q-\Theta(1)q^2}{\frac{n^2}{d^2}\beta^T\Sigma^3\beta+r(n)n(1+\sigma^2)+\Theta(1)q^3}=\frac{q}{\frac{n^2}{d^2}\beta^T\Sigma^3\beta+r(n)n(1+\sigma^2)}(1+\Theta(1)q)
\end{align*}

Where we note that,
\begin{align*}
&\frac{q}{\frac{n^2}{d^2}\beta^T\Sigma^3\beta+r(n)n(1+\sigma^2)+\Theta(1)q^3}-\frac{q}{\frac{n^2}{d^2}\beta^T\Sigma^3\beta+r(n)n(1+\sigma^2)}\\&=\frac{q}{\frac{n^2}{d^2}\beta^T\Sigma^3\beta+r(n)n(1+\sigma^2)}\left(\frac{1}{1+\Theta(1)q^3/(\frac{n^2}{d^2}\beta^T\Sigma^3\beta+r(n)n(1+\sigma^2))}-1\right)
\\&=\frac{q}{\frac{n^2}{d^2}\beta^T\Sigma^3\beta+r(n)n(1+\sigma^2)}\left(\frac{1}{1+\Theta(1)q^3/(\Theta(1)q^2)}-1\right)
\\&=\frac{q}{\frac{n^2}{d^2}\beta^T\Sigma^3\beta+r(n)n(1+\sigma^2)}\left(\Theta(1)q\right)
\end{align*}

\end{proof}
\begin{proposition}\label{prop:weak_implies_strong}
If Assumption \ref{assumpt:strong_canonical_case} is satisfied, then Assumption \ref{assumpt:weak_canonical_case} is satisfied.     
\end{proposition}
\begin{proof}
We note that Assumptions 2.11.1-2 are trivially satisfied.

We also note that $q:=\frac{n}{d}\beta^T\Sigma^2\beta\leq q_{max}$ and $\frac{n^2}{d^2}\beta^T\Sigma^3\beta \leq q_{max}^2$, through the identity that $\Sigma^{k+1}\preceq \lambda_1\Sigma^{k}$ which in turn implies Assumption 2.11.3 as well as that $q\ll 1$.

To finish the proof, we need to show that $\frac{n}{d}\beta^T\Sigma^2\beta\geq cq_{max}$. This comes as follows:
\begin{align*}
    \frac{n}{d}\beta^T\Sigma^2\beta=\frac{n}{d}\sum_{i\in K_n} \lambda_i^2(\beta^Tv_i)^2+\frac{n}{d}\sum_{i\notin K_n} \lambda_i^2(\beta^Tv_i)^2\geq cq_{max}\sum_{i\in K_n} \lambda_i(\beta^Tv_i)^2\geq c'q_{max}
\end{align*}
Where we use the non-vanishing of $\sum_{i\in K_n} \lambda_i(\beta^Tv_i)^2$ (Assumption 2.9.5).
\end{proof}

\begin{proposition}\label{prop:ds_MN_close_expect}
Consider Assumption \ref{assumpt:weak_canonical_case} and $N\to\infty$. Let $\rho:=\frac{n}{d}\lambda_1$. Then, for $c>0$, 
\begin{align*}
|G(c\theta_{ds})-G(c\theta_{MN})|\leq 3c\rho q+15(C_1+C_{noise})c^2\rho q^2    
\end{align*}. 

Further, as $c_{opt}\leq \frac{4}{3(C_1+C_{noise})q}$, 
\begin{align*}
|G(c_{opt}\theta_{ds})-G(c_{opt}\theta_{MN})|\leq 30\rho    
\end{align*}

Further,
\begin{align*}
    |G(c^*\theta_{ds})-G(c_{opt}\theta_{MN})|\leq 3q
\end{align*}

\end{proposition}
\begin{proof}
We note that from the proof of Thm. \ref{theorem:canon_multiplicative}, $c_{opt}\leq O(1)\frac{1}{q}$.   
We start by noting,
\begin{align*}
    G(c\theta_{ds})&=\E{(\beta-c\theta_{ds})^T\Sigma(\beta-c\theta_{ds})}
    \\&=\beta^T\Sigma\beta-2c\beta^T\Sigma\E{\theta_{ds}}+c^2\E{\theta_{ds}^T\Sigma\theta_{ds}}
\end{align*}

We note that by Props. \ref{prop:proj_expect}, \ref{prop:straight_var_in_ds}, we obtain the following bounds:
\begin{align*}
|\beta^T\Sigma\E{\theta_{MN}}-\beta^T\Sigma\E{\theta_{ds}}|\leq 6\rho q
\end{align*}

For the quadratic term, we turn to Props. \ref{prop:straight_var_in_ds}, \ref{prop:var_proj_sig_proj}, \ref{prop:exp_noise_first_order}, and \ref{prop:noise_nN_v_n} to obtain the following bounds,
\begin{align*}
    |\E{\theta_{ds}^T\Sigma\theta_{ds}}-\E{\theta_{MN}^T\Sigma\theta_{MN}}|&\leq 10\rho\sigma^2nr(n)+4(C_1+C_{noise})\rho q^2
    \\&\leq 15(C_1+C_{noise})\rho q^2
\end{align*}

Now, for the last claim,
\begin{align*}
    G(c^*\theta_{ds})-G(c_{opt}\theta_{MN})=\frac{\E{\theta_{MN}^T\Sigma\beta}^2}{\E{\theta_{MN}^T\Sigma\theta_{MN}}}-\frac{\E{\theta_{ds}^T\Sigma\theta_{ds}}^2}{\E{\theta_{ds}^T\Sigma\theta_{ds}}}
\end{align*}

We can see that by Prop. \ref{prop:straight_var_in_ds},
\begin{align*}
    \frac{\E{\theta_{ds}^T\Sigma\theta_{ds}}^2}{\E{\theta_{ds}^T\Sigma\theta_{ds}}}=\frac{(\frac{n}{d}\beta^T\Sigma^2\beta)^2}{\frac{n^2}{d^2}\beta^T\Sigma^3\beta+(1+\sigma^2)nr(n)}(1+o(1))
\end{align*}

We note the following lower bound, (using Props. \ref{prop:proj_expect}, \ref{prop:exp_sig_exp}, and \ref{prop:exp_noise_first_order})
\begin{align*}
   \frac{\E{\theta_{MN}^T\Sigma\beta}^2}{\E{\theta_{MN}^T\Sigma\theta_{MN}}}&\geq \frac{\left(\frac{n}{d}\beta^T\Sigma^2\beta(1+O(1)nr(n))-\frac{n^2}{d^2}\beta^T\Sigma^3\beta(1+O(1)nr(n))\right)^2}{(1+\sigma^2)nr(n)(1+O(1)nr(n))+\frac{n^2}{d^2}\beta^T\Sigma^3\beta(1+O(1)nr(n))} 
   \\&\geq \frac{(\frac{n}{d}\beta^T\Sigma^2\beta)^2}{(1+\sigma^2)nr(n)+\frac{n^2}{d^2}\beta^T\Sigma^3\beta}(1+O(1)nr(n))(1-2C_1q)
   \\&\geq \frac{(\frac{n}{d}\beta^T\Sigma^2\beta)^2}{(1+\sigma^2)nr(n)+\frac{n^2}{d^2}\beta^T\Sigma^3\beta}(1-3C_1q)
\end{align*}

As a result,
\begin{align*}
    G(c^*\theta_{ds})-G(c_{opt}\theta_{MN})\geq -3\frac{1}{C_1+C_{noise}}C_1q\geq -3q
\end{align*}

Now for the upper bound,
\begin{align*}
    \frac{\E{\theta_{MN}^T\Sigma\beta}^2}{\E{\theta_{MN}^T\Sigma\theta_{MN}}}&\leq \frac{(\frac{n}{d}\beta^T\Sigma^2\beta)^2}{\frac{n^2}{d^2}\beta^T\Sigma^3\beta-2C_1\rho q^2}(1+O(nr(n))+o(1))
    \\&\leq \frac{(\frac{n}{d}\beta^T\Sigma^2\beta)^2}{\frac{n^2}{d^2}\beta^T\Sigma^3\beta}(1+O(nr(n))+o(1))(1+3C_1\rho q^2)
\end{align*}

The result follows.

\end{proof}

\begin{proposition}\label{prop:weak_c_opt_bounds}
    Consider Assumption \ref{assumpt:weak_canonical_case}, $\lambda_1\leq\rho\frac{d}{n}$. Then, 
    \begin{align*}
        |c^*-c_{opt}|\leq\kappa \rho
    \end{align*}
    Where $\kappa$ is universal. 
\end{proposition}
\begin{proof}
First, we begin by noting that from Prop. \ref{prop:bounds_effective_ranks} and terms of our assumptions, $\Theta(1)\rho^2/n\leq r(n)n\leq \Theta(1)\rho$. Now, using Props. \ref{prop:proj_expect}, \ref{prop:exp_noise_first_order}, and \ref{prop:exp_sig_exp}, we can see,
\begin{align*}
    c_{opt}&\leq \frac{\frac{n}{d}\beta^T\Sigma^2\beta}{(1+\sigma^2)nr(n)+\frac{n^2}{d^2}\beta^T\Sigma^3\beta}(1+\kappa_1\rho)\\
    c_{opt}&\geq \frac{\frac{n}{d}\beta^T\Sigma^2\beta}{(1+\sigma^2)nr(n)+\frac{n^2}{d^2}\beta^T\Sigma^3\beta}(1-\kappa_2\rho)
\end{align*}

For universal constants $\kappa_1,\kappa_2$.

From Prop. \ref{prop:straight_var_in_ds}, we see that
\begin{align*}
    c^*=\frac{\frac{n}{d}\beta^T\Sigma^2\beta}{(1+\sigma^2)nr(n)+\frac{n^2}{d^2}\beta^T\Sigma^3\beta}
\end{align*}

The result follows.
\end{proof}

\subsection{Main Results}
\begin{theorem}\label{thm:add_improve}
Under Assumption \ref{assump:additive}, whenever $\frac{n}{d}\lambda_1\leq \frac{1}{32}$ and $n\gg 1$,
\begin{equation}\label{equ:additive_gen_err}
G(\theta_{MN})>G(c\theta_{MN})\; \forall c\in(1,2c_{opt}-1),\;c_{opt}>1    
\end{equation}
 Further, when $\lambda_1=o(\frac{d}{n})$ and $\sigma^2\equiv\sigma_{max}^2$, Assumption \ref{assump:additive}.3 is also a necessary condition for Equ. \ref{equ:additive_gen_err} to hold and the universal constant condition can be dropped.  
\end{theorem}
\begin{proof}
We note that by Props. \ref{prop:proj_expect} and \ref{prop:exp_sig_exp}, 
\begin{align*}
\beta^T\E{\Sigma\theta_{MN}}=\beta^T\E{\Pi_X\Sigma}\beta\geq \frac{n(1+r(n))}{tr(\Sigma)}\beta^T\Sigma^2\beta(1+o(1))-\Omega(\frac{n^2}{tr(\Sigma)^2})\beta^T\Sigma^3\beta   
\end{align*}
\begin{align*}
\E{\theta_{MN}^T\Sigma\theta_{MN}}&=\E{\eps^T(XX^T)^{-1}X\Sigma X^T(XX^T)^{-1}\eps}+\beta^T\E{\Pi_X\Sigma\Pi_X}\beta\\&  
\leq (1+\sigma_{max}^2)nr(n)\left(1+O(rn(n)\right)+\Theta\left(\frac{n^2}{tr(\Sigma)^2}\right)\beta^T\Sigma^3\beta
\end{align*}

Using the fact that $G(\theta_{MN})-G(c\theta_{MN})$ is a quadratic with a zero at one.
\begin{align*}
\beta^T\E{\Sigma\theta_{MN}}-\E{\theta_{MN}^T\Sigma\theta_{MN}}&\geq \frac{n}{d}\beta^T\Sigma^2\beta-nr(n)(1+\sigma_{max}^2)\\&-\Omega(\frac{n^2}{tr(\Sigma)^2})\beta^T\Sigma^3\beta-O(n^2r(n)^2)(1+\sigma_{max}^2)  
\\&\geq \frac{n}{d}\beta^T\Sigma^2\beta-nr(n)(1+\sigma_{max}^2)\\&-(\frac{n}{d}\lambda_1)\Theta(1)\frac{n}{d}\beta^T\Sigma^2\beta-\Theta(1)(1+\sigma^2_{max})(\frac{n}{d}\lambda_1)^2nr(n) 
\end{align*}

The sufficient result follows because when $\frac{n}{d}\lambda_1$ is sufficiently small,
\begin{align*}
    (\frac{n}{d}\lambda_1)\Theta(1)\beta^T\Sigma^2\beta+\Theta(1)(1+\sigma^2_{max})(\frac{n}{d}\lambda_1)^2nr(n)\ll \frac{n}{d}\beta^T\Sigma^2\beta-nr(n)(1+\sigma_{max}^2)
\end{align*}

For the necessary part, this follows by,
\begin{align*}
    \lim_{n\to\infty} \frac{(\frac{n}{d}\lambda_1)\Theta(1)\beta^T\Sigma^2\beta+\Theta(1)(1+\sigma^2_{max})(\frac{n}{d}\lambda_1)^2nr(n)}{\frac{n}{d}\beta^T\Sigma^2\beta-nr(n)(1+\sigma_{max}^2)}=0
\end{align*}

\end{proof}

\begin{theorem}\label{theorem:canon_multiplicative}
    Consider Assumption \ref{assumpt:weak_canonical_case}. There exists an $\alpha\in (0,1)$ (not depending on $n$ but possibly on $C_1$ or $C_{noise}$) such that when $q<\frac{1}{216(C_1+C_{noise})}$ and $n$ sufficiently large, $\alpha G(\theta_{MN})\geq G(c_{opt}\theta_{MN})$ and $c_{opt}>1$. Further, when $r(n)=o(1/n)$, then, $\alpha$ only depends on $C_1$ and the restriction on $q$ becomes $q<\frac{1}{216C_1}$. 
\end{theorem}
\begin{proof}

We first focus on $c_{opt}$. We note that the numerator of $c_{opt}$ has the following lower bound:
\begin{align*}
    \E{\beta^T\Sigma\theta_{MN}}\geq \frac{n}{d}\beta^T\Sigma^2\beta-3\frac{n^2}{d^2}\beta^T\Sigma^3\beta=q-3\rho q
\end{align*}

Where we use the fact that $\Sigma^3\preceq \rho\frac{d}{n} \Sigma^2$. 

In a similar vein, we note that
\begin{align*}
    \E{\theta_{MN}^T\Sigma\theta_{MN}}&=(1+o(1))\frac{n^2}{d^2}\beta^T\Sigma^3\beta\left(1+nr(n)/(1-4\rho)^3\right)+(1+\sigma^2)nr(n)(1+nr(n)/(1-2\rho)^2)\\
    &\leq 2\frac{n^2}{d^2}\beta^T\Sigma^3\beta+2(1+\sigma^2)nr(n)\leq 2\rho q+2C_{noise}q^2
\end{align*}
As a result,
\begin{align*}
    c_{opt}&=\frac{\E{\beta^T\Sigma\theta_{MN}}}{\E{\theta_{MN}^T\Sigma\theta_{MN}}}\geq \frac{q-\rho q}{2\rho q+2C_{noise}q^2}
    \\&\geq \frac{1}{2\rho+2C_{noise}q}-\frac{\rho}{2\rho+2C_{noise}}
    \\&\geq \frac{1}{2(\rho+C_{noise}q)}-\frac{1}{2}
\end{align*}

Thus, when $q<\frac{1}{C_{noise}}$, $c_{opt}>1$. We note that when $r(n)=o(1/n)$,
\begin{align*}
  c_{opt}\geq \frac{q-\rho q}{2\rho q+o(1)}=\frac{1-\rho}{2\rho +o(1)}
\end{align*}

Since $\rho\leq \frac{1}{8}$ for $n$ sufficiently large, $c_{opt}>1$.

We note that our objective is to show that for fixed $C_1$, $C_{noise}$ and some $\alpha$ not depending on $n$ or $q$,
\begin{align*}
1-\frac{\E{\beta^T\Sigma\theta_{MN}}^2}{\E{\theta_{MN}^T\Sigma\theta_{MN}}}\leq \alpha-2\alpha \E{\beta^T\Sigma\theta_{MN}}+\alpha\E{\theta_{MN}^T\Sigma\theta_{MN}}
\end{align*}

Which is equivalent to showing that,
\begin{align*}
    0\leq -(1-\alpha)+\frac{\E{\beta^T\Sigma\theta_{MN}}^2}{\E{\theta_{MN}^T\Sigma\theta_{MN}}}-2\alpha\E{\beta^T\Sigma\theta_{MN}}+\alpha^2\E{\theta_{MN}^T\Sigma\theta_{MN}}
\end{align*}

Now we note that by Props. \ref{prop:proj_expect}, \ref{prop:exp_sig_exp}, and \ref{prop:exp_noise_first_order}, 
\begin{align*}
\frac{1}{3}q^2\leq& \E{\theta_{MN}^T\Sigma\theta_{MN}}\leq 3(C_1+C_{noise})q^2\\
\frac{1}{2}q\leq &\E{\theta_{MN}^T\Sigma\beta}\leq 2q
\end{align*}

As a result,
\begin{align*}
\frac{\E{\beta^T\Sigma\theta_{MN}}^2}{\E{\theta_{MN}^T\Sigma\theta_{MN}}}-2\alpha\E{\beta^T\Sigma\theta_{MN}}+\alpha\E{\theta_{MN}^T\Sigma\theta_{MN}}&\geq \frac{1}{9(C_1+C_{noise})}\\&-6\alpha q+\frac{1}{3}\alpha^2q^2    
\end{align*}

Therefore, let $\alpha=1-\frac{1}{18}\frac{1}{C_1+C_{noise}}$.

We can see that,
\begin{align*}
    \alpha G(\theta_{MN})-G(c_{opt}\theta_{MN})&\geq -(1-\alpha)+\frac{1}{9(C_1+C_{noise})}-6\alpha q+\frac{1}{3}\alpha^2q^2 
    \\& =-\frac{1}{18(C_1+C_{noise})}+\frac{1}{9(C_1+C_{noise})}-6q+\frac{2}{3(C_1+C_{noise})}q
    \\&-\frac{1}{3}q^2+\frac{1}{27(C_1+C_{noise})}q^2+\frac{1}{972(C_1+C_{noise})^2}q^2
    \\&=\frac{1}{18(C_1+C_{noise})}-\left(6+\frac{2}{3(C_1+C_{noise})}\right)q\\&+\left(\frac{1}{27(C_1+C_{noise})}-\frac{1}{3}+\frac{1}{972(C_1+C_{noise})^2}\right)q^2
\end{align*}

We note that when $q<\frac{1}{216}\frac{1}{C_1+C_{noise}}$, $\alpha G(\theta_{MN})-G(c_{opt}\theta_{MN})>0$.

When $r(n)=o(1/n)$, then, for $n\gg 1$,
\begin{align*}
    c_{opt}&\geq \frac{q-\rho q}{2\rho q+o(1)}\geq \frac{1}{3\rho}-1\geq \frac{5}{3}\\
    \frac{1}{3}q^2&\leq \E{\theta_{MN}^T\Sigma\theta_{MN}}
\leq 3C_1q^2
\end{align*}

By repeating the process from above, we can obtain the remaining results.

\end{proof}
\begin{proposition}\label{prop:simple_canon_multiplicative}
    Consider Assumption \ref{assumpt:strong_canonical_case}. There exists an $\alpha\in (0,1)$ (not depending on $n$ but possibly on $\alpha'$, $\alpha_{min}$, or $C_{noise}$) such that when $q_{max}<\frac{(\alpha'\alpha_{min})^2}{24(1+C_{noise})}$ and $n$ sufficiently large, $\alpha G(\theta_{MN})\geq G(c_{opt}\theta_{MN})$ and $c_{opt}>1$. Further, when $r(n)=o(1/n)$, then, $\alpha$ only depends on $\alpha'\alpha_{min}$ and the restriction on $q_{max}$ becomes $q_{max}<\frac{(\alpha'\alpha_{min})^2}{24}$. 
\end{proposition}
\begin{proof}
    We follow the proof structure of Theorem \ref{theorem:canon_multiplicative}.

    The key idea to note is that,
    \begin{align*}
        \frac{n}{d}\beta^T\Sigma^2\beta\geq q_{max}\alpha'\alpha_{min}-3q_{max}^2
    \end{align*}

    Similarly, we can see that,
    \begin{align*}
        \E{\theta_{MN}^T\Sigma\theta_{MN}}\leq 2q_{max}^2+2C_{noise}q_{max}^2
    \end{align*}

    As a result,
    \begin{align*}
        c_{opt}\geq \frac{q_{max}\alpha'\alpha_{min}-3 q_{max}^2}{2q_{max}^2+2C_{noise}q_{max}^2}=\frac{\alpha'\alpha_{min}-3q_{max}}{2(1+C_{noise})q_{max}}
    \end{align*}

    Thus, when $q_{max}<\frac{\alpha'\alpha_{min}}{10(C_{noise}+1)}$, $c_{opt}>1$.

    Likewise, if $nr(n)=o(1)$, when $q_{max}<\frac{\alpha'\alpha_{min}}{10}$, we have that $c_{opt}>1$.

    We also note that,
    \begin{align*}
        \E{\theta_{MN}^T\Sigma\beta}&\leq q_{max}\\
        \alpha'\alpha _{min}^2q_{max}^2&\leq\E{\theta_{MN}^T\Sigma\theta_{MN}} 
    \end{align*}
    For the multiplicative improvement, we begin by noting that:
    \begin{align*}
        \frac{\E{\beta^T\Sigma\theta_{MN}}^2}{\E{\theta_{MN}^T\Sigma\theta_{MN}}}-2\alpha\E{\beta^T\Sigma\theta_{MN}}+\alpha\E{\theta_{MN}^T\Sigma\theta_{MN}} &\geq \frac{(q_{max}\alpha'\alpha_{min}-3q_{max}^2)^2}{2(1+C_{noise})q_{max}^2}-2\alpha q_{max}
        \\&\geq \frac{q_{max}^2(\alpha'\alpha_{min})^2-6(\alpha'\alpha_{min})q_{max}^3+9q_{max}^4}{2(1+C_{noise})q_{max}^2}-2\alpha q_{max}
        \\&\geq \frac{(\alpha'\alpha_{min})^2-6q_{max}(\alpha'\alpha_{min})+9q_{max}^2}{2(1+C_{noise})}-2\alpha q_{max}
    \end{align*}
Let $\alpha=1-\frac{1}{4}\frac{(\alpha'\alpha_{min})^2}{1+C_{noise}}$ and suppose $q_{max}\leq \frac{1}{24}\frac{(\alpha'\alpha_{min})^2}{1+C_{noise}}$
As a result,
\begin{align*}
    \alpha G(\theta_{MN})-G(c_{opt}\theta_{MN})&\geq -(1-\alpha)+\frac{(\alpha'\alpha_{min})^2-6q_{max}(\alpha'\alpha_{min})+9q_{max}^2}{2(1+C_{noise})}-2\alpha q_{max}
    \\&\geq \frac{(\alpha'\alpha_{min})^2}{(1+C_{noise})}\left(\frac{1}{4}-\frac{3}{24}-\frac{1}{12}\right)\geq 0
\end{align*}

    For similar reasons, we have the analogous result when $nr(n)=o(1)$.
\end{proof}

\begin{proposition}\label{prop:straight_var_in_ds}
Under Assumption \ref{assumpt:weak_canonical_case} and assume $\sigma_{max}^2$ is polynomial in $n$, then
\begin{align*}
    \frac{n}{d}\beta^T\Sigma^2\beta(1+o(1)-\frac{3\rho}{N})&\leq\E{\theta_{ds}^T\Sigma\beta}\leq \frac{n}{d}\beta^T\Sigma^2\beta(1+\Theta(\frac{1}{N})-\frac{\rho}{2N})\\
    \E{\theta_{ds}^T\Sigma\theta_{ds}}&\leq (1+\tilde{\sigma}^2)nr(n)\left(1+\frac{3}{N}\right)(1+o(1))+\frac{n^2}{d^2}(1+o(1)+\frac{1}{N})\beta^T\Sigma^3\beta\\
    \E{\theta_{ds}^T\Sigma\theta_{ds}}&\geq (1+\tilde{\sigma}^2)nr(n)(1-\frac{1}{N})(1+o(1))+\frac{n^2}{d^2}(1+o(1)-\frac{6}{N})\beta^T\Sigma^3\beta
\end{align*}

Where, $\tilde{\sigma}^2$ is as defined in Prop. \ref{prop:noise_nN_v_n}.
\begin{align*}
    \Var[\beta^T\Sigma\theta_{ds}]&\leq \Theta(\frac{1}{n}(1+\sigma^2\log(n)^2)\\
    \Var[\theta_{ds}^T\Sigma\theta_{ds}]&\leq \Theta\Big(nr(n)^2(\sigma_{max}^2+\sigma_{max}^4)+\frac{\sigma_{max}^2}{Nn}+\frac{1}{N^3}+r(n)^2\\&+\frac{nr(n)\sigma_{max}^2}{\sqrt{Nn}}+\frac{n^2r(n)^2(\sigma_{max}^2+\sigma_{max}^4)}{N}+\frac{n^2r(n)}{Nd^2}\Big)
\end{align*}

\end{proposition}
\begin{proof}
Now, we note that, for $\theta_{ds}^T\Sigma\beta$ we use Prop. \ref{prop:proj_expect}. We have the upper bound,
\begin{align*}
    \E{\theta_{ds}^T\Sigma\beta}&\leq \sum_{i=1}^{N-1}\frac{n/N}{d}\beta^T\Sigma^2\beta\left(1+\Theta(\frac{1}{N})+\frac{n/Nr(n)}{(1-2\rho)^2}\right)-\frac{n^2/N^2}{d^2}\frac{1+2n/Nr(n)(1-\Theta(1)\rho)}{1+\rho}\beta^T\Sigma^3\beta
    \\&\leq \frac{n}{d}\beta^T\Sigma^2\beta\left(1+\Theta(\frac{1}{N})-\frac{1}{N}\frac{n^2}{d^2(1+\rho)}\beta^T\Sigma^3\beta(1+o(1)\right)\leq \frac{n}{d}\beta^T\Sigma^2\beta\left(1+o(1)-\frac{\rho}{N}(1+o(1)\right)
\end{align*}

We also have the lower bound,
\begin{align*}
 \E{\theta_{ds}^T\Sigma\beta}&\geq \sum_{i=1}^{N-1}\frac{n/N}{d}\beta^T\Sigma^2\beta(1+o(1))+\frac{n^2/N^2r(n)}{d}\beta^T\Sigma^2\beta(1+o(1))-3\frac{n^2/N^2}{d^2}\beta^T\Sigma^3\beta    
 \\&\geq \frac{n}{d}\beta^T\Sigma^2\beta(1+o(1)-\frac{3\rho}{N})
\end{align*}

For $\theta_{ds}^T\Sigma\theta_{ds}$, we use Prop. \ref{prop:exp_sig_exp} and Prop. \ref{prop:proj_expect}.
\begin{align*}
    \E{\theta_{ds}^T\Sigma\theta_{ds}}=\sum_{i,j}\E{(\theta_{MN}^{(i)})^T\Sigma\theta_{MN}^{(j)}}
\end{align*}

When, $i=j$, we have the following bounds. We first focus on the noiseless terms. For upper bound,
\begin{align*}
\E{\beta^T\Pi_{X^{(i)}}\Sigma\Pi_{X^{(i)}}\beta}&\leq (1+o(1))\frac{n}{N}r(n)(1+\frac{n}{N}r(n)/(1-\rho)^2)\beta^T\Sigma\beta+2(1+o(1))\frac{n^2/N^2}{d^2}\beta^T\Sigma^3\beta
\\&\leq \frac{1}{N}\beta^T\Sigma\beta\left(nr(n)+2\frac{n^2}{N}r(n)^2+\frac{2}{N}\right)
\end{align*}

For the lower bound,
\begin{align*}
    \E{\beta^T\Pi_{X^{(i)}}\Sigma\Pi_{X^{(i)}}\beta}&\geq \frac{n}{N}r(n)(1+o(1))(1+\frac{n}{N}r(n))\beta^T\Sigma\beta+\frac{n^2/N^2}{2d^2}\beta^T\Sigma^3\beta-\Omega(\frac{n^3}{d^3})\beta^T\Sigma^4\beta
    \\&\geq \frac{n}{N}(1+o(1)r(n)+\frac{n}{N}r(n)^2+\frac{\rho^2}{2N^2})\beta^T\Sigma\beta 
\end{align*}

When $i\neq j$, we have the following bounds. For the upper bound,
\begin{align*}
    \E{\beta^T\Pi_{X^{(i)}}\Sigma\Pi_{X^{(j)}}\beta}&\leq \frac{n^2/N^2}{d^2}(1+o(1)+\frac{nr(n)}{N})^2\beta^T\Sigma^3\beta
    \\&\leq \frac{n^2/N^2}{d^2}(1+o(1)+2\frac{nr(n)}{N})\beta^T\Sigma^3\beta
\end{align*}

For the lower bound,
\begin{align*}
    \E{\beta^T\Pi_{X^{(i)}}\Sigma\Pi_{X^{(j)}}\beta}&\geq \frac{n^2/N^2}{d^2}(1+o(1)+\frac{nr(n)-3\rho}{N})^2\beta^T\Sigma^3\beta
    \\&\geq \frac{n^2/N^2}{d^2}(1+o(1)-\frac{6\rho}{N})\beta^T\Sigma^3\beta
\end{align*}

For the $\eps$ dependent terms, we have the following bounds for some sub-sample of size $n/N$, $X_s$. For the upper bound we use Props. \ref{prop:exp_noise_first_order} and \ref{prop:noise_nN_v_n},
\begin{align*}
    \E{\eps^T(X_sX_s^T)^{-1}X_s\Sigma X_s^T(X_sX_s^T)^{-1}\eps}\leq \tilde{\sigma}^2r(n)\frac{n}{N}(1+3\frac{n}{N}r(n))(1+o(1))
\end{align*}

Conversely, for the lower bound,
\begin{align*}
    \E{\eps^T(X_sX_s^T)^{-1}X_s\Sigma X_s^T(X_sX_s^T)^{-1}\eps}\geq \tilde{\sigma}^2r(n)\frac{n}{N}(1+\frac{n}{N}r(n))(1+o(1))
\end{align*}

By Props. \ref{prop:var_Sig_proj} and \ref{prop:var_exp_beta},
\begin{align*}
    \Var[\beta^T\Sigma\theta_{MN}^{(i)}]&\leq 2\Var[\beta^T\Sigma\Pi_{X^{(i)}}\beta]+2\Var[\beta^T\Sigma (X^{(i)})^T((X^{(i)})(X^{(i)})^T)^{-1}\eps^{(i)}]
    \\&\leq \Theta(\frac{n^3r(n)}{N^3tr(\Sigma)^2})(\beta^T\Sigma^2\beta)^2+\Theta(\frac{n/N}{tr(\Sigma)^2})\beta^T\Sigma^3\beta+\Theta(\frac{n^3}{N^3tr(\Sigma)^3})\beta^T\Sigma^3\beta\beta^T\Sigma^2\beta\\&+\Theta(\sigma^2\log(n)^2\frac{n/N}{d^2})\beta^T\Sigma^3\beta
    \\&\leq \Theta(\frac{nr(n)}{N^3})+\Theta(\frac{1}{nN})+\Theta(\frac{1}{N^3})+\Theta(\frac{\log(n)^2}{Nn})\sigma^2
\end{align*}

Where we use the fact that $\sigma_{max}^2$ is polynomial in $n$ to construct the appropriate $b_n$ for Prop. \ref{prop:var_exp_beta}.

As a result,
\begin{align*}
    \Var[\beta^T\Sigma\theta_{ds}]=\Theta(\frac{1}{n}(1+\log(n)^2\sigma^2))
\end{align*}

Turning to $\theta_{ds}^T\Sigma\theta_{ds}$, we first note that,
\begin{align*}
    \E{(\theta_{ds}^T\Sigma\theta_{ds})^2}&=\sum_{i,j,k,\ell}\E{(\theta_{MN}^{(i)})^T\Sigma\theta_{MN}^{(j)}(\theta_{MN}^{(k)})^T\Sigma\theta_{MN}^{(\ell)}}\\
    \E{\theta_{ds}^T\Sigma\theta_{ds}}^2&=\sum_{i,j,k,\ell}\E{(\theta_{MN}^{(i)})^T\Sigma\theta_{MN}^{(j)}}\E{(\theta_{MN}^{(k)})^T\Sigma\theta_{MN}^{(\ell)}}
\end{align*}

When $i,j$ is distinct from $k,\ell$, the term's respective contributions to the first and second moment are equal. The same is true when all indexes are unique. As a result, we can write the variance as follows:
\begin{align*}
    \Var[\theta_{ds}^T\Sigma\theta_{ds}]&\leq\sum_{i=1}^N\Var[(\theta_{MN}^{(i)})^T\Sigma\theta_{MN}^{(i)}]+24\sum_{i\neq j}^N\left|\E{(\theta_{MN}^{(i)})^T\Sigma \theta_{MN}^{(i)}(\theta_{MN}^{(i)})^T\Sigma \theta_{MN}^{(j)}}\right|
    \\&+24\sum_{i\neq j}\E{(\theta_{MN}^{(j)})^T\Sigma \theta_{MN}^{(i)}(\theta_{MN}^{(i)})^T\Sigma \theta_{MN}^{(j)}}-\E{(\theta_{MN}^{(j)})}^T\Sigma \E{\theta_{MN}^{(i)}}\E{(\theta_{MN}^{(i)})}^T\Sigma \E{\theta_{MN}^{(j)}}
\end{align*}

Using Props. \ref{prop:var_noise_pseudo_inverse}, \ref{prop:var_noise_proj}, and \ref{prop:var_proj_sig_proj}, we have the following:
\begin{align*}
\Var[\beta^T\Pi_X\Sigma\Pi_X\beta] &\leq \frac{1}{N^2}\Theta(\frac{r(n)n^2}{tr(\Sigma)^2})+\frac{1}{N^3}\Theta(\frac{n^3r(n)}{tr(\Sigma)^2})\beta^T\Sigma^3\beta
\\ &+\frac{1}{N^4}\Theta(\frac{n^4}{tr(\Sigma)^4})\left((\beta^T\Sigma^3\beta)^2+\beta^T\Sigma^5\beta\right)
\\&\leq \frac{1}{N^2}\Theta(\frac{r(n)n^2}{tr(\Sigma)^2})+\frac{1}{N^3}\Theta(nr(n))
\\ &+\frac{1}{N^4}\Theta(1)\rho^4 \\
\Var[\beta^T\Pi_X\Sigma X^T(XX^T)^{-1}\eps] &\leq \Theta(n^2r(n)^2)\frac{1}{N^2}\sigma_{max}^2 \\
\Var[\eps^T(XX^T)^{-1}X\Sigma X^T(XX^T)^{-1}\eps] &\leq \sigma_{max}^4\Theta(1)r(n)^2\frac{n^2}{N^2}
\end{align*}

As a result,
\begin{align*}
    \Var[(\theta_{MN}^{(i)})^T\Sigma\theta_{MN}^{(i)}]\leq \Theta\left(\frac{1}{N^2}\left(n^2r(n)^2(\sigma_{max}^2+\sigma_{max}^4)+\frac{n}{Nd^2}nr(n)+\frac{1}{N^2}\right)\right)
\end{align*}

Addressing the second sum, we see that for $i\neq j$,
\begin{align*}
\E{(\theta_{MN}^{(i)})^T\Sigma \theta_{MN}^{(i)}(\theta_{MN}^{(i)})^T\Sigma \theta_{MN}^{(j)}}&\leq \Theta(\frac{n/N}{d})\E{(\theta_{MN}^{(i)})^T\Sigma \theta_{MN}^{(i)}(\theta_{MN}^{(i)})^T\Sigma^2\beta}
\\&\leq \Theta(\frac{n/N}{d})\E{((\theta_{MN}^{(i)})^T\Sigma \theta_{MN}^{(i)})^2}^{1/2}\E{(\beta^T\Sigma\theta_{MN}^{(i)})^2}^{1/2}
\\&\leq \Theta(\frac{1}{N})\Theta(\frac{1}{N}nr(n)\sigma_{max}^2)\Theta(\frac{1}{\sqrt{Nn}})
\end{align*}

Addressing the third term, let $c_i=\beta^Tv_i$
\begin{align*}
    \beta^T\Pi_{X^{(i)}}\Sigma\Pi_{X^{(j)}}\beta=\sum_{k,\ell,m}c_kc_\ell\sqrt{\lambda_k\lambda_m}\lambda_\ell^2(a_k^{(i)})^T(A^{(i)})^{-1}(a_\ell^{(i)})(a_\ell^{(j)})^T(A^{(j)})^{-1}(a_m^{(j)})
\end{align*}

As a result, we essentially need to consider for some $\ell_1, \ell_2, m_1, m_2$, $\E{a_{\ell_1}^TA^{-1}a_{m_1}a_{\ell_2}^TA^{-1}a_{m_2}}$. We note that this is non-zero only in two distinct cases, $\ell_1=m_1$ and $\ell_2=m_2$ as well as $m_1=m_2$, $\ell_1=\ell_2$ and $\ell_1\neq m_1$. Using our techniques, in the first case for a subsample,
\begin{align*}
\E{a_{\ell_1}^TA^{-1}a_{\ell_1}a_{\ell_2}^TA^{-1}a_{\ell_2}}=\Theta(\frac{n^2/N^2}{d^2})    
\end{align*}
For the second case,
\begin{align*}
    \E{a_{\ell_1}^TA^{-1}a_{\ell_2}a_{\ell_2}^TA^{-1}a_{\ell_1}}\leq \Theta(\frac{n}{d^2})
\end{align*}

We note the summands of $(\beta^T\Pi_{X^{(i)}}\Sigma\Pi_{X^{(j)}}\beta)^2$ are as follows:
\begin{align*}
    c_{k_1}c_{k_2}c_{m_1}c_{m_2}\sqrt{\lambda_{k_1}\lambda_{k_2}\lambda_{m_1}\lambda_{m_2}}\lambda_{\ell_1}^{2}\lambda_{\ell_2}^{2}\E{b_{k_1,k_2,m_1,m_2,\ell_1,\ell_2}}
\end{align*}

Where,
\begin{align*}
    b_{k_1,k_2,m_1,m_2,\ell_1,\ell_2}=(a_{k_1}^{(i)})^T(A^{(i)})^{-1}(a_{\ell_1}^{(i)})(a_{\ell_2}^{(i)})^T(A^{(i)})^{-1}(a_{k_2}^{(i)})(a_{m_1}^{(j)})^T(A^{(j)})^{-1}(a_{\ell_1}^{(j)})(a_{\ell_2}^{(j)})^T(A^{(j)})^{-1}(a_{m_2}^{(j)})
\end{align*}

Now due to the independence of the sub-samples, we can see three distinct cases emerge based on the previous two cases. First, when $k_1=\ell_1$, $k_2=\ell_2$, $m_1=\ell_1$, and $m_2=\ell_2$, this yields the sum,
\begin{align*}
    \sum_{k_1,k_2} c_{k_1}^2c_{k_2}^2\lambda_{k_1}^{3}\lambda_{k_2}^3=(\beta^T\Sigma^3\beta)^2
\end{align*}

Second, when $k_1=\ell_1$, $k_2=\ell_2$, $m_1=m_2$, $\ell_1=\ell_2$, and $m_1\neq \ell_1$,
\begin{align*}
    \sum_{k,m\neq k}c_k^2c_m^2\lambda_k^4\lambda_m\leq \beta^T\Sigma^4\beta\beta^T\Sigma\beta
\end{align*}

Finally, when $k_1=k_2\neq \ell_1=\ell_2\neq m_1=m_2$,
\begin{align*}
    \sum_{k\neq m\neq \ell} c_k^2c_m^2\lambda_k\lambda_m\lambda_\ell^4\leq (\beta^T\Sigma\beta)^2tr(\Sigma^4)
\end{align*}

As a result,
\begin{align*}
    \E{(\beta^T\Pi_{X^{(i)}}\Sigma\Pi_{X^{(j)}}\beta)^2}\leq \Theta(\frac{n^4/N^4}{d^4})(\beta^T\Sigma^3\beta)^2+\Theta(\frac{n^3/N^3}{d^4})\beta^T\Sigma^4\beta+\Theta(\frac{r(n)^2}{N^2})
\end{align*}

We can see that, this bound is not sufficiently tight on the $(\beta^T\Sigma^3\beta)^2$ when summed over $N^2$ terms. For terms not including an $r(n)$, we need them to be $o(\frac{1}{N})$ when  because we would like terms in the variance when summed to be $o(\frac{1}{N^2})$.

We begin with the first moment squared (and using the notes to Prop. \ref{prop:proj_expect}):
\begin{align*}
\E{(\theta_{MN}^{(j)})}^T\Sigma \E{\theta_{MN}^{(i)}}\E{(\theta_{MN}^{(i)})}^T\Sigma \E{\theta_{MN}^{(j)}}&\geq \left(\frac{n/N}{d}(1+\Omega(\frac{1}{N})+\frac{1}{N}(nr(n)-3\rho)\right)^4(\beta^T\Sigma^3\beta)^2  
\\&\geq \frac{n^4/N^4}{d^4}(1+\Omega(\frac{1}{N})-3\rho\frac{1}{N})^4(\beta^T\Sigma^3\beta)^2
\\&\geq\frac{n^4/N^4}{d^4}(1+\Omega(\frac{1}{N})-6\frac{1}{N})(\beta^T\Sigma^3\beta)^2
\end{align*}

Now, for the second moment, we note that for $k\neq m$
\begin{align*}
    \E{(a_k^{(i)})^T(A^{(i)})^{-1}(a_k^{(i)})(a_m^{(i)})^T(A^{(i)})^{-1}(a_m^{(i)})}&\leq \E{(a_k^{(i)})^T(A_{km}^{(i)})^{-1}(a_k^{(i)})(a_m^{(i)})^T(A_{km}^{(i)})^{-1}(a_m^{(i)})}\\&\leq \E{tr((A_{km}^{(i)})^{-1})^2}
    \leq \frac{n}{N}\E{tr((A_{km}^{(i)})^{-2})}\\&\leq \frac{n^2}{N^2d^2}(1+O(\frac{1}{N})+\frac{3nr(n)}{N})  
\end{align*}

Taking the difference, we get that,
\begin{align*}
\E{(a_k^{(i)})^T(A^{(i)})^{-1}(a_k^{(i)})(a_m^{(i)})^T(A^{(i)})^{-1}(a_m^{(i)})}-  \E{(\theta_{MN}^{(j)})}^T\Sigma \E{\theta_{MN}^{(i)}}\E{(\theta_{MN}^{(i)})}^T\Sigma \E{\theta_{MN}^{(j)}}\leq \Theta(\frac{n^4}{N^5d^4})   
\end{align*}

For the noise term, we start with:
\begin{align*}
    \E{\left((\eps^{(i)})^T((X^{(i)})(X^{(i)}))^{-1}(X^{(i)})\Sigma\theta_{MN}^{(j)}\right)^2}&\leq \sigma_{max}^2\beta^T\E{\Pi_{X^{(j)}}\Sigma (X^{(i)})^T((X^{(i)})(X^{(i)}))^{-2}(X^{(i)})\Sigma\Pi_{X^{j}}}\beta
    \\&\leq \Theta(\frac{n/N}{d^2})\sigma^2_{max}\beta^T\E{\Pi_{X^{(j)}}\Sigma^3\Pi_{X^{(j)}}}\beta
    \\&\leq \Theta(\frac{n^3/N^3}{d^4})\sigma_{max}^2\beta^T\Sigma^5\beta+\Theta(\frac{nr(n)^2}{N^2})\sigma_{max}^2\beta^T\Sigma\beta
\end{align*}
Where we note by Equ. \ref{equ:tr_sigma_4_normalized} as well as,
\begin{align*}
    \E{X_s^T(X_sX_s^T)^{-2}X_s}\leq \Theta(\frac{n}{d^2})
\end{align*}

Additionally, for $\E{\Pi_{X^{(j)}}\Sigma^3\Pi_{X^{(j)}}}$, we note that $v_i^T\E{\Pi_{X^{(j)}}\Sigma^3\Pi_{X^{(j)}}}v_j=0$ when $i\neq j$. Otherwise,
\begin{align*}
    v_i^T\E{\Pi_{X^{(j)}}\Sigma^3\Pi_{X^{(j)}}}v_i&=\sum_{k}\lambda_i\lambda_k^4\E{a_i^TA^{-1}a_ka_k^TA^{-1}a_i}\\
    &=\sum_i\lambda_i^4\E{(a_i^TA^{-1}a_i)^2}+\sum_{i\neq k}\lambda_i\lambda_k^4\E{a_i^TA^{-1}a_ka_k^TA^{-1}a_i}
    \\&\leq \Theta(\frac{n^2/N^2}{d^2})\lambda_i^5+\Theta(\frac{ntr(\Sigma^4)/N}{d^2})\lambda_i
\end{align*}

Finally, we have to handle the term,
\begin{align*}
&\E{\left((\eps^{(i)})^T((X^{(i)})(X^{(i)}))^{-1}(X^{(i)})\Sigma(X^{(j)})^T((X^{(j)})(X^{(j)}))^{-1}\eps^{(j)}\right)^2}\\&\leq \sigma_{max}^2\E{(\eps^{(i)})^T((X^{(i)})(X^{(i)}))^{-1}(X^{(i)})\Sigma (X^{(j)})^T((X^{(j)})(X^{(j)}))^{-2}(X^{(j)})\Sigma (X^{i)})^T((X^{(i)})(X^{(i)}))^{-1}\eps^{(i)}} \\
&\leq \Theta(\frac{n/N}{d^2})\sigma_{max}^2\E{(\eps^{(i)})^T((X^{(i)})(X^{(i)}))^{-1}(X^{(i)})\Sigma^3(X^{(i)})^T((X^{(i)})(X^{(i)}))^{-1}\eps^{(i)}}
\\&\leq \Theta(\frac{n/N}{d^2})\sigma_{max}^4\E{tr\left(\Sigma^3(X^{i)})^T((X^{(i)})(X^{(i)}))^{-2}(X^{(i)})\right)}\leq \Theta(\frac{n^2/N^2}{d^4}tr(\Sigma^4))\sigma_{max}^4\leq \Theta(\frac{n r(n)^2}{N^2})\sigma_{max}^4
\end{align*}

As a result,
\begin{align*}
    &\Var[\theta_{ds}^T\Sigma\theta_{ds}]\\&\leq \Theta\left(nr(n)^2(\sigma_{max}^2+\sigma_{max}^4)+\frac{\sigma_{max}^2}{Nn}+\frac{1}{N^3}+r(n)^2+\frac{nr(n)\sigma_{max}^2}{\sqrt{Nn}}+\frac{n^2r(n)^2(\sigma_{max}^2+\sigma_{max}^4)}{N}+\frac{n^2r(n)}{Nd^2}\right)
\end{align*}

\end{proof}

\begin{proposition}\label{prop:lower_bound_G_ds}
    We assume Assumption \ref{assumpt:weak_canonical_case}.
    \begin{align*}
        \min_cG(c\theta_{ds})=\Omega((1+\sigma^2)nr(n))
    \end{align*}
\end{proposition}
\begin{proof}

By properties of quadratics,
\begin{align*}
    \min_c G(c\theta_{ds})=1-\frac{\E{\beta^T\Sigma\theta_{ds}}^2}{\E{\theta_{ds}^T\Sigma\theta_{ds}}}
\end{align*}

We note that by Prop. \ref{prop:straight_var_in_ds},
\begin{align*}
   \E{\theta_{ds}^T\Sigma\theta_{ds}}\geq \Theta(\frac{n^2}{d^2})\beta^T\Sigma^3\beta \geq \Theta(1)q^2=\Theta(1) 
\end{align*}

As a result, it is sufficient to show $\E{\theta_{ds}^T\Sigma\theta_{ds}}-\E{\beta^T\Sigma\theta_{ds}}^2 =\Omega((1+\sigma)^2nr(n))$

    We begin by noting that,
    \begin{align*}
        (\E{\theta_{ds}^T\Sigma\beta})^2&=\sum_{i,j}\E{\theta_{MN}^{(i)}}^T\Sigma\beta^T\beta^T\Sigma\E{\theta_{MN}^{(j)}}
        \\&=\sum_{i,j}\E{\theta_{MN}^{(i)}}^T\Sigma\beta^T\beta^T\Sigma\E{\theta_{MN}^{(i)}}
        \\&\leq \sum_{i,j}\E{\theta_{MN}^{(i)}}^T\Sigma\E{\theta_{MN}^{(i)}}\beta^T\Sigma\beta
    \end{align*}

As a result, we only need to consider this difference,
\begin{align*}
    \sum_{i=1}^d\E{(\theta_{MN}^{(i)})^T\Sigma\theta_{MN}^{(i)}}-\E{\theta_{MN}^{(i)}}^T\Sigma\E{\theta_{MN}^{(i)}}
\end{align*}

Now, by the proof of Prop. \ref{prop:straight_var_in_ds} where $\tilde{\sigma}^2$ is as defined in that proposition and we use $X_s$ to represent the ith subsample,
\begin{align*}
\E{(\theta_{MN}^{(i)})^T\Sigma\theta_{MN}^{(i)}} &=\sum_{i\neq j}c_i^2\lambda_j^2\lambda_i\E{a_i^TA^{-1}a_ja_j^TA^{-1}a_i}+\tilde{\sigma}^2\E{tr(\Sigma X_s^T(X_sX_s^T)^{-2}X_s)}\\&+\sum_{i=1}^d\lambda_i^3\E{(a_i^TA^{-1}a_i)^2}   
\\&\geq (1+\tilde{\sigma}^2)nr(n)(1+\frac{1}{N})(1+\Theta(1)\rho)+\sum_{i=1}^dc_i^2\lambda_i^3\E{(a_i^TA^{-1}a_i)^2}
\end{align*}

Where $c_i=\beta^Tv_i$.

Since, 
\begin{align*}
    \E{\theta_{MN}^{(i)}}^T\Sigma\E{\theta_{MN}^{(i)}}=\sum_{i=1}^dc_i^2\lambda_i^3\E{a_i^TA^{-1}a_i}^2
\end{align*}

As a result, we care about lower bounding, $\Var[a_i^TA^{-1}a_i]$.
Now, using Jensen's Inequality and Prop. \ref{prop:tr_A_bounds},
\begin{align*}
    \E{a_i^TA^{-1}a_i}&=\E{a_i^TA_i^{-1}a_i}-\E{\frac{\lambda_i}{1+\lambda_ia_i^TA_i^{-1}a_i}(a_i^TA_i^{-1}a_i)^2}
    \\&\leq \frac{n}{Ntr(\Sigma)-\lambda_i} +\frac{n^2r(n)}{N^2(tr(\Sigma)-\lambda_i)(1-2\rho)^2}-\frac{\lambda_i}{1+2\lambda_i\frac{k}{d}}\E{(a_i^TA_i^{-1}a_i)^2}
    \\&\leq \frac{n}{Ntr(\Sigma)-\lambda_i} +\frac{n^2r(n)}{N^2(tr(\Sigma)-\lambda_i)(1-2\rho)^2}-\frac{\lambda_i}{1+2\lambda_i\frac{k}{d}}\E{tr(A_i^{-1})^2+2tr(A_i^{-2})}
    \\&\leq \frac{n}{Ntr(\Sigma)-\lambda_i} +\frac{n^2r(n)}{N^2tr(\Sigma)(1-2\rho)^2}-\frac{\lambda_i}{1+2\lambda_i\frac{k}{d}}(\E{tr(A_i^{-1})}^2+2\E{tr(A_i^{-2})})
    \\&\leq \frac{n}{Ntr(\Sigma)-\lambda_i} +\frac{n^2r(n)}{N(tr(\Sigma)-\lambda_i)(1-2\rho)^2}\\&-\frac{\lambda_i}{1+2\lambda_i\frac{k}{d}}\left(\frac{n^2}{N^2(tr(\Sigma)-\lambda_i)^2}+\frac{n/N+2(n/N)^2r(n)}{(tr(\Sigma)-\lambda_i)^2}\right)
\end{align*}

Using that $\frac{1}{1-x}\leq 1+x+2x^2$ for small $x$, we have,
\begin{align*}
    \frac{k}{tr(\Sigma)-\lambda_i}&\leq \frac{k}{tr(\Sigma)}+\frac{k\lambda_i}{tr(\Sigma)^2}+\frac{2k\lambda_i^2}{tr(\Sigma)^3}
    \\&\leq \frac{k}{tr(\Sigma)}(1+\frac{\lambda_i}{tr(\Sigma)}(1+O(\frac{1}{n})))
\end{align*}

As a result,
\begin{align*}
    \E{a_i^TA^{-1}a_i}^2&\leq \frac{n^2}{N^2tr(\Sigma)^2}+2\frac{n^3r(n)}{N^3tr(\Sigma)^2}+\frac{n^4r(n)^2}{N^4tr(\Sigma)^2}-2\lambda_i\frac{n^3}{N^3tr(\Sigma)^3}+\frac{n^4}{N^4tr(\Sigma)^4}\lambda_i^2
    \\&+4\frac{n^4r(n)}{N^4tr(\Sigma)^3}\lambda_i+4\frac{n^2}{N^2tr(\Sigma)^4}\lambda_i^2
    \\&=\frac{n^2}{N^2tr(\Sigma)^2}-2\lambda_i\frac{n^3}{N^3tr(\Sigma)^3}+\lambda_i^2\frac{n^4}{N^4tr(\Sigma)^4}+o(\frac{n^2}{Ntr(\Sigma)^2}nr(n))+o(\frac{n}{Ntr(\Sigma)^2})
\end{align*}

On the flip side,
\begin{align*}
    \E{(a_i^TA^{-1}a_i)^2}&=\E{\left(a_i^TA_i^{-1}a_i-\frac{\lambda_i(a_i^TA_i^{-1}a_i)^2}{1+\lambda_ia_i^TA_i^{-1}a_i}\right)^2}
    \\&\geq \E{(a_i^TA_i^{-1}a_i)^2}-\lambda_i\E{(a_i^TA_i^{-1}a_i)^3}+\lambda_i^2\E{(a_i^TA_i^{-1}a_i)^4}
\end{align*}

Now, by Prop. \ref{prop:tr_A_bounds},
\begin{align*}
    \E{(a_i^TA_i^{-1}a_i)^2}&\geq \E{a_i^TA_i^{-1}a_i}^2=\E{tr(A_i^{-1})}^2\geq \frac{n^2}{N^2(tr(\Sigma)-\lambda_i)^2}+o(\frac{n^2}{Ntr(\Sigma)^2}nr(n))
    \\&\geq \frac{n^2}{N^2tr(\Sigma)^2}+o(\frac{n^2}{Ntr(\Sigma)^2}nr(n))\\
    \E{(a_i^TA_i^{-1}a_i)^4}&\geq \E{(a_i^TA_i^{-1}a_i)^2}\geq \frac{n^4}{N^4tr(\Sigma)^4}+o(\frac{n^4}{Ntr(\Sigma)^4}nr(n))\\
    \E{(a_i^TA_i^{-1}a_i)^3}&=\E{tr(A_i^{-1})^3+6tr(A_i^{-1})tr(A_i^{-2})+8tr(A_i^{-3})}
    \\&\leq (\frac{n^2}{N^2}+6\frac{n}{N}+8)\E{tr(A_i^{-3})}
    \\&\leq \frac{n(n^2/N^2+6n/N+8)}{Ntr(\Sigma)^3}(1+O(\frac{1}{n})+O(\frac{nr(n)}{Ntr(\Sigma)^3}))
    \\&\leq \frac{n^3}{N^3tr(\Sigma)^3}+o(\frac{n^2}{N tr(\Sigma)^3})+o(\frac{n^3}{N tr(\Sigma)^3}nr(n))
\end{align*}

Where the last statement used Prop. \ref{prop:second_third_quadratic_forms}

As a result,
\begin{align*}
    \E{(a_i^TA^{-1}a_i)^2}&\geq \frac{n^2}{N^2tr(\Sigma)^2}+o(\frac{n^2}{Ntr(\Sigma)^2}nr(n))+\lambda_i^2\frac{n^4}{N^4tr(\Sigma)^4}+\lambda_i^2o(\frac{n^4}{Ntr(\Sigma)^4}nr(n))\\&-\lambda_i\frac{n^3}{N^3tr(\Sigma)^3}+\lambda_io(\frac{n^2}{Ntr(\Sigma)^3})+o(\frac{n^3}{Ntr(\Sigma)^3}nr(n))
    \\&\geq \frac{n^2}{N^2tr(\Sigma)^2}-\lambda_i\frac{n^3}{N^3tr(\Sigma)^3}+\lambda_i^2\frac{n^4}{N^4tr(\Sigma)^4}+o(\frac{n}{Ntr(\Sigma)^2})+o(\frac{n}{Ntr(\Sigma)^2}nr(n))
\end{align*}

Finally, we can see that,
\begin{align*}
    \Var[a_i^TA^{-1}a_i]\geq o(\frac{n^2}{Ntr(\Sigma)^2}nr(n))+o(\frac{n}{Ntr(\Sigma)^2})
\end{align*}

Thus,
\begin{align*}
\sum_{i=1}^d\E{(\theta_{MN}^{(i)})^T\Sigma\theta_{MN}^{(i)}}-\E{\theta_{MN}^{(i)}}^T\Sigma\E{\theta_{MN}^{(i)}}\geq \Omega((1+\tilde{\sigma}^2)nr(n))+ o(nr(n))+o(\frac{1}{n})
\end{align*}

Where we note that $nr(n)=\Omega(\frac{1}{n})$.

The result follows by using Prop. \ref{prop:noise_nN_v_n}.

\end{proof}

\begin{proposition}
    Suppose Assumption \ref{assumpt:rate_improve} holds. Then, $G(\theta_{ds})=\Theta(1)$ and $$\min_cG(c\theta_{ds})=o(1)$$
\end{proposition}
\begin{proof}
We note that, for an lower bound on the generalization error without inflation we have that,
\begin{align*}
    G(\theta_{ds})&=\beta^T\Sigma\beta-2\E{\theta_{ds}^T\Sigma\beta}+\E{\theta_{ds}^T\Sigma\theta_{ds}}
    \\&\geq \beta^T\Sigma\beta -4\frac{n}{d}\beta^T\Sigma^2\beta+\frac{1+\tilde{\sigma^2}}{2}nr(n)+\frac{1}{4}\frac{n^2}{d^2}\beta^T\Sigma^3\beta
    \\&\geq 1-4q\geq \frac{6}{10}
\end{align*}
For the upper bound,
\begin{align*}
    G(\theta_{ds})&=\beta^T\Sigma\beta-2\E{\theta_{ds}^T\Sigma\beta}+\E{\theta_{ds}^T\Sigma\theta_{ds}}
    \\&\leq \beta^T\Sigma\beta -\frac{n}{2d}\beta^T\Sigma^2\beta+\frac{1+\tilde{\sigma}^2}{2}nr(n)+4\frac{n^2}{d^2}\beta^T\Sigma^3\beta
    \\&\leq 1+4q^2 \leq \frac{14}{10} 
\end{align*}
Where we note that by Prop. \ref{prop:noise_nN_v_n} (by seeing that since $q\leq \rho$ and $q=\Theta(1)$),
\begin{align*}
    \tilde{\sigma}^2nr(n)\leq (1+10\rho)\sigma^2nr(n)
\end{align*}

Now, turning to $\min_cG(c\theta_{ds})$, we note that since $G(c\theta_{ds})\geq 0$, it is sufficient to show that $\min_cG(c\theta_{ds})\leq o(1)$. This can be seen as follows:
\begin{align*}
 \min_c G(c\theta_{ds})&=1-\frac{\E{\theta_{ds}^T\Sigma\beta}^2}{\E{\theta_{ds}^T\Sigma\theta_{ds}}}\\&\leq 1-\frac{(\frac{n}{d}\beta^T\Sigma^2\beta)^2(1+o(1)-\frac{3\rho}{N})^2}{(1+\tilde{\sigma}^2)r(n)n(1+3\frac{n}{N}r(n))(1+o(1))+\frac{n^2}{d^2}\beta^T\Sigma^3\beta(1+o(1))(1+\frac{n}{N}r(n))} 
 \\&\leq 1-\frac{q^2(1+o(1)-\frac{3\rho}{N})^2}{(1+\tilde{\sigma}^2)r(n)n(1+3\frac{n}{N}r(n))(1+o(1))+q^2(1+o(1))(1+\frac{n}{N}r(n))}
\end{align*}
We start off by noting that since $q^2=\Theta(1)$, 
\begin{align*}
    1-\frac{q^2}{q^2(1+o(1))+(1+\tilde{\sigma}^2)nr(n)}&\leq \frac{q^2(1+o(1))+(1+\tilde{\sigma}^2)nr(n)-q^2}{q^2(1+o(1))+(1+\tilde{\sigma}^2)nr(n)}
    \\&\leq \frac{q^2o(1)+(1+\tilde{\sigma}^2)nr(n)}{q^2(1+o(1))+(1+\tilde{\sigma}^2nr(n)}=o(1)
\end{align*}

For the same reason, the remaining terms are $o(1)$ by the same reasoning (and using that $(1+\sigma^2)nr(n)=o(1)$).

\end{proof}
\begin{proposition}\label{prop:c_estimate_mult}
    We assume that Assumption \ref{assumpt:weak_canonical_case} holds, $\sigma_{max}^2=o(\sqrt{n})$. Let $c^*=\arg\min_c G(c\theta_{ds})$. Then, we can construct a specific number of splits ($N$) to create a $\theta_{ds}$,  $\hat{c}^*$, and $\delta_n=o(1)$ such that,
    \begin{align*}
        \mathbb{P}(|\hat{c}^*-c^*|>\delta_n)=o(1)\\
        \mathbb{P}(|R(\hat{c}^*\theta_{ds})-G(\hat{c}\theta_{ds})|>\delta_n)=o(1)\\
    \end{align*}
When the assumptions of Theorem \ref{theorem:canon_multiplicative}, then,
\begin{align*}
    \mathbb{P}(R(\hat{c}^*\theta_{ds})>\alpha G(\theta_{MN}))\to 0
\end{align*}
Where $\alpha$ is as constructed in Theorem \ref{theorem:canon_multiplicative}.
\end{proposition}
\begin{proof}

For simplicity, we would like a subsample that is $o(n)$, so that $\theta_{ds}$ taken on a sample that is $n-n/N$ (and thus none of the expectations are affected up to a $o(1)$ difference). That is, on the larger subsample, we want to use the data splitting procedure discussed in Section \ref{sec:data_splitting}. We need each split size, $N=N(n)$ to be such that $\sigma_{max}^2/N=o(1)$ and $\sigma_{max}^2/(n/N)=o(1)$. This can trivially be satisfied by $N=\Theta(\sqrt{n})$.

Our goal is to estimate the following, 
\begin{align*}
c^*=\frac{\E{\theta_{ds}^T\Sigma\beta}}{\E{\theta_{ds}^T\Sigma\theta_{ds}}}  
\end{align*}

From the Prop. \ref{prop:straight_var_in_ds},
\begin{align*}
c^*=\frac{\frac{n}{d}\beta^T\Sigma^2\beta}{\frac{n^2}{d^2}\beta^T\Sigma^3\beta+(1+\tilde{\sigma}(\lfloor n/N\rfloor)^2)r(n)n}(1+o(1))    
\end{align*}

Here, we let,
\begin{align*}
    \tilde{\sigma}^2(k)=\E{tr(\Lambda(X^*)(X^*(X^*)^T)^{-1}X^*\Sigma(X^*)^T(X^*(X^*)^T)^{-1})}/\E{tr(\Sigma (X^*)^T(X^*(X^*)^T)^{-2}X^*)}
\end{align*}

Where $X^*$ is a sub-sample of $X$ of size $k$.

Consider the samples within $X^{(N)}$. That is, $x_{(N-1)\lfloor\frac{n}{N}\rfloor+1},...,x_n$. For notational convenience, we denote them $z_1,...,z_k$ with corresponding outputs $\eta_1,...,\eta_k$. We note that $k=n/N+o(1)$.

Let $\hat{q}=\frac{1}{k}\sum_{i=1}^{ k}\theta_{ds}^Tz_i\eta_i$. We note that,  $\E{\hat{q}}=q(1+o(1))$ by Prop. \ref{prop:straight_var_in_ds}.

Let $\hat{r}=\frac{1}{k}\sum_{i=1}^k(\theta_{ds}^Tz_i)^2$. Now, by Prop. \ref{prop:straight_var_in_ds},
\begin{align*}
    \E{\hat{r}}=(1+\tilde{\sigma}(\lfloor n/N\rfloor)^2)nr(n)(1+o(1))+\frac{n^2}{d^2}\beta^T\Sigma^3\beta(1+o(1))
\end{align*}

Let $\hat{c}^*=\hat{q}/\hat{r}$.

Our first goal is to use the Continuous Mapping Theorem (Prop. \ref{prop:cts_mapping_theorem}). As such, our goal is to show that the variance is $o(1)$.

We first focus on the conditional variance where we assume everything but $X^{(N)}$ is known (we denote this event $X^s$). From Prop. \ref{prop:var_x_y}, we see the following:

\begin{align*}
\Var\left[\frac{1}{k}\sum_{i=1}^k\theta_{ds}^Tz_i\eta_i\,|\,X^s\right]&=\frac{1}{k}\theta_{ds}^T\Var[z_1\eta_1]\theta_{ds}\leq \frac{1}{k}\theta_{ds}^T\Sigma\theta_{ds}(\sigma_{max}^2+\beta^T\Sigma\beta)+\frac{1}{k}(\theta_{ds}^T\Sigma\beta)^2\\
\Var\left[\frac{1}{k}\sum_{i=1}^k(\theta_{ds}^Tz_i)^2\,|\,X^s\right]&=\frac{1}{k}\Var[\theta_{ds}^Tz_1z_1^T\theta_{ds}]\leq \frac{2}{k}(\theta_{ds}^T\Sigma\theta_{ds})^2
\end{align*}

By Law of Total Variance (Prop. \ref{prop:total_law_probability}),
\begin{align*}
    \Var[\frac{1}{k}\sum_{i=1}^k\theta_{ds}^Tx_iy_i]&=\frac{1}{k}\E{\Var[\theta_{ds}^Tx_1y_1\,|\,X^s]}+\Var\left[\E{\theta_{ds}^Tx_1y_1\,|\, X^s}\right]
    \\&=\frac{1}{k}\E{\theta_{ds}^T\Sigma(\sigma_{max}^2+\beta^T\Sigma\beta)\theta_{ds}+(\theta_{ds}^T\Sigma\beta)^2}+\Var[\theta_{ds}^T\Sigma\beta]
    \\&=o(1)+\Var[\theta_{ds}^T\Sigma\beta]
\end{align*}

Where we use the fact that $(\sigma_{max}^2+\beta^T\Sigma\beta)/k=o(1)$ and from the proof of Prop. \ref{prop:concentration_estimators} and because by Prop. \ref{prop:straight_var_in_ds}, the variances of $\theta_{ds}^T\Sigma\theta_{ds}$ and $\theta_{ds}^T\Sigma\beta$ are $o(1)$ and the first moments are $O(1)$ (hence the second moments are $O(1)$).

Similarly, for the remaining term,
\begin{align*}
    \Var\left[\frac{1}{k}\sum_{i=1}^k(\theta_{ds}^Tx_i)^2\right]&=\frac{1}{k}\E{\Var[(\theta_{ds}^Tx_1)^2\,|\,X^s]}+\Var\left[\E{(\theta_{ds}^Tx_1)^2\,|\, X^s}\right]\\&=\frac{2}{k}\E{(\theta_{ds}^T\Sigma\theta_{ds})^2}+\Var[\theta_{ds}^T\Sigma\theta_{ds}]=o(1)
\end{align*}

As a result, $\hat{q}$ converges to $q$ in probability and $\hat{r}$ converges to $r$. Thus, $\hat{c}^*$ converges in probability to $c^*$ at $o(1)$ by Prop. \ref{prop:cts_mapping_theorem} and Prop. \ref{prop:noise_nN_v_n_o_1}. This is because,
\begin{align*}
    \mathbb{P}(\hat{r}= 0)= \mathbb{P}(\cup_{i=1}^k\theta_{ds}^Tz_iz_i^T\theta_{ds}= 0)\leq \sum_{i=1}^k\mathbb{P}(z_1\neq 0,\theta_{ds}\neq 0)\leq 0
\end{align*}

Because, $z_1\sim N(0,\Sigma)$ which is not zero almost surely and $\Pi_{X}$ is almost surely rank $n$. 

Additionally, by the same reasoning, $|R(\hat{c}^*\theta_{ds})-G(c^*\theta_{ds})|$ converges to zero in probability.

Since there are a finite number of terms that are decaying to zero, by choosing $\delta_n$ with a sufficiently slow decay to zero, you can find a sequence of $\delta_n$ such that the result holds.

For the statement regarding Theorem \ref{theorem:canon_multiplicative}. By construction,
\begin{align*}
    \alpha G(\theta_{MN})\geq G(c_{opt}\theta_{MN})
\end{align*}

Because $\alpha$ is not particularly optimized, it is pretty clear from the proof of Theorem \ref{theorem:canon_multiplicative} that when $q$ satisfies its assumptions (noting that $C_1\geq 1$),
\begin{align*}
    \alpha G(\theta_{MN})- G(c_{opt}\theta_{MN})\geq 4q
\end{align*}

As a result,
\begin{align*}
    \mathbb{P}(R(\hat{c}^*\theta_{ds})\geq \alpha G(\theta_{MN}))\leq \mathbb{P}(|R(\hat{c}^*\theta_{ds})-G(c^*\theta_{ds})|\geq q)\to0
\end{align*}

\end{proof}

\begin{proposition}\label{prop:concentration_estimators}
Consider Assumption \ref{assumpt:weak_canonical_case}. Let $c>0$ be constant assume $n$ sufficiently large. Let $M$ be some large number that is fixed. Suppose one of the following is true:
\begin{enumerate}
    \item $\sigma_{max}^2=o(\sqrt{n})$; or
\item $\sigma^2nr(n)=o(1)$, $\sigma^2=o(n/\log(n)^2)$ and $\sigma_{max}^2/\sigma^2=o(N^{1/2})$;

\end{enumerate}

For any $M>c>0$, $\Var[c\theta_{ds}]=o(1)$. As a result, when $\min_c G(c\theta_{ds})=\omega(1)$, $R(c\theta_{ds})$ concentrates around its mean (as in Def. \ref{def:mean_concentration}.3).
\end{proposition}
\begin{proof}

We just need to show that $\Var[R(c\theta_{ds})]=o(1)$. This is shown when $\Var[\theta_{ds}^T\Sigma\beta]$ and $\Var[\theta_{ds}^T\Sigma\theta_{ds}]$ are $o(1)$.  

By Prop. \ref{prop:straight_var_in_ds}, we can easily see in both of those cases $\Var[\theta_{ds}^T\Sigma\beta]=o(1)$.

For the remaining term, we use Prop. \ref{prop:straight_var_in_ds} and note that in the first case, we can see easily see that it is $o(1)$ by the assumptions. The second case can be seen by noting that every $\sigma_{max}^2$ term is linked with a $r(n)n$ term and is either divided by $N^{1/2}$ or $n^{1/2}$. The conclusion comes from noting that $N\leq n$. 

\end{proof}
\subsection{Miscellaneous}
\begin{proposition}\label{prop:spiked_cov_model_improv}
We use the notation established in Section \ref{sec:def_not}. Consider $\Sigma$ is a spiked covariance model ($\Sigma=I_d+vv^T$ for some $v$). Assume $\eps_i\perp x_i$, $\norm{v}^2\leq \frac{1}{10}\frac{d}{n}$, and $\lim\inf_{n\to \infty}\beta^Tv/\sigma>1+\limsup_{n\to\infty} \frac{n}{d}\norm{v}^2$. Then, the Inflation Property occurs for $n$ sufficiently large.
\end{proposition}
\begin{proof}
For convenience, we want to show that without loss of generality $v=\rho e_1$ for some rotated $\beta$. 

This can be seen by noting that there exists a rotation matrix $Q$ such that $Q\Sigma Q^T=I_d+\rho^2e_1e_1^T$. We can obtain the WLOG conclusion by noting the following:
\begin{align*}
    \Pi_{XQ}&=(XQ)^T(XX^T)^{-1}XQ=Q^T\Pi_XQ\\
    (XQ)^\dagger\eps&=Q^TX^T(XX^T)^{-1}\eps\\
    \hat{\theta}_{MN}&:=Q^T\theta_{MN}=Q^T\Pi_X\beta+Q^TX^\dagger\eps\stackrel{d}{=}\Pi_{XQ}\hat{\beta}+(XQ)^\dagger Q\eps\\
    \E{(c\theta_{MN}-\beta)^Txx^T(c\theta_{MN}-\beta)}&=\E{(c\theta_{MN}-\beta)^TQ(Q^Tx)(Q^Tx)^TQ^T(c\theta_{MN}-\beta)}
    \\&=\E{(cQ^T\theta_{MN}-Q^T\beta)^T(Q^Tx)(Q^Tx)^T(cQ^T\theta_{MN}-Q^T\beta)}
    \\&=\E{(c\hat{\theta}_{MN}-\hat{\beta})^T(Q^Tx)(Q^Tx)^T(c\hat{\theta}_{MN}-\hat{\beta})}
\end{align*}
Where $c>0$ is an arbitrary constant and $\hat{\beta}=Q\beta$. The conclusion comes from letting $\rho=\norm{v}$

We first calculate, $r(n)n=o(1)$ as follows:
\begin{align*}
    r(n)n=\frac{tr(I_d+2\rho e_1e_1^T+\rho^2e_1e_1^T )}{tr(I_d+\rho e_1e_1^T)^2}=n\frac{d+\rho+\rho^2}{(d+\rho)^2}\leq \frac{1}{100n}(1+o(1))
\end{align*}

Let $\gamma =\beta^Te_1$ and noting that $tr(\Sigma)=d+\rho=d(1+\rho/d)=d(1+o(1))$. Therefore, by Props. \ref{prop:proj_expect}, \ref{prop:exp_sig_exp}, and \ref{prop:exp_noise_first_order},

\begin{align*}
    \E{\beta^T\Sigma\theta_{MN}}&=\E{\beta^T\theta_{MN}}+\rho\gamma\E{e_1^T\Pi_X\beta}=\frac{n}{d}\beta^T\Sigma\beta(1+o(1))+\rho\gamma\frac{n}{d}e_1^T\Sigma\beta(1+o(1))
    \\&=\frac{n}{d}\beta^T\Sigma\beta(1+o(1))+\rho\gamma^2\frac{n}{d}(1+o(1))
    \\
    \E{\theta_{MN}^T\Sigma\theta_{MN}}&= \E{\theta_{MN}^T\theta_{MN}}+\rho\E{\theta_{MN}^Te_1e_1^T\theta_{MN}}\\&=\beta^T\E{\Pi_X}\beta+\sigma^2\E{tr((XX^T)^{-1})}+\rho\E{\theta_{MN}^Te_1e_1^T\theta_{MN}}
    \\&\leq \frac{n}{d}(\beta^T\Sigma\beta+\sigma^2) (1+o(1))+\rho\E{\theta_{MN}^Te_1e_1^T\theta_{MN}}
\end{align*}

Now let $c_i=e_i^T\beta$, $\lambda_i$ be the ith order eigenvalue of $\Sigma$, and $A=\sum_{i=1}^d\lambda_ia_ia_i^T$ where i.i.d. $a_i\sim N(0,1)$. We then calculate the remaining object,
\begin{align*}
    \E{\theta_{MN}^Te_1e_1^T\theta_{MN}}&=\E{\eps^T(XX^T)^{-1}Xe_1e_1X^T(XX^T)^{-1}\eps}+\E{\beta^TX^T(XX^T)^{-1}Xe_1e_1^TX^T(XX^T)^{-1}X\beta}
    \\&=\E{e_1^TX^T(XX^T)^{-2}Xe_1}\sigma^2+\sum_{i,j}c_ic_j\lambda_1\sqrt{\lambda_i\lambda_j}\E{a_i^TA^{-1}a_1a_1^TA^{-1}a_j}
    \\&=\E{tr(X^T(XX^T)^{-2}X)}\sigma^2+\sum_{i}c_i^2\lambda_1\lambda_i\E{a_i^TA^{-1}a_1a_1^TA^{-1}a_i}
\end{align*}

Where we note the following:
\begin{align*}
\E{e_1^TX^T(XX^T)^{-2}Xe_1}&=\lambda_1\E{a_1^TA^{-2}a_1}\leq (1+\rho)\frac{n}{d^2}
\\\sum_{i}c_i^2\lambda_1\lambda_i\E{a_i^TA^{-1}a_1a_1^TA^{-1}a_i}&=\lambda_1^2c_1^2\E{(a_1^TA^{-1}a_1)^2}+\lambda_1\sum_{i\neq 1}c_i^2\lambda_i\E{a_i^TA_1^{-2}a_i}
\\&=(1+\rho)^2\gamma^2\frac{n^2}{d^2}(1+o(1))+(1+\rho)\beta^T\Sigma\beta\frac{n}{d^2}(1+o(1))
\\&=(1+\rho)^2\gamma^2\frac{n^2}{d^2}(1+o(1))+(1+\rho)\beta^T\Sigma\beta\frac{n}{d^2}(1+o(1))
\end{align*}

As a result,
\begin{align*}
    c_{opt}&\geq \frac{\frac{n}{d}\beta^T\Sigma\beta+\rho\gamma^2\frac{n}{d}}{\frac{n}{d}(\beta^T\Sigma\beta+\sigma^2)+(1+\rho)\frac{n}{d^2}+(1+\rho)^2\gamma^2\frac{n^2}{d^2}}(1+o(1))
\end{align*}

Because $c_{opt}>1$ if and only if the numerator is larger than the denominator, we can analyze:
\begin{align*}
 (*)=\frac{n}{d}\beta^T\Sigma\beta+\rho\gamma^2\frac{n}{d}-&\frac{n}{d}(\beta^T\Sigma\beta+\sigma^2)-(1+\rho)\frac{n}{d^2}-(1+\rho)^2\gamma^2\frac{n^2}{d^2}
\\&\geq \frac{n}{d}\left((\rho-(1+\rho)^2\frac{n}{d})\gamma^2-\sigma^2\right)(1+o(1))
\end{align*}

We note that $\beta^Tv=\sqrt{\rho}\gamma$. And the result comes from our assumptions. When $n$ is sufficiently large and noting that the $n/d$ term is irrelevant because it can be factored out of the numerator and denominator of $c_{opt}$.
\end{proof}
\begin{proposition}\label{prop:lambda_reg_isotropic_high_dim}
    We use the notation established in Section \ref{sec:def_not}. Consider $\Sigma=I_d$. Let $\theta_\lambda:=X^T(XX^T+\lambda I_n)^{-1}Y$ and $\sigma^2>0$. Let $\lambda_{opt}:=\arg \min\{ R(\theta_\lambda): \lambda>-\lambda_{min}(XX^T)\}$. Then, $\lambda_{opt}> 0$ a.s. 
\end{proposition}
\begin{proof}
Let $\theta_\lambda:=X^T(XX^T+\lambda I_n)^{-1}Y$. This is the ridge regularized estimator of $\beta$ with penalty $\lambda$. 

We ultimately want to show that the derivative of the risk with respect to $\lambda$ around zero is strictly less than 0 and the risk is convex with respect to $\lambda$ when $\lambda<0$ (i.e. the minimum of the risk occurs when $\lambda>0$).
\begin{align*}
R(\theta_\lambda)&=\beta^T\left(X^T(XX^T+\lambda I_n)^{-1}X\right)^2\beta+\sigma^2tr\left(X^T(XX^T+\lambda I_n)^{-2}X\right)\\&-2\beta^TX^T(XX^T+\lambda I_n)^{-1}X\beta+\norm{\beta}^2    
\end{align*}
We note the following technical derivative results which uses $d(A^{-1})=-A^{-1}d(A)A^{-1}$ and $\frac{d((XX^T+\lambda I))}{d\lambda}=I$ which implies that $d((XX^T+\lambda I)^{-k})=-k(XX^T+\lambda I)^{-k-1}$:
\begin{align*}
    \frac{d}{d\lambda}X^T(XX^T+\lambda I_n)^{-1}X\biggr|_{\lambda=0}&=-X^T(XX^T)^{-2}X\\
    \frac{d}{d\lambda}X^T(XX^T+\lambda I_n)^{-2}X\biggr|_{\lambda=0}&=-2X^T(XX^T)^{-3}X\\
     \frac{d}{d\lambda}\left(X^T(XX^T+\lambda I_n)^{-1}X\right)^2\biggr|_{\lambda=0}&=-2X^T(XX^T)^{-2}X
\end{align*}

As a result,
\begin{align*}
\frac{d}{d\lambda}R(\theta_\lambda)\biggr|_{\lambda=0}=-2\beta^TX^T(XX^T)^{-2}X\beta-2\sigma^2tr\left((XX^T)^{-1}\right)+2\beta^TX^T(XX^T)^{-2}X\beta<0
\end{align*}

Turning to the second derivative, when we assume $\lambda>-\lambda_{min}(XX^T)$,
\begin{align*}
    \frac{d}{d\lambda}R(\theta_\lambda)&=2\beta^TX^T(XX^T+\lambda I_n)^{-1}X\beta+6\sigma^2\tr\left(X^T(XX^T+\lambda I_n)^{-4}X\right)-2\beta^TX^T(XX^T+\lambda I_n)^{-3}X\beta
    \\&+4\beta^TX^T(XX^T+\lambda I_n)^{-1}XX^T(XX^T+\lambda I_n)^{-3}X\beta
    \\&=2\beta^TX^T(XX^T+\lambda I_n)^{-1}X\beta+6\sigma^2\tr\left(X^T(XX^T+\lambda I_n)^{-4}X\right)-2\beta^TX^T(XX^T+\lambda I_n)^{-3}X\beta
    \\&+4\beta^TX^T(XX^T+\lambda I_n)^{-1}(XX^T+\lambda I_n)(XX^T+\lambda I_n)^{-3}X\beta-4\lambda\beta^TX^T(XX^T+\lambda I_n)^{-4}X\beta
    \\&=2\beta^TX^T(XX^T+\lambda I_n)^{-1}X\beta+6\sigma^2\tr\left(X^T(XX^T+\lambda I_n)^{-4}X\right)
    \\&+2\beta^TX^T(XX^T+\lambda I_n)^{-3}X\beta-4\lambda\beta^TX^T(XX^T+\lambda I_n)^{-4}X\beta
\end{align*}

Where we note that,
\begin{align*}
    d^2(B)=Bd^2(B)+2d(B)^2+d^2(B)B
\end{align*}

As a result, $\frac{d}{d\lambda}R(\theta_\lambda)>0$ when $\lambda\leq 0$. Therefore, $\lambda_{opt}>0$.
\end{proof}
\begin{proposition}\label{prop:make_unbiased}
Consider the assumptions of Prop. \ref{prop:proj_expect}. Suppose further that either $\sigma^2$ or $\norm{\beta}^2$ is $\omega(\frac{n}{tr(\Sigma)})$. Then, 
\begin{align*}
    \lim _{n\to\infty}G(\E{\Pi_X}^{-1}\theta_{MN})=\infty
\end{align*}
\end{proposition}
\begin{proof}
By Prop. \ref{prop:proj_expect}, $\E{\Pi_X}=\Theta(\frac{n}{tr(\Sigma)})\Sigma$. 

As a result,
\begin{align*}
    G(\E{\Pi_X}^{-1}\theta_{MN})=\Theta(1)\frac{tr(\Sigma)^2}{n^2}\E{\theta_{MN}^T\Sigma^{-1}\theta_{MN}}-\Theta(1)\frac{tr(\Sigma)}{n}\E{\theta_{MN}^T\beta}+\beta^T\Sigma\beta
\end{align*}

As a threshold matter, $\frac{tr(\Sigma)}{n}\E{\theta_{MN}^T\beta}=\Theta(1)\beta^T\Sigma\beta$. So we can focus $\E{\theta_{MN}^T\Sigma^{-1}\theta_{MN}}$.

From the proof of Prop. \ref{prop:exp_noise_first_order}, 
\begin{align*}
    \frac{tr(\Sigma)^2}{n^2}\E{\eps^T(XX^T)^{-1}X\Sigma^{-1}X^T(XX^T)^{-1}\eps}\geq \Theta(1)\sigma^2\frac{tr(\Sigma)}{tr(\Sigma)^2}n\frac{tr(\Sigma)^2}{n^2}
    \geq \Theta(1)\sigma^2\frac{tr(\Sigma)}{n}
\end{align*}

From the proof of Prop. \ref{prop:exp_sig_exp}, we can see a similar result, namely,
\begin{align*}
    \frac{tr(\Sigma)^2}{n^2}\E{\beta^T\Pi_X\Sigma^{-1}\Pi_X\beta}\geq \Theta(1)\beta^T\Sigma^2\beta+\Theta(1)\frac{tr(\Sigma)}{n}\norm{\beta}^2
\end{align*}
\end{proof}
\begin{proposition}\label{prop:optimal_direction_shrink}
Consider the assumptions of Theorem \ref{thm:add_improve}. Suppose $\sigma_{max}=0$ Then, suppose $v\neq 0$
\begin{align}\label{equ:opt_v_equ}
    \lim_{n\to\infty}\sup _{(\beta,\Sigma)\in\mathcal{F}}G((1-c)\theta_{MN}+cv)\to\infty\;\;\;\forall c\in\R-\{0\} 
\end{align}
Where $(\beta,\Sigma)\in\mathcal{F}$ means that $\beta$ and $\Sigma$ satisfy the assumptions of Theorem \ref{thm:add_improve}.
\end{proposition}
\begin{proof}

We begin by noting the expansion,
\begin{align*}
   G\left((1-c)\theta_{MN}+cv\right)=G((1-c)\theta_{MN})-2c(1-c)\E{(\beta-\theta_{MN})^T\Sigma v}+c^2v^T\Sigma v 
\end{align*}

Consider some $v\neq 0$ and $c\in \mathbb{R}$ where $c\neq 0$. WLOG consider $\norm{v}=1$.

We first consider the following setup for some $0<q\ll 1$:
\begin{align*}
    \Sigma&=q\frac{d}{n}vv^T+\eps (I_d-vv^T)\\
    \beta&=\frac{1}{\sqrt{q}}\frac{\sqrt{n}}{\sqrt{d}}v
\end{align*}
Where $\eps=1-\frac{1}{d}-\frac{q}{n}$.

We can further see that,
\begin{align*}
    \frac{n}{d}\beta^T\Sigma^2\beta-nr(n)&=\frac{n^2}{qd^2}v^T(q^2\frac{d^2}{n^2}vv^T+\eps^2(I_d-vv^T))v-o(1)\\
    &=q-o(1)>0
\end{align*}

Where $nr(n)=n\frac{tr(q^2\frac{d}{n}vv^T)+\eps^2(d-1)}{d^2}=o(1)$.

We note that for this specification of the problem, $\beta^T\Sigma\beta=1$, $tr(\Sigma)=d$ and the remaining requirements of Theorem \ref{thm:add_improve} hold given the diagonal nature of $\Sigma$ (under a basis which includes $v$).

Consider the following parameterized by $q\ll 1$. We also note that by the assumptions of Theorem \ref{thm:add_improve}, $-\beta$ is also admissible. As such, WLOG,
\begin{align*}
\sqrt{q}\frac{\sqrt{d}}{\sqrt{n}}\E{(\beta-\theta_{MN})^T\Sigma\beta}=\E{(\beta-\theta_{MN})^T\Sigma v}    
\end{align*}
Then, for either $\beta$ or $-\beta$, the following holds:
\begin{align*}
G\left((1-c)\theta_{MN}+cv\right)&= G((1-c)\theta_{MN})+2c(1-c)\E{(\beta-\theta_{MN})^T\Sigma v}+c^2\E{v^T\Sigma v}
\\&=G((1-c)\theta_{MN})+2c(1-c)\frac{1}{\sqrt{q}}\frac{\sqrt{d}}{\sqrt{n}}\E{(\beta-\theta_{MN})^T\Sigma\beta}+c^2\frac{1}{q}\frac{d}{n}\beta^T\Sigma\beta\to\infty
\end{align*}

\end{proof}   
\begin{proposition}\label{prop:snr_ex_prop}
Let $\Sigma$ be in the class of covariance matrices such that $(1+\sigma^2)nr(n)=o(1)$, $\lambda_d(\Sigma)=\Omega(1)$, and $\sigma_{max}^2/\sigma^2=O(n)$. Then, when the assumptions of Theorem \ref{theorem:canon_multiplicative} are satisfied, we obtain its conclusions, and,
\begin{align*}
    \beta^T\Sigma\beta/\E{\eps^2}=o(1)
\end{align*}
occurs if and only if  $\beta^T\Sigma\beta/\sigma^2\to\infty$.

\end{proposition}
\begin{proof}

Because $\beta^T\Sigma\beta=1$, we only need to consider when $\sigma^2\to\infty$. 

Using the notation and results of Prop. \ref{prop:noise_nN_v_n_gen} we can see that,
    \begin{align*}
        \frac{1}{2r(n)}\E{\frac{x_1^T\Sigma x_1}{\norm{x_1}^4}\E{\eps_1^2|x_1}}\leq \sigma^2\leq \frac{2}{r(n)}\E{\frac{x_1^T\Sigma x_1}{\norm{x_1}^4}\E{\eps_1^2|x_1}}
    \end{align*}

Our overall goal is to show that there are some constants such that, $c\E{\eps_1^2}\leq \sigma^2\leq C\E{\eps_1^2}$. This would prove the result because then, if $\sigma^2\to\infty$, then $\E{\eps_1^2}\to\infty$. By the same reasons the other direction would occur.

For the lower bound, we compute as follows.

To split the $\E{\eps_1^2|x_1}$ and $\frac{x_1^T\Sigma x_1}{\norm{x_1}^4}$, we use the definition of covariance and Cauchy-Schwartz.
\begin{align*}
\E{\frac{x_1^T\Sigma x_1}{\norm{x_1}^4}\E{\eps_1^2|x_1}}&=\E{\frac{x_1^T\Sigma x_1}{\norm{x_1}^4}}\E{\eps_1^2}+\Cov[\frac{x_1^T\Sigma x_1}{\norm{x_1}^4},\E{\eps_1^2|x_1}]   \\
&\geq \E{\frac{x_1^T\Sigma x_1}{\norm{x_1}^4}}\E{\eps_1^2}-\Var\left[\frac{x_1^T\Sigma x_1}{\norm{x_1}^4}\right]^{1/2}\Var\left[\E{\eps^2_1|x_1}\right]^{1/2}
\end{align*}

We note that,
\begin{align*}
    \Var\left[\E{\eps^2_1|x_1}\right]^{1/2}\leq \E{\E{\eps_1^2|x_1}^2}^{1/2}\leq \sigma_{max}\E{\eps_1^2}^{1/2}\leq \frac{\sigma_{max}}{\E{\eps_1^2}^{1/2}}\E{\eps_1^2}
\end{align*}

For $n$ sufficiently large because we only care about the  $\E{\eps_1^2}\to\infty$ case.

Our next goal is to use the concentration of $\frac{x_1^T\Sigma x_1}{\norm{x_1}^4}$ to basically make  the covariance term small.

We can use the same trick to separate, $x_1^T\Sigma x_1$ and $\frac{1}{\norm{x_1}^4}$.
\begin{align}\label{equ:cauch_lower_bound_noise_ex}
    \E{\frac{x_1^T\Sigma x_1}{\norm{x_1}^4}}&\geq\E{x_1^T\Sigma x_1}\E{\frac{1}{\norm{x_1}^4}}-\Var[x_1^T\Sigma x_1]^{1/2}\Var\left[\frac{1}{\norm{x_1}^4}\right]^{1/2}
\end{align}

We first note that,
\begin{align*}
  \E{x_1^T\Sigma x_1}&=tr(\Sigma^2) \\
  \E{\frac{1}{\norm{x_1}^4}}&\geq \frac{1}{\E{\norm{x_1}^2}^2}=\frac{1}{d^2}
\end{align*}

For the variance terms, let $z\sim N(0,I_d)$
\begin{align*}
    \Var[x_1^T\Sigma x_1]&= \E{(x_1^T\Sigma x_1)^2}-tr(\Sigma^2)^2=\E{(z^T\Sigma^2z)^2}-tr(\Sigma^2)^2
    \\&=tr(\Sigma^2)^2+2tr(\Sigma^4)-tr(\Sigma^2)^2=2tr(\Sigma^4)\\
    \Var\left[\frac{1}{\norm{x_1}^4}\right]&\leq \E{\frac{1}{(z^T\Sigma z)^4}}-\frac{1}{d^4}\leq \frac{1}{d^4(1-1/h(n))^4}(1+o(1))-\frac{1}{d^4}
\end{align*}

Where we use Prop. \ref{prop:second_third_quadratic_forms} for the $\Var[x_1^T\Sigma x_1]$ term and for the $\frac{1}{\norm{x_1}^4}$ term we let $h(n)=\frac{n^{1/2}}{\log(n)}$ and use Prop. \ref{prop:general_chi_square_moment_taught} (which applies because both $tr(\Sigma)^2/tr(\Sigma^2)$ and $tr(\Sigma)/\lambda_1$ are at least order $n$).

Using the fact that for small $\gamma$, $\frac{1}{(1-\gamma)^4}\leq 8\gamma$,
\begin{align*}
    \Var\left[\frac{1}{\norm{x_1}^4}\right]\leq \frac{8}{d^4h(n)}
\end{align*}

As a result,
\begin{align*}
    \E{\frac{x_1^T\Sigma x_1}{\norm{x_1}^4}}\geq r(n)-\frac{16}{h(n)^{1/2}}\left(\frac{tr(\Sigma^4)}{d^4}\right)^{1/2}=r(n)(1+o(1))
\end{align*}

Where we note that, by the proof of Prop. \ref{prop:var_noise_pseudo_inverse}, $tr(\Sigma^4)/d^4\leq \rho^2\frac{1}{n}r(n)^2$.

Because of Cauchy-Schwartz, we obtain the following bound on the remaining variance term,
\begin{align*}
    \Var\left[\frac{x_1^T\Sigma x_1}{\norm{x_1}^4}\right]\leq \Var[x_1^T\Sigma x_1]\Var\left[\frac{1}{\norm{x_1}^4}\right]\leq \frac{16}{h(n)^{1/2}}\left(\frac{tr(\Sigma^4)}{d^4}\right)\leq \frac{16}{nh(n)}r(n)^2
\end{align*}

As a result, since $\sigma_{max}^2/\E{\eps_1^2}=O(n)$,
\begin{align*}
   \Var\left[\frac{x_1^T\Sigma x_1}{\norm{x_1}^4}\right]^{1/2}\Var[\E{\eps_1^2|x_1}]^{1/2}=o(r(n)) 
\end{align*}

As a result,
\begin{align*}
    \sigma^2\geq \frac{1}{2}\E{\eps_1^2}
\end{align*}

For the upper bound, we can use a similar technique.
\begin{align*}
\E{\frac{x_1^T\Sigma x_1}{\norm{x_1}^4}\E{\eps_1^2|x_1}}&=\E{\frac{x_1^T\Sigma x_1}{\norm{x_1}^4}}\E{\eps_1^2}+\Cov[\frac{x_1^T\Sigma x_1}{\norm{x_1}^4},\E{\eps_1^2|x_1}]   \\
&\leq \E{\frac{x_1^T\Sigma x_1}{\norm{x_1}^4}}\E{\eps_1^2}+\Var\left[\frac{x_1^T\Sigma x_1}{\norm{x_1}^4}\right]^{1/2}\Var\left[\E{\eps^2_1|x_1}\right]^{1/2}
\end{align*}

By Prop. \ref{prop:general_chi_square_moment},
\begin{align*}
\E{\frac{1}{\norm{x_1}^4}}\leq 4\frac{1}{d^4}(1+o(1))
\end{align*}

Thus, using the technique in Equ. \ref{equ:cauch_lower_bound_noise_ex}, we have that,
\begin{align*}
    \E{\frac{x_1^T\Sigma x_1}{\norm{x_1}^4}}\leq 16r(n)
\end{align*}

As a result,
\begin{align*}
    \sigma^2\leq 64\E{\eps_1^2}
\end{align*}

\end{proof}
\begin{proposition}\label{prop:examples_shown}
    All the examples in Example \ref{example:eigen_decay} satisfy Assumption \ref{assumpt:weak_canonical_case} (sans the assumptions relating $\beta$ and $\P_\eps$).
\end{proposition}
\begin{proof}
For the first example, we can see that $tr(\Sigma)=qd+\frac{d}{d-n}(d-n)(1-q)=d$. The second part of Assumption 2.9 is satisfied because we are considering $q\leq \frac{1}{8}$ small and for $n $ sufficiently large $\epsilon\geq \frac{1}{2}$.  

We further note that $K(n)=\{1,...,n\}$ and so $\text{card}(K_n)=n$.

Finally, we note that $nr(n)=n\frac{\frac{d^2}{n^2}nq^2+\epsilon^2(d-n)}{d^2}=q^2(1+o(1))$. As a result, the fifth assumption is satisfied when $\alpha_{min}>1$.

For the second example, we note that $\alpha\in (0,1)$ by our assumption and that $\int_{1}^dx^{-\alpha}dx=\frac{d^{1-\alpha}-1}{1-\alpha}>8$. As a result, taking $\lambda_i=\frac{1}{\gamma}\frac{d}{n}$, we have that $tr(\Sigma)=d$. We also note that $\lambda_d=d^{-\alpha}\geq d^{-1}$.

We note that $\alpha_{min}\gamma\frac{d}{n}\leq \gamma\frac{d}{n}i^{-\alpha}$ implies that $i\leq \alpha_{min}^{-1/\alpha}$. As a result, $\text{card}(K_n)$ is a fixed finite quantity.

Finally, suppose $\alpha> \frac{1}{2}$,
\begin{align*}
   nr(n)&=\gamma^2n\frac{\frac{d^2}{n^2}\sum_{i=1}^{d}i^{-2\alpha}}{d^2}=\gamma^2\frac{\sum_{i=1}^di^{-2\alpha}}{n}=\Theta(\gamma^2\frac{d^{1-2\alpha}}{n(2\alpha-1)})
   \\&=\Theta\left(\frac{d^{1-\alpha}d^{-\alpha}}{n}\right)=\Theta(d^{-\alpha})=o(1)
\end{align*}

Now, when $\alpha=\frac{1}{2}$,
\begin{align*}
      nr(n)&=\gamma^2n\frac{\frac{d^2}{n^2}\sum_{i=1}^{d}i^{-2\alpha}}{d^2}=\gamma^2\frac{\sum_{i=1}^di^{-1}}{n}=O\left(\frac{\log(d)}{n}\right)=O\left(\frac{d^{1-\alpha}}{n}\frac{\log(d)}{d^{1-\alpha}}\right)=o(1)
\end{align*}

Now, suppose $\alpha<\frac{1}{2}$,
\begin{align*}
      nr(n)&=\gamma^2n\frac{\frac{d^2}{n^2}\sum_{i=1}^{d}i^{-2\alpha}}{d^2}=\gamma^2\frac{\sum_{i=1}^di^{-2\alpha}}{n}=\Theta\left(\gamma^2\frac{d^{1-2\alpha}}{n(1-2\alpha)}\right)
   \\&=\Theta\left(\frac{d^{1-\alpha}d^{-\alpha}}{n}\right)=\Theta(d^{-\alpha})=o(1)
\end{align*}

As a result, $nr(n)=o(1)$ and so we do not need to check the remaining assumption.

For our last example, we note that sum of the eigenvalues in the first block are is between $qd$ and $2qd$ (without any scaling). The sum of second block is of the order $\frac{d^{1-\alpha_2+\alpha_1}}{n^{\alpha_1}}$ which by assumption is order $d$.

We can see that when $q<\frac{1}{8}$, the second assumption holds because $\lambda_d$ does not decay faster than $1/d$. The third assumption holds because we have that $\lambda_{n+1}=o(d/n)$ and $\lambda_1=\Theta(d/n)$.

Finally, $nr(n)\leq 2q^2+O\left(\frac{d^{2\alpha_1}}{n^{2\alpha_1}d^2}\sum_{i=1}^{d-n}i^{-2\alpha_2}\right)$.

Now, if $\alpha_2\neq \frac{1}{2}$,
\begin{align*}
    \frac{d^{2\alpha_1}}{n^{2\alpha_1}d^2}\sum_{i=1}^{d-n}i^{-2\alpha_2}=O\left(\frac{d^{2\alpha_1}d^{1-2\alpha_2}}{n^{2\alpha_1}d^2}\right)=O\left(\left(\frac{d^{1+\alpha_1-\alpha_2}}{n^{2\alpha_1}}\right)^2d^{-3}\right)=o(1)
\end{align*}

When $\alpha_2=\frac{1}{2}$, we have that,
\begin{align*}
    \frac{d^{2\alpha_1}}{n^{2\alpha_1}d^2}\sum_{i=1}^{d-n}i^{-2\alpha_2}=O\left(\frac{d^{2\alpha_1}}{n^{2\alpha_1}d^2}\log(d)\right)=O\left(\frac{\log(d)}{n^{2\alpha_1}d^{2-2\alpha_1}}\right)=o(1)
\end{align*}
Where we use that $\log(d)/d^{\ell}=o(1)$ for all $\ell>0$.

Hence, $nr(n)\leq 2q^2(1+o(1))$ and the result holds.

\end{proof}
\bibliographystyle{plainnat}
\bibliography{bib}
\end{document}